\documentclass[11pt, reqno]{amsart}
\usepackage{multirow} 
\usepackage{tikz}
\usepackage{caption}
\usetikzlibrary{
 decorations.pathmorphing,%
 decorations.markings,%
 decorations.pathreplacing,%
 decorations.text,%
 calc,%
 patterns,%
 shapes.geometric,%
 arrows,
 decorations.shapes,
 positioning,
 plotmarks,
 shadings,
 math,
 intersections
}

\definecolor{jeffColor}{RGB}{102, 0, 204}
\definecolor{yaizaColor}{RGB}{0, 153, 153}

\definecolor{periodColor}{RGB}{255, 167, 105}
\definecolor{dark-green}{RGB}{135, 194, 130}
\tikzset{>=latex} % sets default arrow type in tikz
\tikzset{font=\small}
\tikzset{mark size=1.5pt, mark options=thin}
\tikzset{pin distance=4pt,
 every pin edge/.style={<-, thin, shorten <= -2pt}}
\usepackage{array,booktabs,tabularx}
\usepackage{graphicx}
\usepackage{amssymb}
\usepackage{enumerate}
\usepackage{color}
\usepackage{esint}
\usepackage{mathtools}
\usepackage{dsfont}
\usepackage{ esint }
\usepackage{placeins}
\usepackage[final]{ed}
\usepackage[final]{showkeys}

\newcommand{\smalleq}[1]{\scalebox{.9}{$#1$}}%
\usepackage{comment}
\usepackage{enumitem}
\usepackage{soul}
\usepackage{needspace}
\usepackage[normalem]{ulem}

\renewcommand{\L}{\mathcal{L}}
\usepackage{multicol}
\usepackage{tabto}
\usepackage{hyperref}

\newcommand{\Ell}{\operatorname{Ell}}
\newtheorem{lemma}{Lemma}
\newtheorem{theorem}{Theorem}

\newtheorem{proposition}[lemma]{Proposition}

\theoremstyle{definition}
\newtheorem{definition}[lemma]{Definition}
\newtheorem{remark}[lemma]{Remark}

\usepackage{dsfont}
\newcommand{\1}{\mathds{1}}

\newcommand{\J}{\mc{J}}
\newcommand{\B}{\mc{B}}

\newcommand{\R}{{\mathbb R}}
\newcommand{\Z}{{\mathbb Z}}

\newcommand{\G}{{\mathcal G}}
\newcommand{\T}{\mathcal{T}}
\newcommand{\sub}[1]{_{_{#1}}}

\newcommand{\ep}{\varepsilon}

\newcommand{\cs}{$\clubsuit$}

\newcommand{\x}{\ensuremath{\times}}

\newcommand{\Id}{\operatorname{Id}}
\newcommand{\mc}[1]{\mathcal{#1}}
\newcommand{\re}{\mathbb{R}}

\newcommand{\WFh}{\operatorname{WF_h}}
\newcommand{\MSh}{\operatorname{MS_h}}

\newcommand{\red}[1]{{\color{purple}{#1}}}

\newcommand{\energy}{s}

\newcommand{\TM}{T^*\!M}
\newcommand{\FR}{\mathfrak{I}_{_{\!0}}}
\newcommand{\SNH}{S\!N^*\!H}

\newcommand{\s}{SN}

\newcommand{\SM}{S^*\!M}

\newcommand{\SigH}{\Sigma_{_{\!H\!,p}}}
\newcommand{\comp}{\operatorname{comp}}

\newcommand{\LM}{{_{\!L^2(M)}}}

\newcommand{\smooth}{H_{\textup{scl}}^{-N}\to H_{\textup{scl}}^N}

\newcommand{\resfun}{resolution function}
\newcommand{\resfuns}{resolution functions}

\renewcommand{\a}{{\bf a}}

\def\XXint#1#2#3{{\setbox0=\hbox{$#1{#2#3}{\int}$} \vcenter{\hbox{$#2#3$}}\kern-.5\wd0}}

\DeclareMathOperator{\vol}{vol}

\DeclareMathOperator{\supp}{supp}
\DeclareMathOperator{\inj}{inj}

\newcommand{\e}{\varepsilon}

\newcommand{\Oph}{Op_h}
\numberwithin{equation}{section}
\numberwithin{lemma}{section}

\newcommand{\Pt}{{\bf{P}}_t}
\newcommand{\pt}{{\bf{p}}_t}
\newcommand{\p}{{\bf{p}}}
\newcommand{\Fjl}{{F_{{j,\ell}}^{^{A_1,A_2}}(E,h)}}
\newcommand{\HAs}{\Pi_{_{H_1,H_2}}^{^{A_1,A_2}}}

\newcommand{\RAs}{E_{_{H_1,H_2}}^{^{A_1,A_2}}}
\newcommand{\RAst}{E_{_{\tilde H_{1,h}, \tilde H_{2,h}}}^{^{A_1,A_2}}}

\DeclareMathOperator{\dv}{\operatorname{dv}}

\newcommand{\ds}[1]{d\sigma\sub{\!#1}}
\newcommand{\KR}{\mathcal{K}}
\newcommand{\Hm}{{_{\!H_{\operatorname{scl}}^{\!\!\frac{k-2m+1}{2}}\!\!\!(M)}}}

\newcommand{\Sab}{ \Sigma\sub{[a,b]}}
\newcommand{\SE}{ \Sigma\sub{E}}
\newcommand{\scl}{\operatorname{scl}}
\newcommand{\Tt}{\widetilde{T}}
\newcommand{\Cnr}{C\sub{\!\operatorname{nr}}}
\newcommand{\Cnp}{C\sub{\!\operatorname{np}}}
\newcommand{\Cnl}{C\sub{\!\operatorname{nl}}}
\newcommand{\tr}{\operatorname{tr}}

\newcommand{\Decay}{{\bf{F}}(T_0)}
\newcommand{\ti}{\mathfrak{t}}
\newcommand{\Ti}{{\bf T}}

\newcommand{\JE}{\mc{J}\sub{E}}
\newcommand{\IE}{\mc{I}\sub{E}}
\newcommand{\LE}{\mc{L}\sub{E}}
\newcommand{\BE}{\mc{B}\sub{E}}
\newcommand{\GE}{\mc{G}\sub{E}}
\newcommand{\hyp}{{\mc{Z}}}
\newcommand{\subM}{{W}}

\title[Weyl remainders]{Weyl remainders: an application of geodesic beams}

\author{Yaiza Canzani}
\address{Department of Mathematics,University of North Carolina at Chapel Hill, Chapel Hill, NC, USA}
\email{canzani@email.unc.edu}

\author{Jeffrey Galkowski}
\address{Department of Mathematics, University College London, London, UK}
\email{j.galkowski@ucl.ac.uk}

\begin{document}
%%%%%%%%%%%%%%%%%%%%%%%%%%%%%%%%%%%%%%%%%%%%%%%%%%%%%%%%%%%%%%%%%%%%%%%%%%%%%%%
%%%%%%%%%%%%%%%%%%%%%%%%%%%%%%%%%%%%%%%%%%%%%%%%%%%%%%%%%%%%%%%%%%%%%%%%%%%%%%%
%%%%%%%%%%%%%%%%%%%%%%%%%%%%%%%%%%%%%%%%%%%%%%%%%%%%%%%%%%%%%%%%%%%%%%%%%%%%%%%
%%%%%%%%%%%%%%%%%%%%%%%%%%%%%%%%%%%%%%%%%%%%%%%%%%%%%%%%%%%%%%%%%%%%%%%%%%%%%%%
\begin{abstract}
 We obtain new \emph{quantitative} estimates on Weyl Law remainders under dynamical assumptions on the geodesic flow. On a smooth compact Riemannian manifold {$(M,g)$} of dimension $n$, let $\Pi_\lambda$ denote the kernel of the {spectral projector for the Laplacian,} $\1_{[0,\lambda^2]}(-\Delta_g)$. Assuming \emph{only} that the {set of} near periodic geodesics over $\subM\subset M$ {has} small measure, we prove that as $\lambda \to \infty$
$$
\int_{\subM} \Pi_\lambda(x,x)dx=(2\pi)^{-n}\vol\sub{\mathbb{R}^n}\!(B)\vol_g(\subM)\,\lambda^n+O\Big(\frac{\lambda^{n-1}}{\log \lambda
}\Big),
$$
where $B$ is the unit ball. One consequence of this result is that the improved remainder holds on \emph{all} product manifolds, {in particular giving improved estimates for the eigenvalue counting function in the product setup.}
Our results also include logarithmic gains on asymptotics for the off-diagonal spectral projector $\Pi_\lambda(x,y)$ under the assumption that the {set of} geodesics that pass near both $x$ and $y$ {has} small measure,
{and} quantitative improvements for Kuznecov sums under non-looping type assumptions. The key {technique} used in our study of the spectral projector is that of geodesic beams.
\end{abstract}
%%%%%%%%%%%%%%%%%%%%%%%%%%%%%%%%%%%%%%%%%%%%%%%%%%%%%%%%%%%%%%%%%%%%%%%%%%%%%%%
%%%%%%%%%%%%%%%%%%%%%%%%%%%%%%%%%%%%%%%%%%%%%%%%%%%%%%%%%%%%%%%%%%%%%%%%%%%%%%%
%%%%%%%%%%%%%%%%%%%%%%%%%%%%%%%%%%%%%%%%%%%%%%%%%%%%%%%%%%%%%%%%%%%%%%%%%%%%%%%
%%%%%%%%%%%%%%%%%%%%%%%%%%%%%%%%%%%%%%%%%%%%%%%%%%%%%%%%%%%%%%%%%%%%%%%%%%%%%%%
\vspace*{-1.2cm}
\maketitle

\vspace*{-1cm}

\section{Introduction}

Let $(M,g)$ be a smooth compact Riemannian manifold of dimension $n$, $\Delta_g$ be {the negative definite} Laplace-Beltrami operator acting on $L^2(M)$, and $\{\lambda_j^2\}_{j=0}^\infty$ be {the eigenvalues of $-\Delta_g$}, repeated with multiplicity, $0=\lambda_0^2 < \lambda_1^2 \leq \lambda_2^2 \leq \dots$.
 {In this article we {obtain} improved asymptotics for both pointwise and integrated Weyl Laws. That is,} we study asymptotics for the Schwartz kernel of the spectral projector 
$$
\Pi_\lambda :L^2(M,g)\to \bigoplus_{\lambda_j\leq \lambda}\ker (-\Delta_g-\lambda_j^2),
$$
i.e. $\Pi_\lambda$ is the {orthogonal} projection operator onto functions with frequency at most $\lambda$. 
If $\{\phi_{\lambda_j}\}_{j=1}^\infty$ is an orthonormal basis of eigenfunctions, $-\Delta_g \phi_{\lambda_j}=\lambda_j^2 \phi_{\lambda_j}$, the Schwartz kernel {of $\Pi_\lambda$} is
$$
\Pi_{\lambda}(x,y)=\sum_{\lambda_j\leq \lambda} \phi_{\lambda_j}(x)\overline{\phi_{\lambda_j}}(y),\qquad (x,y)\in M\times M.
$$
Asymptotics for the spectral projector play a crucial role in the study of eigenvalues and eigenfunctions for the Laplacian, with applications to the study of physical phenomena such as wave propagation and quantum evolution. One of the oldest problems in spectral theory is to understand how eigenvalues distribute on the real line.  {Let $N(\lambda):=\#\{{j:\;}\lambda_j\leq \lambda\}$ be the eigenvalue counting function. Motivated by black body radiation, Hilbert conjectured that{, as $\lambda\to \infty$,}}
$$
N(\lambda)={(2\pi)^{-n}}\vol_{\re^n}(B)\vol_g(M)\lambda^n+E(\lambda),\qquad \qquad  {E(\lambda)=o(\lambda^{n}).}
$$
 Here, $\vol_{\re^n}(B)$ is the volume of the unit ball $B \subset \re^n$, $\vol_g(M)$ is the Riemannian volume of $M$, and $\dv_g$ is the volume measure induced by the Riemannian metric. {The conjecture was proved by Weyl ~\cite{We:12} and is known as the Weyl Law. We refer to $E(\lambda)$ as a \emph{Weyl remainder.}} {In 1968, H\"ormander~\cite{Ho68}, provided a framework for the study of $E(\lambda)$ and {generalized the {works} of {Avakumovi\'c~\cite{Ava} and Levitan \cite{Lev}}, who proved $E(\lambda)=O(\lambda^{n-1})$;}  a result {that} is sharp on the round sphere and is thought of as the standard remainder.} %{The} asymptotic expression for $N(\lambda)$ is known as a Weyl Law, and $E(\lambda)$ as a Weyl Remainder.

{{The article} \cite{Ho68} provided a framework for the study of Weyl remainders which led to many advances,} including the work of Duistermaat--Guillemin~\cite{DuGu:75} who showed $E(\lambda)=o(\lambda^{n-1})$ when the {set of} periodic {geodesics} has measure $0$. Recently,~\cite{IoWy:19} {verified this {dynamical} condition on} all product manifolds.
A {striking} application of our main theorem on Weyl remainders is: 
\begin{theorem}
\label{t:product}
Let $(M_i,g_i)$ be smooth compact Riemannian manifolds of dimension $n_i\geq 1$ {for $i=1,2$}. Then, with $M=M_1\times M_2$, $g=g_1\oplus g_2$, and $n:=n_1+n_2$,
$$
N(\lambda)={(2\pi)^{-n}}\vol_{\re^n}(B)\vol_g(M)\lambda^{n}+O\big(\lambda^{n-1}/\,{\log \lambda}\big),\qquad {\lambda\to \infty.}
$$
\end{theorem}
{For future reference, we} note that
$
{N(\lambda)=\int_{M}\Pi_{\lambda}(x,x)\dv_g(x)}
$
and thus $N(\lambda)$ can be studied by understanding the kernel of $\Pi_{\lambda}$ restricted to the diagonal. We study both on and off diagonal Weyl remainders in this article. {The} main idea is to adapt the geodesic beam techniques developed by authors \cite{Gdefect,CG17,CG19a} to study Weyl remainders. These techniques were originally used to study averages of quasimodes over submanifolds by decomposing {the quasimodes} into geodesic beams and controlling the averages in terms of {the} $L^2$ norms of these beams.  {In this work t}he key point is to study the {eigenvalue} counting function by viewing it as a sum of quasimodes averaged over the diagonal in $M\times M$. We start our exposition {in the setting of} {the} on diagonal estimates.

\subsection{On diagonal Weyl remainders}

The connection between the spectrum of the Laplacian and the properties of periodic geodesics on $M$ has been known since at least the works~\cite{Cha:74,CdV:73, We:75}, with their {relation} to Weyl remainders first explored in the seminal work~\cite{DuGu:75}. To control $E(\lambda)$ we impose dynamical conditions on the periodicity properties of the {geodesic flow $\varphi_t:T^*M \setminus \{0\} \to T^*M{\setminus \{0\}}$, i.e., the Hamiltonian flow of $(x,\xi)\mapsto|\xi|_{g(x)}$.} For $t_0>0$, $T>0$, and $R>0$, define the set of near periodic directions in $U \subset S^*M$ by
\begin{equation}
\label{e:sadWithoutNumber}
\mc{P}^R\sub{U}(t_0,T):=\bigg\{ \rho \in U\, :\, \bigcup_{t_0\leq |t|\leq T}\varphi_t(B\sub{S^*\!M}(\rho,R))\cap B\sub{S^*\!M}\!(\rho,R)\neq \emptyset\bigg\}.
\end{equation}
%Also, let $\blue{B\sub{S^*\!M}\!\big({\scriptstyle \mc{P}^R(t_0,T)}, S\big)}:=\big\{ \rho\in S^*M\,:\, d\big({\scriptstyle \mc{P}^R(t_0,T)},\rho\big)< S\big\}$ for $S>0$.
Given two sets $U\subset V\subset T^*M$, and $R>0$, we write $B\sub{V}(U,R):=\{\rho \in V:\; d(U, \rho)<R\}$, where $d$ is the {distance induced by some {fixed} metric on $T^*M$},  $B(U,R)=B\sub{T^*\!M}(U,R)$, and {$B\sub{V}(\rho,R)=B\sub{V}(\{\rho\},R)$.}

We phrase our dynamical conditions in terms of a \resfun\, {$\Ti=\Ti(R)$. This is a function of the scale, $R$, at which the manifold is resolved, which increases as $R\to 0^+$. We use $\Ti$ to measure the time for which balls of radius $R$ can be propagated under the geodesic flow while satisfying a given dynamical assumption, e.g. being non periodic.}
\begin{definition}
We say a decreasing, {continuous} function ${\Ti}:(0,\infty)\to (0,\infty)$ is a \emph{\resfun}.
In addition, we say a {\resfun} $\Ti$ is \emph{sub-logarithmic}, if {it is differentiable and} $$
{(\log \log R^{-1})'}=-1\big/R\log R^{-1}\leq [\log \Ti(R)]'\leq 0,\qquad 0<R<1.\,
$$
We {measure how close $\Ti$ is to being logarithmic through}
\begin{equation}\label{e:Omega}
\Omega(\Ti):=\limsup_{R\to 0^+}\Ti(R)\big/\log R^{-1}.
\end{equation}
\end{definition}
\noindent {Simple examples of sub-logarithmic {\resfuns} are $\Ti(R)=\alpha (\log R^{-1})^\beta$ for any $\alpha>0$ and $0<\beta \leq 1$.} {See Remark~\ref{r:rate function} for comments on the use of general \resfuns.}

For improved integrated Weyl remainders, {we need a} condition on the geodesic flow.

\begin{definition}
\label{d:non periodic}
Let $\Ti$ be a \resfun.
Then $U\subset S^*M$ is said to be \emph{$\Ti$ non-periodic} {with constant $\Cnp$} provided there exists $t_0>0$ such that 
\begin{equation}\label{e:volume}
\limsup_{R\to 0^+}\mu\sub{S^*\!M}\!\Big(B\sub{S^*\!M}\!\big({\smalleq{ \mc{P}^R\sub{U}(t_0,\Ti(R))}}, R\big)\Big)\,\Ti(R)\leq \Cnp.
\end{equation}
{We say $U$ is $\Ti$ non-periodic if there is such $\Cnp$}, and $\subM\subset M$ is $\Ti$ non-periodic if $S^*_{\subM}M$ is.
\end{definition}
\noindent{See Appendix~\ref{s:index} for the notation $\mu\sub{\!{S^*_xM}}$ and $\dim_{\textup{box}}$ used below}.
%{{For} ${U}\subset T^*\!M$ we write $\mu\sub{\!{U}}$ for the Liouville measure induced on ${U}$,} and $\dim_{\textup{box}}{U}$ for the Minkowski box dimension of ${U}$ (see e.g.~\cite[Page 333]{StSh:05}). Note that if $\subM\subset M$ is open with smooth boundary then $\dim_{\textup{box}}\partial \subM=n-1$.

%%%%%%%%%%%%%%%%%%%%%%%%%%%%%%%%%%%%%%%%%%%%%%%%%%%%%%%%%%%%%%%%%%%%%%%%%%%%%%%
\begin{theorem}
\label{t:laplaceWeyl}
Let $(M,g)$ be a Riemannian manifold of dimension $n$, $\subM \subset M$ be an open subset with $\dim_{\textup{box}}\partial \subM<n$, {and ${\Omega_0}>0$}. {There exists} $C\sub{0}>0$ such that if $\Ti$ is a sub-logarithmic rate {function} with $\Omega(\Ti)<{\Omega_0}$ and $\subM$ is $\Ti$ non-periodic, then there is { $\lambda_0$ such that for all $\lambda >\lambda_0$}
$$
\Big|\int_{\subM} \Pi_\lambda(x,x) \dv_g(x)-{(2\pi)^{-n}}\vol_{\re^n}(B)\vol_g(\subM)\lambda^n\Big| \leq C\sub{0}\,\lambda^{n-1}\big/\,\Ti \big(\lambda^{-1}\big).
$$
In particular, if $M$ is $\Ti$ non-periodic, {then there is $\lambda_0$ such that for all $\lambda >\lambda_0$}
$$
\Big|N(\lambda)-{(2\pi)^{-n}}\vol_{\re^n}(B)\vol_g(M)\lambda^n\Big| \leq C\sub{0}\,\lambda^{n-1}\big/\,\Ti \big(\lambda^{-1}\big).
$$
\end{theorem}
%\blue{\cs add examples and mention Volovoy. showcase the strength of the result.\cs}

%%%%%%%%%%%%%%%%%%%%%%%%%%%%%%%%%%%%%%%%%%%%%%%%%%%%%%%%%%%%%%%%%%%%%%%%%%%%%%
We illustrate an application of Theorem~\ref{t:laplaceWeyl} in Figure \ref{f:perturb}. In this example we construct a surface of revolution with both a periodic and a non-periodic set ({see Definition~\ref{d:non periodic}}). In particular, Theorem~\ref{t:laplaceWeyl} applies with $\subM$ contained in the {non-periodic (green)} set. 
{One can obtain little oh improvements for the statement in Theorem \ref{t:laplaceWeyl}, but {this requires} the more general version given in Theorem~\ref{t:genWeyl} instead (see Remark \ref{r:little-oh}).} {See Table \ref{ta:exWeyl} in \S\ref{s:concreteApps} for {some} {additional} examples.}

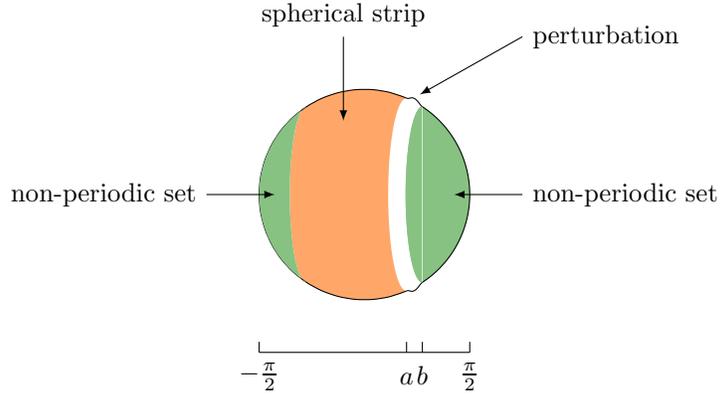
\begin{figure}[htbp]
\captionsetup{width=.9\textwidth}
\begin{tikzpicture}
\def \perspectiveang{45};
\def \perspectiveangb{11};
\def \a{.4};
\def \b{.55};
\def \mult{5};
\def \eps{.001};
\def \anga{{atan2({sqrt(1-\a*\a)},-\a)}};
\def \angb{{atan2({-sqrt(1-\a*\a)},-\a)}};
\def \angd{{360+atan2({-sqrt(1-\b*\b)},-\b)}};
\def \angc{{atan2({sqrt(1-\b*\b)},-\b)}};
\def \ange{{atan2({sqrt(1-\a*\a)},\a)}};
\def \angf{{atan2({-sqrt(1-\a*\a)},\a)}};
\def \angg{{atan2({sqrt(1-\b*\b)},\b)}};
\def \angh{{atan2({-sqrt(1-\b*\b)},\b)}};

\begin{scope}[xscale=-1.4,yscale=1.4]

\fill[periodColor] plot[domain=\anga:\angg, smooth, variable=\x] ({cos(\x)},{sin(\x)})--plot[domain=0:180,smooth, variable =\x] ({\b+sqrt(1-\b*\b)*sin(\perspectiveangb)*sin(\x)},{sqrt(1-\b*\b)*cos(\x)})--plot[domain=\angh:\angb, smooth, variable=\x] ({cos(\x)},{sin(\x)})-- plot[domain=0:180,smooth, variable =\x] ({-\a+sqrt(1-\a*\a)*sin(\perspectiveangb)*sin(\x)},{-sqrt(1-\a*\a)*cos(\x)});

\draw[domain=\anga:\angb, smooth, variable=\x] plot ({cos(\x)},{sin(\x)});

\fill[dark-green] plot[domain=\angc:\angd, smooth, variable=\x] ({cos(\x)},{sin(\x)})-- plot[domain=0:180,smooth, variable =\x,dotted] ({-\b+sqrt(1-\b*\b)*sin(\perspectiveangb)*sin(\x)},{sqrt(1-\b*\b)*cos(\x)});;

\fill[dark-green] plot[domain=\angc:\angd, smooth, variable=\x] ({cos(\x)},{sin(\x)})-- plot[domain=0:180,smooth, variable =\x,dotted] ({-\b+sqrt(1-\b*\b)*sin(\perspectiveangb)*sin(\x)},{sqrt(1-\b*\b)*cos(\x)});

\fill[dark-green] plot[domain=\angg:\angh, smooth, variable=\x] ({cos(\x)},{sin(\x)})-- plot[domain=0:180,smooth, variable =\x,dotted] ({\b+sqrt(1-\b*\b)*sin(\perspectiveangb)*sin(\x)},{-sqrt(1-\b*\b)*cos(\x)});

%\draw[domain=0:180,smooth, variable =\x]plot ({\a+sqrt(1-\a*\a)*sin(\perspectiveangb)*sin(\x)},{sqrt(1-\a*\a)*cos(\x)});
\draw[domain=\angc:\angd, smooth, variable=\x] plot ({cos(\x)},{sin(\x)});
\draw (-\b,-1.5)--(-\b,-1.4);
\draw (-\a,-1.5)--(-\a,-1.4);
\draw (-\a,-1.9)node[above]{$a$};
\draw (-\b,-1.9)node[above]{$b$};
\draw (-1,-1.5)node[below]{$\frac{\pi}{2}$}--(-1,-1.4);
\draw (-1,-1.5)--(1,-1.5);
\draw (1,-1.5)node[below]{$-\tfrac{\pi}{2}$}--(1,-1.4);
\draw[->] (-1.5,1.5)node[right]{perturbation}--({(-\a-\b)/2-.05},{sqrt(1-(\a+\b)*(\a+\b)/4)+.07});

\draw[scale=1, domain=-\b:-\a, smooth, variable=\x] plot ({\x},{sqrt(1-\x*\x)+\mult*pow(2,-1/sqrt(-\x-\a+\eps))*pow(2,-1/sqrt(\x+\b+\eps))});

\draw[scale=1, domain=-\b:-\a, smooth, variable=\x] plot ({\x},-{sqrt(1-\x*\x)-\mult*pow(2,-1/sqrt(-\x-\a+\eps))*pow(2,-1/sqrt(\x+\b+\eps))});
\draw[->] (.2,1.5)node[above]{spherical strip}--(.2,.7);
\draw[->] (-1.5,0) node[right]{non-periodic set}--(-.85,0);
\draw[->] (1.5,0) node[left]{non-periodic set}--(.85,0);
\end{scope}
\end{tikzpicture}
\caption{An example of a perturbation of the sphere with both a non-periodic (green) and a periodic (orange) physical space set. The perturbed metric coincides with the round metric outside the strip $(a,b)$. Trajectories which remain in the spherical strip are $2\pi$ periodic, while those which enter the non-periodic set are mostly non-periodic. See Section~\ref{s:perturbEx} for a precise description of this example.}% In this example we take $f_+\geq 0$ non-trivial and $f_-\equiv 0$ non-trivial.}
\label{f:perturb}
\end{figure}

{The assumptions of Theorem~\ref{t:laplaceWeyl} apply to a wide variety of Riemannian manifolds. Indeed, in addition to the concrete examples in \S\ref{s:concreteApps}, the authors~\cite{CaGa:22} use Theorem~\ref{t:laplaceWeyl} to give a logarithmic improvement in the remainder for the Weyl law that works for `typical' metrics on any smooth manifold. 
%The notion of typicality used there, called predominance, is an analog of full Lebesgue measure in the space of Riemannian metrics. For example, a predominant set of metrics is dense. The authors show that there is $C_n>0$ such that for any $\nu$ large enough, a predominant $C^\nu$ metric on $M$ is $\Ti$ non-periodic with $\Ti(R)=(\log R^{-1})^{1/\alpha_\nu}$, and $\alpha_\nu=C_n+\log_2\nu$. An application of Theorem~\ref{t:laplaceWeyl} then implies that for these metrics
%$$N(\lambda)=(2\pi )^{-n}\vol_{\mathbb{R}^n}(B)\vol_g(M)\lambda^n+O(\lambda^{n-1}/(\log \lambda)^{1/\alpha_\nu}).$$
%This theorem~\cite[Theorerm 1]{CaGa:22}, is made possible by the flexibility provided by the assumptions in Theorem~\ref{t:laplaceWeyl} together with a delicate dynamical systems argument. 
This result is the first \emph{quantitative} estimate for the remainder in Weyl laws that holds for most metrics.
}

 We next discuss $E_\lambda(x)$, the remainder in the {on diagonal} pointwise Weyl law
\begin{equation}
\label{e:pwWeyl}
\Pi_{\lambda}(x,x)={(2\pi)^{-n}}\vol_{\re^n} (B)\lambda^n +E_\lambda(x), \qquad x \in M.
\end{equation}

{The Weyl remainder in~\cite{Ho68} comes from} the estimate $E_{\lambda}(x)=O(\lambda^{n-1})$ for $x\in M$ (again, sharp on the round sphere). The connection between $E_\lambda(x)$ and geodesic loops through $x$ is studied in the {works of Safarov, Sogge--Zelditch}~\cite{Saf88,SoggeZelditch} and often appears in estimates for sup-norms of eigenfunctions. To control the pointwise remainder $ E_\lambda(x)$ we impose dynamical conditions on the looping properties of geodesics joining $x$ with itself. 
For $t_0>0$, $T>0$, {$R>0$,} and $x,y\in M$, define
\begin{equation}
\label{e:L}
 \mathcal{L}_{x,y}^{{R}}(t_0, T):=\bigg\{ {{\rho} \in S_x^*M}:\; \bigcup_{t_0\leq |t|\leq T}\varphi_t(B(\rho,R)) \,\cap\, {B({S_y^*M},R)} \neq \emptyset\bigg\}.
\end{equation}
%Also, let $B\big({\scriptstyle \mathcal{L}_{x,y}^{{R}}(t_0, T)}, S\big):=\big\{ \rho \in S_y^*M:\, d\big({\scriptstyle \mathcal{L}_{x,y}^{{R}}(t_0, T)}, \rho \big)<{R}\big\}$ for $S>0$.
\begin{definition}\label{d:non looping through}
Let $\Ti$ be a \resfun, $t_0> 0$, {$\Cnl>0$,} and $x, y \in M$. Then, \emph{$(x,y)$ is said to be a {$(t_0, \Ti)$}
%$[t_0, T(\lambda)]$ 
non-looping pair with constant $\Cnl$ {when} }
$$
 \limsup_{R\to 0^+} \bigg({\mu\sub{{S^*_xM}}}\Big(B\sub{{S^*_xM}}\!\big({\smalleq{ \mathcal{L}_{x,y}^{{R}}(t_0, \Ti(R))}}, R\big)\Big) {\mu\sub{{S^*_yM}}}\Big(B\sub{{S^*_yM}}\!\big({\smalleq{ \mathcal{L}_{y,x}^{{R}}(t_0, \Ti(R))}}, R\big)\Big) \Ti(R)^2 \bigg)\leq \Cnl.
$$ 
We say $x$ is \emph{$(t_0,\Ti)$ non-looping with constant $\Cnl$} if $(x,x)$ is a $(t_0,\Ti)$ non-looping pair with constant $\Cnl$.
\end{definition}
 
Note that if $t_0<\inj(M)$, {where $\inj(M)$ is the injectivity radius of $M$}, then for $x$ to be $(t_0, \Ti)$ non-looping
 is the same as being $(\ep, \Ti)$ non-looping for any $0<\ep \leq t_0$. In this case, we write $x$ is $(0, \Ti)$ non-looping.

To state our estimates on the pointwise Weyl remainder, we let $\lambda>0$, {and, for points $x,y \in M$ with $d(x,y)<\inj M$,}
define
\begin{equation}\label{e:R^0}
 E_\lambda^0(x,y):=\Pi_\lambda(x,y)-\frac{1}{(2\pi)^n}\int_{|\xi|_{g_y}<\lambda}e^{i\langle \exp_{y}^{-1}(x),\xi\rangle }\frac{d\xi}{\sqrt{|g_y|}} .
\end{equation}
Here, the integral is over $T_y^*M$, $\exp_x:T_x^*M\to M$ is the the exponential map, and $|g_y|$ denotes the determinant of the metric $g$ at $y$, when $g$ is thought of as matrix in local coordinates. 
%Note, that the RHS in \eqref{e:R^0} is coordinate invariant.

%%%%%%%%%%%%%%%%%%%%%%%%%%%%%%%%%%%%%%%%%%%%%%%%%%%%%%%%%%%%%%%%%%%%%%%%%%%%%%%
\begin{theorem}
\label{t:laplaceDiag} 
Let $\alpha,\beta \in \mathbb{N}^n$, $0< \delta<\frac{1}{2}$, $\Cnl>0$, and ${\Omega_0}>0$. There exists $C\sub{0}>0$ such that {the following holds. If $\Ti$ is a sub-logarithmic \resfun\, with $\Omega(\Ti)<{\Omega_0}$, there is $\lambda_0>0$ such that if $x_0\in M$ is $(0,\Ti)$ non-looping with constant $\Cnl$, then for all $\lambda>\lambda_0$
$$
\sup_{x, y \in B(x_0, \lambda^{-\delta})}\big|\partial_x^\alpha \partial_y^\beta E_\lambda^0(x,y)\big|\leq C\sub{0}\,\lambda^{n-1+|\alpha|+|\beta|}\big/\,\Ti \big(\lambda^{-1}\big).
$$
}
\end{theorem}
%%%%%%%%%%%%%%%%%%%%%%%%%%%%%%%%%%%%%%%%%%%%%%%%%%%%%%%%%%%%%%%%%%%%%%%%%%%%%%%
\noindent See Table \ref{ta:ex} in \S\ref{s:concreteApps} for {some} examples to which Theorem \ref{t:laplaceDiag} {applies}.

Theorems~\ref{t:laplaceWeyl} and~\ref{t:laplaceDiag} fit in a long history of work on asymptotics of the kernel of the spectral projector and the eigenvalue counting function. Many authors considered pointwise Weyl sums~\cite{Pl:49,Ga:53,Ava,Lev,Se:67,Ho68}, {eventually proving the sharp remainder estimates}. The article~\cite{Ho68} provided a method which was used in many later works:~\cite{DuGu:75} showed $E(\lambda)=o(\lambda^{n-1})$ under the assumption that the {set of} periodic trajectories has measure 0,~\cite{Saf88,SoggeZelditch} improved estimates on $E_\lambda(x)$ to $o(\lambda^{n-1})$ under the assumption that the set of looping directions through $x$ has measure 0 {(see also the book of Safarov--Vassiliev~\cite{SaVa:97})}. See~\cite{CaHa:15,CaHa:18} for corresponding estimates that are uniform in a small neighborhood of the diagonal {and} Ivrii~\cite{Iv:80} for the case of manifolds with boundaries.

While $o(1)$ improvements were available under dynamical assumptions, until now, quantitative improvements in remainders were available in geometries where one has an effective parametrix to $\log \lambda$ times e.g. manifolds without conjugate points~\cite{Berard77,Bo:17, Ke:19} or non-Zoll convex analytic rotation surfaces~\cite{Vo:90,Vo:90b}. We point out that the closest results to ours are those of Volovoy~\cite{Vo:90}. There, quantitative estimates on $E(\lambda)$ are obtained under stronger assumptions than those of Theorem~\ref{t:laplaceWeyl}. In particular, $\subM$ is required to be equal to $M$ and the volume in~\eqref{e:volume} is required to be bounded by a positive power of $R$, rather than $\Ti(R)^{-1}$.

The estimates in this article are available \emph{without additional} geometric assumptions. This comes from our use of {the 'geodesic beam techniques' developed in the} authors' work~\cite{Gdefect, CG17, CG19a} and which in turn draw upon the semiclassical approach of Koch--Tataru--Zworski~\cite{KTZ}. Theorems~\ref{t:laplaceWeyl} and~\ref{t:laplaceDiag} can be thought of as the quantitative analogs of the {main results in~\cite{DuGu:75} and  of~\cite{Saf88}},~\cite{SoggeZelditch} respectively. In fact, these results can be recovered from Theorems~\ref{t:laplaceWeyl} and~\ref{t:laplaceDiag} by allowing $\Ti(R)$ to grow arbitrarily slowly as $R\to 0^+$ {(see~\cite[Appendix B]{CG19a})}. We also note that our estimates include both $C^\infty$ asymptotics for $\Pi_\lambda(x,y)$ and uniformity in certain shrinking neighborhoods of the diagonal without any additional effort and hence include the results from~\cite{CaHa:15,CaHa:18}.

\begin{remark}
To recover the results of~\cite{Saf88,SoggeZelditch, CaHa:15,CaHa:18} one needs uniformity in $o(1)$ neighborhoods of points of interest. As stated, Theorem~\ref{t:laplaceDiag} does not quite include this since it works in a $\lambda^{-\delta}$ neighborhood of $x$. However, the full version of our estimates, Theorem~\ref{t:MAIN2}, allows for the neighborhood of $x$ to shrink arbitrarily slowly and thus recovers these earlier results.
\end{remark}

%%%%%%%%%%%%%%%%%%%%%%%%%%%%%%%%%%%%%%%%%%%%%%%%%%%%%%%%%%%%%%%%%%%%%%%%%%%%%%%
%%%%%%%%%%%%%%%%%%%%%%%%%%%%%%%%%%%%%%%%%%%%%%%%%%%%%%%%%%%%%%%%%%%%%%%%%%%%%%%
\subsection{Off diagonal Weyl remainders}
%%%%%%%%%%%%%%%%%%%%%%%%%%%%%%%%%%%%%%%%%%%%%%%%%%%%%%%%%%%%%%%%%%%%%%%%%%%%%%%
%%%%%%%%%%%%%%%%%%%%%%%%%%%%%%%%%%%%%%%%%%%%%%%%%%%%%%%%%%%%%%%%%%%%%%%%%%%%%%%
The off diagonal behavior of $\Pi_\lambda(x,y)$ plays a crucial role in understanding monochromatic random waves (see e.g.~\cite{Ca20}) as well as in estimates for $L^p$ norms of Laplace eigenfunctions (see e.g.~\cite[Section 5.1]{SoggeBook}). This problem is more complicated than the on diagonal situation since understanding the far off diagonal (i.e., $d(x,y)>\inj (M)$) regime typically involves parametrices for $e^{it\sqrt{-\Delta_g}}$ for $t>\inj(M)$, which are difficult to control. Notably, our geodesic beam techniques allow us to overcome this difficulty when estimating errors.

To control $\Pi_\lambda(x,y)$ off-diagonal, we introduce a dynamical condition on the non-recurrence properties of the geodesics joining a point $x$ with itself. {To our knowledge, this is the first time non-recurrence is used in understanding off-diagonal Weyl remainders.} For {$x \in M$, ${U} \subset S_x^*M$, $t_0>0$, $T>0$, and} $R>0$, let
$$
\mc{R}^{R}\sub{{U},{\pm}}(t_0,T):= \bigcup_{{t_0\leq \pm t\leq T}}\varphi_t\big (B({U},R) \big) \cap B\sub{{S^*_xM}}\!({U},R).
$$

%%Also, let $B\sub{\blue{S^*_xM}}\!\big({\scriptstyle \mc{R}^{R}(t_0,T,A)}, S \big):=\{ \rho \in \blue{S^*_xM}:\, d({\scriptstyle \mc{R}^{R}(t_0,T,A)}, \rho )<S\}$ for $S>0$.
\begin{definition}
\label{d:non rec}
Let {$\ti$ and} $\Ti$ be \resfuns\, and $R_0>0$. We say \emph{{$x \in M$} is ${(\ti,}\Ti)$ non-recurrent {at scale $R_0$}} if for all $\rho\in S^*_xM$ there exists a choice of $\pm$ such that for all $A\subset B\sub{S^*_xM}(\rho,R_0)$, {$\e>0$}, $r>0$ with $\Ti(r)>{\ti(\e)}$, {and $0<R<R_0$,}
$$
\mu\sub{{S^*_xM}}\Big(B\sub{{S^*_xM}}\!\Big({\smalleq{ \mc{R}^{{r}R}\sub{A,\pm}({\ti(\e)},\Ti(r))}}\,,\, r R\Big)\Big)<\e \,\mu\sub{{S^*_xM}} \Big(B\sub{S^*_xM}\!(A,R)\Big).
$$
\end{definition}

%Again, we note that if $t_0=\inj(M)$, where $\inj(M)$ is the injectivity radius of $M$, then being $(t_0, \Ti)$ non-recurrent is the same as being $(\ep, \Ti)$ non-recurrent for any $0<\ep \leq t_0$. In this case we will simply say that the point is $(0, \Ti)$ non-recurrent.

%In Lemma \ref{l:non loop then non rec} we prove that
%$$
% \text{$x$ is $[t_0, T(\lambda)]$ non-looping through $x$}
%\quad \Longrightarrow \quad
%\text{$x$ is $[t_0, T(\lambda)]$ non-recurrent}.
%$$ 
%We are now ready to state our first off diagonal result.

%%%%%%%%%%%%%%%%%%%%%%%%%%%%%%%%%%%%%%%%%%%%%%%%%%%%%%%%%%%%%%%%%%%%%%%%%%%%%%%
%\begin{theorem}
%\label{t:laplaceoff always} 
%For any two multi-indices $\alpha,\beta \in \mathbb N^n$, $0<\delta<\frac{1}{2}$, $\Cnr,C_l>0$ there exists $C_{\alpha,\beta}>0$ such that the following holds.
%Let $x_0,y_0 \in M$ with $d(x_0,y_0)<\frac{\inj(M)}{2}$ and $\Ti$ a sub-logarithmic \resfun such that both $x_0$ and $y_0$ are $\Ti$ non-recurrent, and such that $(x_0,y_0)$ is a $(0, \Ti)$ non-looping pair. Then,
%$$
%\sup_{x \in B(x_0, \lambda^{-\delta})}\sup_{y \in B(y_0, \lambda^{-\delta})}\big|\partial_x^\alpha \partial_y^\beta R^0_\lambda(x,y)\big|\leq C_{\alpha,\beta}\frac{\lambda^{n-1+|\alpha|+|\beta|}}{\Ti \big(\lambda^{-1}\big)}.$$
%\end{theorem}
%%%%%%%%%%%%%%%%%%%%%%%%%%%%%%%%%%%%%%%%%%%%%%%%%%%%%%%%%%%%%%%%%%%%%%%%%%%%%%%
\medskip

If $(x,y)$ is a {$(t_0, \Ti)$} non looping pair for some $t_0>0$ we measure the difference between $\Pi_\lambda(x,y)$ and its smoothed version {which takes into account propagation} up to time $t_0$.
Let $\rho\in \mc{S}(\mathbb{R})$ with $\hat{\rho}(0)\equiv 1$ on $[-1,1]$ and $\supp \hat{\rho}\subset [-2,2]$. 
For $\sigma>0$ we define
\begin{equation}\label{e:rho def lambda}
 \rho_{\sigma}(s):= \sigma \, \rho \big(\sigma \, s \big). 
\end{equation}
 For $x,y \in M$, $t_0>0$, and $\lambda>0$, let
\begin{equation}
\label{e:smoothie}
 E_\lambda^{t_0}:=\Pi_\lambda-\rho\sub{t_0}*\Pi_{\lambda},
\end{equation}
where the convolution is taken in the $\lambda$ variable. 
%\begin{remark}
%Note that the smoothing in~\eqref{e:smoothie} results from the fact that we are fundamentally using the Schr\"odinger propagator in our smoothing rather than the wave propagator. In particular, observe that 
%$$
%\partial_s(\rho\sub{t_0 \lambda}*\Pi_{\lambda})(s)=t_0\lambda\, \rho\big(\lambda( s+\lambda^{-2}\Delta_g)\big)=\frac{\lambda}{2\pi}\int \hat{\rho}\big(\frac{\tau}{t_0}\big)e^{i\tau\lambda (\lambda^{-2}\Delta_g+s)}d\tau.
%$$
%so we consider only the singularities in the Schr\"odinger propagator to time $2t_0$. The fact that our estimates are stated in terms of this type of smoothing results from our use of semiclassical rather than homogeneous methods. Note for $\chi\equiv 1$ on $[-2,2]$, and $s\leq 1$, 
%$$
%(1-\chi(\lambda^{-2}\Delta_g))\partial_s(\rho\sub{t_0\lambda}\Pi_\lambda)=O(\lambda^{-\infty})_{H^{-N}\to H^N},
%$$
%therefore standard methods of semiclassical analysis e.g.~\cite[Chapter 10]{EZB} can be used to write an explicit oscillatory expression for $\rho_{t_0\lambda}*\Pi_\lambda(1).$ 
%\end{remark}
Below is our first off diagonal result.
%%%%%%%%%%%%%%%%%%%%%%%%%%%%%%%%%%%%%%%%%%%%%%%%%%%%%%%%%%%%%%%%%%%%%%%%%%%%%%%
\begin{theorem}
\label{t:laplaceoff} 
Let $\alpha,\beta \in \mathbb{N}^n$, $0<\delta<\frac{1}{2}$, $\Cnl>0$, ${R_0>0}$, ${\Omega_0}>0$, ${\e}>0$, {and $\ti$ be a \resfun,} {there is $C\sub{0}>0$ such that if $\Ti_j$ is a sub-logarithmic \resfun\, with $\Omega(\Ti_j)<{\Omega_0}$ for $j=1,2$ {and $\Ti_{\max}=\max(\Ti_1,\Ti_2)$}, then there is $\lambda_0>0$ such the following holds. If $x_0,y_0 \in M$ and $t_0>0$ are such that $x_0$ and $y_0$ {are respectively {$(\ti,\Ti_1)$ and $(\ti,\Ti_2)$} non-recurrent {at scale $R_0$}}, and $(x_0,y_0)$ is a $(t_0,{\Ti_{\max}})$ non-looping pair with constant $\Cnl$, then for $\lambda>\lambda_0$
$$
\sup_{x \in B(x_0, \lambda^{-\delta})}\sup_{y \in B(y_0, \lambda^{-\delta})}\big|\partial_x^\alpha \partial_y^\beta E^{t_0{+\e}}_\lambda(x,y)\big|\leq C\sub{0}\,\lambda^{n-1+|\alpha|+|\beta|}\bigg/\!\sqrt{\Ti_1 \big(\lambda^{-1}\big)\Ti_2(\lambda^{-1})}.
$$}
\end{theorem}
%%%%%%%%%%%%%%%%%%%%%%%%%%%%%%%%%%%%%%%%%%%%%%%%%%%%%%%%%%%%%%%%%%%%%%%%%%%%%%%
\noindent See Table \ref{ta:ex} in \S\ref{s:concreteApps} for {some} examples to which Theorem \ref{t:laplaceoff} applies.

To compare Theorems~\ref{t:laplaceDiag} and~\ref{t:laplaceoff}, note that for $x,y\in M$ with $d(x,y)<\e<\inj(M)$,
$$\bigg|\partial^\alpha_x\partial^\beta_y \bigg(\rho\sub{\ep \lambda}*{\Pi_\lambda}(x,y)-\frac{1}{(2\pi)^n}\int_{|\xi|_{g_y}<\lambda}e^{i\langle \exp_{y}^{-1}(x),\xi\rangle }q_{{\lambda}}(x,y,\xi)\frac{d\xi}{\sqrt{|g_y|}}\bigg)\bigg|\leq C\sub{0}\lambda^{n-2+|\alpha|+|\beta|}$$
where $q_{{\lambda}}(x,y, \xi){=} 1+\lambda^{-1}q_{-1}(x,y,\xi)$ and $q_{-1}(x,y,\xi)=O(d(x,y))$ (see e.g.~\cite[Proof of Proposition 10]{CaHa:15}).
Then, {for {points $x,y$ with} $d(x,y)<\lambda^{-\delta}$}, modulo terms smaller than our remainder, $E^0_\lambda(x,y)$ as defined in \eqref{e:R^0} is the same as $E^\ep_\lambda(x,y)$.

For any $t_0<\infty$, it is possible to write an oscillatory integral expression for {$\rho_{t_0}*\Pi_{\lambda}(x,y)$}. However, its precise behavior in $\lambda$ depends heavily on the geometry of $(M,g)$; in particular, on the structure of the set of geodesics from $x$ to $y$. This explains why we state our estimates in terms of $E_{\lambda}^{t_0}$. 
%{\begin{remark}
%Note that our non-periodic, non-looping, and non-recurrent conditions are all monotonic in $\Ti$ in the sense that if $\Ti_1(R)\leq \Ti_2(R)$, and one of these conditions hold with the {\resfun} $\Ti_2$, then it also holds with $\Ti_1$.
%\end{remark}}

More generally, our results apply to averages of $\Pi_\lambda(x,y)$ with $x \in H_1$ and $y \in H_2$, where $H_1,H_2$ are any two smooth submanifolds of $M$. {This type of integral is known as a Kuznecov sum~\cite{Zel} and appears in the analytic theory of automorphic forms~\cite{Br:81,Ku:80,Iw:84,Good,Hej}.} All our dynamical assumptions for points $x,y \in M$ above may be defined for the submanifolds $H_1,H_2 \subset M$ instead. In doing so, the only change needed is to use the sets of unit co-normal directions $S\!N^*\!H_1$ and $S\!N^*\!H_2$, instead of $S_x^*M$ and $S_y^*M$. See Definitions \ref{d:non loop gral} and \ref{d:non rec gral} for a detailed explanation. In what follows $\ds{H_1}$ and $\ds{H_2}$ denote the volume measures induced by the Riemannian metric on $H_1$ and $H_2$ respectively.

%%%%%%%%%%%%%%%%%%%%%%%%%%%%%%%%%%%%%%%%%%%%%%%%%%%%%%%%%%%%%%%%%%%%%%%%%%%%%%%
\begin{theorem}
\label{t:laplaceOff for H}
Let $\alpha,\beta \in \mathbb{N}^n$, ${1}\leq k_1\leq n$, ${1}\leq k_2\leq n$, $\Cnl>0$, ${\Omega_0}>0$, ${\e}>0$, ${R_0}>0$, {and $\ti$ be a \resfun}. There is $C\sub{0}>0$ such that
if $\Ti_j$ is a sub-logarithmic \resfun\, with $\Omega(\Ti_j)<{\Omega_0}$ for $j=1,2$ {and $\Ti_{\max}=\max(\Ti_1,\Ti_2)$} the following holds.
If $t_0>0$, and $H_j\subset M$ are submanifolds of codimension $k_j$ such that $(H_1,H_2)$ is a $(t_0,{\Ti_{\max}})$ non-looping pair with constant $\Cnl$, and $H_j$ is $(\ti,\Ti_j)$ non-recurrent {at scale $R_0$} {for $j=1,2$}, {then there is $\lambda_0>0$ such that} for $\lambda>\lambda_0$
$$
\left|\int_{H_1}\int_{H_2}\partial_x^\alpha \partial_y^\beta E^{t_0{+\e}}_\lambda(x,y)\;\ds{H_1}(x)\ds{H_2}(y)\right|\leq C\sub{0}\,\lambda^{\frac{k_1+k_2}{2}-1+|\alpha|+|\beta|}\bigg/\!\sqrt{\Ti_1 \big(\lambda^{-1}\big)\Ti_2(\lambda^{-1})}.
$$
\end{theorem}
%%%%%%%%%%%%%%%%%%%%%%%%%%%%%%%%%%%%%%%%%%%%%%%%%%%%%%%%%%%%%%%%%%%%%%%%%%%%%%%
\noindent See Table \ref{ta:ex} in \S\ref{s:concreteApps} for {some} examples to which Theorem \ref{t:laplaceOff for H} {applies}.

 To our knowledge, Theorem~\ref{t:laplaceOff for H} is the first theorem {to give} improved remainders for Kuznecov sum remainders under dynamical assumptions.
%For examples where the non-recurrent and non-looping assumptions hold, we refer the reader to~\cite{CG19dyn}. \blue{\cs need to discuss how to actually say this\cs}
Theorems~\ref{t:laplaceDiag}, \ref{t:laplaceoff}, and \ref{t:laplaceOff for H} are consequences of our results for general semiclassical pseudodifferential operators {(see Theorems~\ref{t:MAIN} and~\ref{t:MAIN2})}.

\begin{remark}[Little oh improvements]\label{r:little-oh}
When the expansion rate $\Lambda_{\max}=0$ (see~\eqref{e:Lmax}) and our {dynamical} assumptions hold for $\Ti(R)\gg \log R^{-1}$, our theorems can be used to obtain $o(1/\log \lambda)$ improvements over standard remainders. In special situations where the geodesic flow has sub-exponential expansion, we expect similar results with improvements beyond $o(1/\log \lambda)$.
\end{remark}

%%%%%%%%%%%%%%%%%%%%%%%%%%%%%%%%%%%%%%%%%%%%%%%%%%%%%%%%%%%%%%%%%%%%%%%%%%%%%%%
%%%%%%%%%%%%%%%%%%%%%%%%%%%%%%%%%%%%%%%%%%%%%%%%%%%%%%%%%%%%%%%%%%%%%%%%%%%%%%%
 \subsection{Applications} \label{s:concreteApps}
 %%%%%%%%%%%%%%%%%%%%%%%%%%%%%%%%%%%%%%%%%%%%%%%%%%%%%%%%%%%%%%%%%%%%%%%%%%%%%%%
 %%%%%%%%%%%%%%%%%%%%%%%%%%%%%%%%%%%%%%%%%%%%%%%%%%%%%%%%%%%%%%%%%%%%%%%%%%%%%%%
 In this section we present {some} examples to which our theorems {apply}. For each of them we give a reference for the detailed proofs that the relevant assumptions are satisfied. Note that Appendix~\ref{a:concrete} contains many examples not listed in Tables~\ref{ta:exWeyl} and~\ref{ta:ex}, and that the results from~\cite{CG19dyn} can be used to find additional examples. {With the exception of the final {three} rows of Table~\ref{ta:exWeyl} {with $W=M$}, all the estimates in Tables~\ref{ta:exWeyl} and~\ref{ta:ex} are new.}

In Table~\ref{ta:exWeyl}, we list examples where the assumptions of Theorem~\ref{t:laplaceWeyl} hold. The final two examples are due to Volovoy~\cite{Vo:90b}.

\smallskip
 \begin{table}[!h]
 \setlength{\extrarowheight}{4pt}%
 \newcommand{\nl}{\\[4pt]\hline}
 \noindent\begin{tabular}{|c|c|c|c|}
 \hline
$M$&$\subM$&$|E_\lambda|\lesssim$&\S\nl
 product manifolds& any&$\tfrac{\lambda^{n-1}}{\log \lambda}$&\ref{s:prod}\nl
perturbed spheres& in the non-periodic set&$\tfrac{\lambda^{n-1}}{\log \lambda}$& \ref{s:perturbEx}\nl
%triaxial ellipsoid& ??? &\nl
manifolds without conjugate points& any&$\tfrac{\lambda^{n-1}}{\log \lambda}$&\ref{s:noConj}\nl
non-Zoll convex analytic surfaces of revolution&any& $\tfrac{\lambda^{n-1}}{\log \lambda}$&\cite{Vo:90b}\nl
compact Lie group rank $>1$ with bi-invariant metric&any&$\tfrac{\lambda^{n-1}}{\log \lambda}$& \cite{Vo:90b}\nl
\multicolumn{3}{c}{}\\
\end{tabular}
\caption{This table lists examples with $\Ti$ non-periodic subsets { with $\Ti (R)= c \log R^{-1}$}. Theorem~\ref{t:laplaceWeyl} holds for all these examples. Here, $E_\lambda=\int_WE_\lambda(x)\dv_g$ with $E_\lambda(x)$ as in~\eqref{e:pwWeyl}.}
\label{ta:exWeyl}
\end{table}

 In Table~\ref{ta:ex} we list {some} examples for which Theorems \ref{t:laplaceoff} and \ref{t:laplaceOff for H} hold. 
 In each case there exists $t_0>0$ such that $(H_1, H_2)$ is a $(t_0, {\max(\Ti_1,\Ti_2)})$ non-looping pair. 
Note that we omit labeling points for which $\Ti_2=\inj(M)$ since being $\inj(M)$ non-recurrent is an empty statement. In these cases the gain in the pointwise Weyl law is $\sqrt{\log \lambda}$ instead of $\log \lambda$.\medskip

 %\FloatBarrier
 \begin{table}[!h]
 \setlength{\extrarowheight}{9pt}%
 \newcommand{\nl}{\\[9pt]\hline}
 \noindent\begin{tabular}{|c|c|c|c|c|c|c|}
 \hline
&$H_1$&$H_2$&$\Ti_1$&$\Ti_2$&$|E_\lambda|\lesssim$&\S\nl
\multicolumn{7}{|l|}{\bf{Manifolds with conjugate points}}\nl
 product manifolds&$x$ any point \tiny{(nL)} &$y$ any point \tiny{(nL)}&${\smalleq{\log R}}$&${\smalleq{\log R}}$&$\tfrac{\lambda^{n-1}}{{\log\lambda}}$&\ref{s:prod}\nl
spherical pendulum&$x$ not a pole \tiny{(nL)}&$x$ \tiny{(nL)}&${\smalleq{\log R}}$&${\smalleq{\log R}}$&$\tfrac{\lambda^{n-1}}{{\log\lambda}}$& \ref{s:spherePend}\nl
spherical pendulum&$x$ not a pole \tiny{(nL)}&$y$ a pole&${\smalleq{\log R}}$&$\smalleq{\inj M}$&$\tfrac{\lambda^{n-1}}{\sqrt{\log\lambda}}$&\ref{s:spherePend}\nl
 \setlength{\extrarowheight}{2pt}
\multirow{2}{*}{perturbed spheres}& \multirow{2}{*}{$\substack{\text{\normalsize{$x$ non-periodic,}}\\\text{\normalsize{ not a pole,}}}$ {\tiny(nL)}}&\multirow{2}{*}{$y$ a pole}&\multirow{2}{*}{${\smalleq{\log R}}$}& \multirow{2}{*}{$\smalleq{\inj M}$} &\multirow{2}{*}{$\tfrac{\lambda^{n-1}}{\sqrt{\log\lambda}}$}&\multirow{2}{*}{\ref{s:perturbEx}}
\\
&&&&&&
\\
\hline
 \setlength{\extrarowheight}{3pt}%
 \setlength{\extrarowheight}{2pt}
\multirow{2}{*}{perturbed spheres}& \multirow{2}{*}{$\substack{\text{\normalsize{$x$ non-periodic,}}\\\text{\normalsize{ not a pole,}}}$ {\tiny(nL)}}&\multirow{2}{*} {$x$ \tiny(nL)}&\multirow{2}{*}{${\smalleq{\log R}}$}& \multirow{2}{*}{${\smalleq{\log R}}$} &\multirow{2}{*}{$\tfrac{\lambda^{n-1}}{{\log\lambda}}$}&\multirow{2}{*}{\ref{s:perturbEx}}
\\
&&&&&&
\\
\hline
%
%triaxial ellipsoid& $x\notin $ middle ellipse \tiny{(nL)}&$y$ any point \tiny{(nR)}&${\smalleq{\log R}}$&${\smalleq{\log R}}$&\nl
%any &$x$ any point& $d(y,\mc{C}_{x}^{n-1,R,{\smalleq{\log R}}})>R$&${\smalleq{\log R}}$&${\smalleq{\log R}}$&\nl
\multicolumn{7}{|l|}{ \setlength{\extrarowheight}{3pt}{\bf{Manifolds without conjugate points}}}\nl
 \setlength{\extrarowheight}{2pt}%
\multirow{2}{*}{any}&\multirow{2}{*}{$\smalleq{k_1>1}\,$\tiny{(nL)}}&\multirow{2}{*}{$\substack{\smalleq{k_2>1}\\\smalleq{k_1+k_2>n+1}}\,$\tiny{(nL)}}&\multirow{2}{*}{${\smalleq{\log R}}$}&\multirow{2}{*}{${\smalleq{\log R}}$}
&\multirow{2}{*}{$\tfrac{\lambda^{\frac{k_1+k_2}{2}-1}}{\log\lambda}$}&\multirow{2}{*}{\ref{s:noConj}}\\
&&&&&&
\\
\hline
 \setlength{\extrarowheight}{3pt}%
any&geodesic sphere \tiny{(nL)}&geodesic sphere \tiny{(nL)}&${\smalleq{\log R}}$&${\smalleq{\log R}}$&$\tfrac{1}{{\log\lambda}}$&\ref{s:flowInv}\nl
%\multicolumn{6}{|l|}{{\bf{Manifolds with Anosov geodesic flow}}}\nl
%Anosov surface&$x$ any point& any curve&${\smalleq{\log R}}$&${\smalleq{\log R}}$&\nl
Anosov&horosphere$^+$ \tiny{(nRvc)}& horosphere$^-$ \tiny{(nRvc)}&${\smalleq{\log R}}$&${\smalleq{\log R}}$&$\tfrac{1}{{\log\lambda}}$&\ref{s:anosov}\nl
Anosov, $K_g\leq0$&totally geodesic \tiny{(nL)}& totally geodesic \tiny{(nL)}&${\smalleq{\log R}}$&${\smalleq{\log R}}$&$\tfrac{\lambda^{\frac{k_1+k_2}{2}-1}}{{\log\lambda}}$&\ref{s:anosov}\nl
Anosov, $K_g\leq 0$&totally geodesic \tiny{(nL)}& horosphere \tiny{(nRvc)}&${\smalleq{\log R}}$&$\log R$&$\tfrac{\lambda^{\frac{k_1-1}{2}}}{{\log\lambda}}$&\ref{s:anosov}\nl
\multicolumn{7}{c}{}\\
\end{tabular}
\caption{The table lists examples where Theorems~\ref{t:laplaceoff} and~\ref{t:laplaceOff for H} hold. We write {\tiny(nL)} when $H_i$ is $\Ti_i$ non-looping and {\tiny(nRvc)} when $H_i$ is $\Ti_i$ non-recurrent via coverings. Horosphere$^\pm$ denotes stable/unstable horospheres, and $K_g$ the sectional curvature. A manifold is called Anosov if it has Anosov geodesic flow (see \S\ref{s:anosov} for a definition). The label $E_\lambda$ represents the integrated error term $\int_{H_1}\int_{H_2} E_\lambda(x,y) \ds{H_1}\ds{H_2}$.}
\label{ta:ex}
\end{table}

\subsection{Further improvements}
Many experts believe that, for a Baire generic Riemannian metric on a smooth compact manifold, there is $\delta>0$ such that  $E(\lambda)=O(\lambda^{n-1-\delta})$. Presently, this type of improved remainder is only available when the geodesic flow has special structure e.g. the flat torus, non-Zoll convex analytic surfaces of revolution, or compact Lie groups of rank $>1$ with bi-invariant metric~\cite{Vo:90b}. Specifically, the geodesic flow must expand only polynomially in time, {$\|d\varphi_t\|_{L^\infty(T\SM)}\leq C\langle t\rangle^N$ for some $N>0$.} Typically, geodesics will instead expand exponentially in some places and, because of this, Egorov's theorem generally only holds to logarithmic times. In fact, the only ingredient in our proof which restricts us to logarithmic improvements is Egorov's theorem. Under {the assumption of polynomial expansion} one can prove an Egorov theorem to polynomial times and hence obtain polynomially improved remainders using our methods. We do not pursue this here since the present article is intended to apply on a general manifold and the polynomial times involved in such an Egorov theorem are not explicit. We instead plan to address the integrable case specifically in a future article.

%%%%%%%%%%%%%%%%%%%%%%%%%%%%%%%%%%%%%%%%%%%%%%%%%%%%%%%%%%%%%%%%%%%%%%%%%%%%%%%
%%%%%%%%%%%%%%%%%%%%%%%%%%%%%%%%%%%%%%%%%%%%%%%%%%%%%%%%%%%%%%%%%%%%%%%%%%%%%%%
\subsection{Weyl laws for general operators}
%%%%%%%%%%%%%%%%%%%%%%%%%%%%%%%%%%%%%%%%%%%%%%%%%%%%%%%%%%%%%%%%%%%%%%%%%%%%%%%
%%%%%%%%%%%%%%%%%%%%%%%%%%%%%%%%%%%%%%%%%%%%%%%%%%%%%%%%%%%%%%%%%%%%%%%%%%%%%%%
Let $P(h) \in \Psi^m(M)$ be a self-adjoint, semiclassical pseudodifferential operator with principal symbol $p$, that is positive and classically elliptic in the sense that there is $C>0$ such that 
\begin{equation}
p(x,\xi)\geq \tfrac{1}{C}|\xi|^m,\qquad |\xi|\geq C.\label{e:positivity}
\end{equation}
%Without loss of generality, we will assume that $p\geq 0$ on $|\xi|\geq C$ (otherwise taking instead $-P$).
Let $\{E_j(h)\}_j$ be the eigenvalues of $P$ repeated with multiplicity.
{For $s\in \re$ we work with $\Pi_h(s):=\1_{(-\infty,s]}(P(h)),$ which is the orthogonal projection operator}
$$
\Pi_h(s): L^2(M) \to \bigoplus_{ E_j(h)\leq s} \ker(P(h)-E_j(h)).
$$
For $x,y \in M$ we write $\Pi_h(s; x,y)$ for its kernel 
\begin{equation}\label{e:efxn decomp}
\Pi_h(s;x,y):=\sum_{ E_j(h)\leq s}\phi\sub{E_j(h)}(x)\overline{\phi\sub{E_j(h)}}(y),
\end{equation}
where $\{\phi\sub{E_j(h)}\}_j$ is an orthonormal basis for $L^2(M)$ with
$
P(h)\phi\sub{E_j(h)}=E_j(h)\phi\sub{E_j(h)}.
$

Let
$
\varphi_t:T^*M \to T^*M
$
denote the Hamiltonian flow for $p$ at time $t$.
We recall the \emph{maximal expansion rate} for the flow and the \emph{Ehrenfest time} at frequency $h^{-1}$ respectively:
\begin{equation}\label{e:Lmax}
\Lambda_{\max}:=\limsup_{|t|\to \infty}\frac{1}{|t|}{\log} \sup_{{\{p\in[a-\e,b+\e]\}}}\|d\varphi_t(x,\xi)\|, \qquad 
T_e(h):=\frac{\log h^{-1}}{2\Lambda_{\max}}.
\end{equation}
Note that $\Lambda_{\max}\in[0,\infty)$ and if $\Lambda_{\max}=0$, we may replace it by an arbitrarily small constant. 

\begin{definition}\label{d:non period gral}
Let $a, b \in \re$ with $a\leq b$. Let $t_0>0$ and $\Ti$ be a {\resfun}. 
{A set} $U\subset T^*M$ is said to be \emph{$ \Ti$ 
non-periodic for $p$ in the window $[a,b]$} provided that for all $E\in [a,b]$ Definition \ref{d:non periodic} {holds with $\varphi_t$ being the Hamiltonian flow for $p$, and with $S^*M$ replaced by $p^{-1}(E)$}.
\end{definition}
The following is our most general version of the Weyl Law. {We write $\pi\sub{M}:T^*M \to M$ for the natural projection {and $H_p$ for the Hamiltonian vector field for $p$}.} %, and we denote the subprincipal symbol of $P$ by $\sigma\sub{\!-1}(P)$.}

\begin{theorem}
\label{t:genWeyl}
Let $0< \delta<\frac{1}{2}$, $\ell\in \mathbb{R}$, and $\mc{V}\subset \Psi^\ell(M)$ a bounded subset, $U\subset T^*M$ open, $t_0>0$, $C\sub{U}>0$, and $a,b\in\mathbb{R}$ with $a\leq b$.
Suppose $d\pi\sub{M}H_p\neq 0$ on $p^{-1}([a,b])\cap \overline{U}$. Then, there is $C\sub{0}>0$ such that the following holds. Let {$K>0$}, $A\in \mc{V}$ with $\WFh(A)\subset U$, $\Lambda>\Lambda_{\max}$, $\Ti$ be a sub-logarithmic {\resfun} with $\Lambda \Omega(\Ti)<1-2\delta$, and suppose $U$ is $\Ti$ non-periodic in the window $[a,b]$ with
\begin{equation}\label{e:thePoachedChicken}
\limsup_{R\to 0}\sup_{t\in[a,b]} \Ti(R)\mu\sub{p^{-1}(t)}(B(\partial U, R))\leq C\sub{U}.
\end{equation}
 Then, there is $h_0>0$ such that for all $0<h<h_0$, and $E\in[a,b+Kh]$
\begin{equation}
\label{e:genWeyl}
\begin{aligned}
&\bigg|\sum_{-\infty< E_j(h)\leq E}\langle A\phi\sub{E_j(h)},\phi\sub{E_j(h)}\rangle-\tr \big(A\,\rho\sub{t_0/h}*\Pi_h(E) \big)\bigg| %A\phi\sub{E_j(h)},\phi\sub{E_j(h)}\rangle-\int_{a\leq p\leq E}\sigma(A)dxd\xi\\
%&\qquad-h\Big(\int_{a\leq p\leq E}\sigma\sub{-1}(A)dxd\xi + \int \sigma(A)\sigma\sub{-1}(P)d\mu_{p^{-1}(a)}-\int \sigma(A)\sigma\sub{-1}(P)d\mu_{p^{-1}(E)}\Big)
%\Bigg |\\
\leq C\sub{0}\,h^{1-n}\big/\,\Ti(h).
\end{aligned}
\end{equation}
\end{theorem}

Since the second term in~\eqref{e:genWeyl} involves only short time propagation for the Schr\"odinger group $e^{itP/h}$, its asymptotic expansion in powers of $h$ can in principle be obtained. This calculation is routine, but long, so we do not include it here. For the details when $P=-h^2\Delta_g$, we refer the reader to~\cite[Proposition 2.1]{DuGu:75}. In addition, if $U\subset T^*M$ has smooth boundary which intersects $p^{-1}(E)$ transversally for $E\in[a,b]$, then~\eqref{e:thePoachedChicken} holds. Although the statement of Theorem~\ref{t:genWeyl} is cumbersome when $U$ with rough boundary is allowed, it is natural to consider dynamical assumptions on this type of set. Indeed, many dynamical systems exhibit the so-called `chaotic sea' with `integrable islands' behavior where the dynamics are aperiodic in the sea; a set which typically has very rough boundary.
%\blue{Note that when $P(h)=-h^2\Delta_g$, then $\sigma(P)(x,\xi)=|\xi|^2_{g(x)}$ and $\sigma\sub{-1}(P)=0$. Thus, this result improves on....\cs}

Next, we consider generalized Kuznecov{~\cite{Ku:80}} type sums of the form 
\begin{equation*}
\begin{aligned}
\HAs(s)
&:=\int_{H_1}\int_{H_2} A_1\Pi_h(s) A_2^*\, (x,y)\,\ds{H_1}(x)\ds{H_2}(y),
%&=\int_{H_1}\int_{H_2} A_1(x,hD_x)A_2(y,hD_y) \Pi_h(s;x,y)\ds{H_1}(x)\ds{H_2}(y),
\end{aligned}
\end{equation*}
where $A_1,A_2\in \Psi^\infty(M)$ {and $H_1, H_2 \subset M$ are two submanifolds of $M$}.
%In the special case where $H=H_1=H_2$ and $A=A_1=A_2$, we will write 
%$$
%\HA(s):=\Pi_{_{H_1,H_1}}^{^{A_1,A_1}}(s).
%$$
%We note that both $\HAs(s)$ and $\HA(s)$ depend on $h$, but we omit this dependence to ease notation.

Let $H\subset M$ be a smooth submanifold. For $a,b \in \re$, $a \leq b$, define
 \begin{equation}\label{e:SigmaAB}
\Sab^H:=p^{-1}([a,b])\cap N^*\!H.
 \end{equation}

\begin{definition}
We say a submanifold $H\subset M$ of codimension $k$ is \emph{conormally transverse for $p$ in the window $[a,b]$} if given $f_1,\dots, f_{k}\in C_c^\infty(M;\R)$ locally defining $H$, i.e. with
$H= \bigcap_{i=1}^k\{f_i=0\}$ and $\{df_i\}\text{ linearly independent on }H,$ 
we have
\begin{equation}
\label{e:transverse} \Sab^H \subset \bigcup_{i=1}^k\{ H_pf_i \neq 0\},
\end{equation}
{Here, we interpret $f_i$ as a function on the cotangent bundle by pulling it back through the canonical projection map.}
\end{definition}

\begin{remark}
If $P(h)=-h^2\Delta_g$, then $p(x,\xi)=|\xi|^2_{g(x)}$. Working with $a=b=1$, we have $\Sab^H= S\!N^*H$. In this setup every submanifold $H \subset M$ is conormally transverse for $p$.
\end{remark}

\begin{definition}\label{d:non loop gral}
Let $H_1, H_2 \subset M$ be two smooth submanifolds. Let $a, b \in \re$ with $a\leq b$. Let $t_0>0$, $\Ti$ a {\resfun}, {and $\Cnl>0$}. We say
$(H_1,H_2)$ is a \emph{$(t_0, \Ti)$ 
non-looping pair in the window $[a,b]$ {with constant $\Cnl$}} provided that Definition \ref{d:non looping through} holds for all $E\in [a,b]$ 
 with $\varphi_t$ being the Hamiltonian flow for $p$ and with $\mathcal{L}_{x,y}^{{R}}$ changed to
 $$\mathcal{L}\sub{H_1,H_2}^{{R},{E}}(t_0, T):=\bigg\{ \rho \in \SE^{H_1}:\; \bigcup_{t_0\leq |t|\leq T}\varphi_t({B(\rho,R)}) \,\cap {B\big(\SE^{H_2},R\big)}\neq \emptyset\bigg\},$$
 {and with ${S_x^*M}$ and ${S_y^*M}$ replaced with $\SE^{H_1} $ and $\SE^{H_2}$ respectively.}
 We say \emph{ $H$ is $(t_0,\Ti)$ non-looping} if $(H,H)$ is a $(t_0,\Ti)$ non-looping pair.
\end{definition}

\begin{definition}\label{d:non rec gral}
Let $H \subset M$ be a smooth submanifold. Let $a, b \in \re$ with $a\leq b$. Let $t_0>0$, {$R_0>0$}, {$0<\Cnr<1$}, and let $\Ti$ be a {\resfun}. 
 $H$ is said to be \emph{$\Ti$ non-recurrent in the window $[a,b]$ {with constants ${(R_0,\Cnr)}$}} provided Definition \ref{d:non rec} holds for any $E\in[a,b]$ with {$ S^*_xM$ replaced by $\SE^H$ and where $\varphi_t$ is the Hamiltonian flow for $p$}.
\end{definition}

To state our main estimate for Kuznecov sums, let $\rho\in \mc{S}(\mathbb{R})$ with $\hat{\rho}(0)\equiv 1$ on $[-1,1]$ and $\supp \hat{\rho}\subset [-2,2]$. 
For $T>0$ we define
\begin{equation}\label{e:rho def}
 \rho\sub{h,T}(t):= \tfrac{T}{h} \, \rho \Big(\tfrac{T}{h} \, t \Big). 
\end{equation}
We then introduce the remainder 
\begin{equation}\label{e:gralremainder}
\RAs(T,h; {\energy})=\HAs({\energy})-\rho\sub{h,T}*\HAs({\energy}).
\end{equation}
%%%%%%%%%%%%%%%%%%%%%%%%%%%%%%%%%%%%%%%%%%%%%%%%%%%%%%%%%%%%%%%%%%%%%%%%%%%%%%%
\begin{theorem}
\label{t:off}
Let $P(h) \in \Psi^m(M)$ be a self-adjoint semiclassical pseudodifferential operator with classically elliptic symbol $p$. {Let $\ti$ be a {\resfun} and ${\e>0}$.}
{For $j=1,2,$ let $H_j\subset M$ be submanifolds with co-dimension $k_j$. Let $a, b\in \re$ such that $H_j$ is conormally transverse for $p$ in the window $[a,b]$ for $j=1,2$. Let ${R_0}>0,$ $t_0>0$, and for $j=1,2$, let $\Ti_j$ be sub-logarithmic \resfuns \,{and $\Ti_{\max}=\max(\Ti_1,\Ti_2)$}. } Suppose  $H_j$ is ${(\ti,\Ti_j)}$ non-recurrent in the window $[a,b]$ with {constant} ${R_0}$ {for each $j=1,2$}, and $(H_1,H_2)$ is a $(t_0, {\Ti_{\max}})$ non-looping pair in the window $[a,b]$ {with constant $\Cnl$}.
Then, for all $A_1, A_2\in \Psi^\infty(M)$, there exist $h_0>0$ and $C\sub{0}>0$ such that for all $0<h\leq h_0$, $K>0$, and ${\energy}\in [a-Kh,b+Kh]$
$$
\Big|\RAs(t_0{+\e},h; {\energy})\Big|\leq C\sub{0}\, h^{1-\frac{k_1+k_2}{2}}\Big/\!\sqrt{\Ti_
1(h)\Ti_2(h)}.
$$
\end{theorem}
%%%%%%%%%%%%%%%%%%%%%%%%%%%%%%%%%%%%%%%%%%%%%%%%%%%%%%%%%%%%%%%%%%%%%%%%%%%%%%%

{\begin{remark}
We omit the precise dependence of the constant $C\sub{0}$ on various parameters in Theorem~\ref{t:off}. Instead, we refer the reader to our main theorem on averages, Theorem~\ref{t:MAIN}, where we have introduced notation to handle uniformity in families of submanifolds $H_1$ and $H_2$. \end{remark}}

\subsection{Outline of the paper and ideas from the proof}\label{s:outline}
In \S\ref{s:mainDyn} we introduce the notion of good coverings by tubes and various assumptions on these coverings which allow us to {adapt} the results of~\cite{CG19a} {to our setup}. We also state our main averages theorem in its full generality (Theorem~\ref{t:MAIN}). \S\ref{s:meas-then-dyn} studies how the dynamical assumptions in the introduction relate to the assumptions on coverings by tubes from~\S\ref{s:mainDyn}. In \S\ref{s:basic} we {adapt the crucial estimates coming from {the} geodesic beam techniques~\cite{CG19a} so that they can be applied to the study of Weyl remainders.} Next, in \S\ref{s:lip}, we estimate the scale (in the energy) at which averages of the spectral projector behave like Lipschitz functions {in the spectral parameter}. With this in hand, we are able to approximate $\Pi_h$ using $\rho\sub{h,T(h)}*\Pi_h$ {with $T(h)=\sqrt{\Ti_1(h)\Ti_2(h)}$}.
%\marginpar{\ya{note that thm 6 is stated with $\Ti_i(R(h))$ while thm 5 with $\Ti_i(h)$.}}
Finally, \S\ref{s:smoothed} shows that the $\rho\sub{h,T(h)}*\Pi_h$ approximation is close to $\rho_{h,t_0}*\Pi_h$, finishing the proof of our main theorem on averages. \S\ref{s:weyl} contains the proof of our theorems on the Weyl remainder. This section follows the same strategy as that for averages: an estimate for the Lipschitz scale of the trace of the spectral projector, followed by relating $\rho\sub{h,T(h)}*\Pi_h$ to $\rho\sub{h,t_0}*\Pi_h$. {In Appendix~\ref{s:index} we present an index of notation and in} Appendix~\ref{a:concrete} we give examples including those from Table~\ref{ta:ex} to which our theorems can be applied. 

The main idea of this article is to view the kernel of the spectral projector $\1_{[t-s,t]}(P)$ as a quasimode for $P$. This allows us to use {the} geodesic beam techniques from~\cite{CG19a} to control {the} energy scale at which the projector behaves like a Lipschitz function and hence to {estimate} the error when the projector is smoothed at very small scales. This idea is used a second time when controlling $(\rho_{h,T(h)}-\rho_{h,t_0}) *\Pi_h$ to {estimate} the contribution from small volumes of {the possibly looping} tubes. A simple argument using Egorov's theorem controls the remaining non-looping tubes. The crucial insight used to handle the Weyl law is to view the kernel of the spectral projector as a distribution on $M\times M$, where it is a quasimode for ${\bf{P}}:=P\otimes 1$, {and to study the Weyl Law} via integration of the kernel over the diagonal. {By doing this, we are able to reduce the problem to bounding an average of a quasimode over a submanifold, a setting in which geodesic beam techniques} apply.

Note that Theorems~\ref{t:laplaceWeyl} and~\ref{t:genWeyl} are proved in \S\ref{s:laplaceWeyl} and \S\ref{s:laplaceGen} respectively. {Theorem~\ref{t:product} is a corollary of Theorem~\ref{t:laplaceWeyl}; the necessary dynamical properties are proved in Appendix~\ref{s:prod}.}
Theorems~\ref{t:laplaceDiag},~\ref{t:laplaceoff},~\ref{t:laplaceOff for H}, {and~\ref{t:off}} follow from an application of Theorem \ref{t:MAIN2} (See \S\ref{s:itAllFollows} {for Theorems~\ref{t:laplaceDiag},~\ref{t:laplaceoff}, and~\ref{t:laplaceOff for H}. Theorem~\ref{t:off} is a direct corollary of Theorem~\ref{t:MAIN2}.}). {The fact that Theorem~\ref{t:MAIN2} follows from Theorem~\ref{t:MAIN} is proved in \S~\ref{t:MAIN2} and Theorem~\ref{t:MAIN} is proved in \S\ref{s:imTheMAIN}.}

\begin{remark}[Resolution functions]\label{r:rate function}
There are several reasons why we state our theorems in terms of a general resolution function. First, it is necessary to allow $\Ti(R)$ to grow arbitrarily slowly as $R\to 0$ to recover the $o(1)$ results of~\cite{Saf88,SoggeZelditch,DuGu:75} {(see Remark \ref{r:little-oh}).} Second, while it may appear from Tables~\ref{ta:exWeyl} and~\ref{ta:ex}, that $\Ti(R)$ is always either $c\log R^{-1}$ or the trivial case of $\inj(M)$, this is not always true. In fact, one can check that many integrable examples are non-looping or non-periodic for $\Ti(R)\gg \log R^{-1}$.
At the moment, the authors are not aware of concrete examples with $\Ti(R)\ll \log R$. However, it is likely that for any sub-logarithmic {\resfun} $\Ti$, with $\Ti(R)\to \infty$ as $R\to 0^+$, a modification of the construction from~\cite{BuPa:95} yields a metric on the sphere for which there is a point $x$ such that $x$ is not $(t_0,\Ti)$ non-looping for any $t_0>0$, but there is  a {\resfun}  $\Ti_1$ with $\Ti_1(R){\longrightarrow} \infty$ as $R\to 0^+$ and $t_0>0$ such that $x$ is $(t_0,\Ti_1)$ non-looping. Also, note that our non-periodic, non-looping, and non-recurrent conditions are all monotonic in $\Ti$ in the sense that if $\Ti_1(R)\leq \Ti_2(R)$, and one of these conditions hold with the {\resfun} $\Ti_2$, then it also holds with $\Ti_1$.
\end{remark}

\noindent {\sc Acknowledgements.}
The authors would like to thank {Dmitry} Jakobson, Iosif Polterovich, John Toth, Dmitri Vassiliev and Steve Zelditch for helpful comments {on the existing literature} and Maciej Zworski for {suggestions on how to improve} the exposition and presentation, and Leonid Parnovski for comments on a previous draft. The authors are grateful to the National Science Foundation for partial support under grants DMS-1900434 and DMS-1502661 (JG) and DMS-1900519 (YC). {Y.C. is grateful to the Alfred P. Sloan Foundation. } 
%%%%%%%%%%%%%%%%%%%%%%%%%%%%%%%%%%%%%%%%%%%%%%%%%%%%%%%%%%%%%%%%%%%%%%%%%%%%%%%
%%%%%%%%%%%%%%%%%%%%%%%%%%%%%%%%%%%%%%%%%%%%%%%%%%%%%%%%%%%%%%%%%%%%%%%%%%%%%%%
%%%%%%%%%%%%%%%%%%%%%%%%%%%%%%%%%%%%%%%%%%%%%%%%%%%%%%%%%%%%%%%%%%%%%%%%%%%%%%%
%%%%%%%%%%%%%%%%%%%%%%%%%%%%%%%%%%%%%%%%%%%%%%%%%%%%%%%%%%%%%%%%%%%%%%%%%%%%%%%
\section{Results with dynamical assumptions via coverings by tubes}
\label{s:mainDyn}
%%%%%%%%%%%%%%%%%%%%%%%%%%%%%%%%%%%%%%%%%%%%%%%%%%%%%%%%%%%%%%%%%%%%%%%%%%%%%%%
%%%%%%%%%%%%%%%%%%%%%%%%%%%%%%%%%%%%%%%%%%%%%%%%%%%%%%%%%%%%%%%%%%%%%%%%%%%%%%%
%%%%%%%%%%%%%%%%%%%%%%%%%%%%%%%%%%%%%%%%%%%%%%%%%%%%%%%%%%%%%%%%%%%%%%%%%%%%%%%
%%%%%%%%%%%%%%%%%%%%%%%%%%%%%%%%%%%%%%%%%%%%%%%%%%%%%%%%%%%%%%%%%%%%%%%%%%%%%%%
We divide this section in four parts.
In Section \ref{s:dyn} we introduce the analogues of Definitions \ref{d:non loop gral} and \ref{d:non rec gral} via the use of coverings by bicharacteristic tubes. {Microlocalization to these tubes will eventually be used to generate bicharacteristic beams.} In Section~\ref{s:uniform} we introduce the uniformity assumptions that allow us to obtain uniform control of the constants in all our results. In Section \ref{s:gral} we state the most general version of our results, using the definitions via coverings by tubes, and the uniformity assumptions. 

%%%%%%%%%%%%%%%%%%%%%%%%%%%%%%%%%%%%%%%%%%%%%%%%%%%%%%%%%%%%%%%%%%%%%%%%%%%%%%%
%%%%%%%%%%%%%%%%%%%%%%%%%%%%%%%%%%%%%%%%%%%%%%%%%%%%%%%%%%%%%%%%%%%%%%%%%%%%%%%
\subsection{Dynamical assumptions via coverings by tubes}\label{s:dyn}
%%%%%%%%%%%%%%%%%%%%%%%%%%%%%%%%%%%%%%%%%%%%%%%%%%%%%%%%%%%%%%%%%%%%%%%%%%%%%%%
%%%%%%%%%%%%%%%%%%%%%%%%%%%%%%%%%%%%%%%%%%%%%%%%%%%%%%%%%%%%%%%%%%%%%%%%%%%%%%%
Let $H\subset M$ be a smooth submanifold that is conormally transverse for $p$ in the window $[a,b]$.
Let $\hyp\subset T^*M$ with 
\begin{equation}\label{e:mc H}
\Sab^H \subset \hyp
\end{equation}
be a hypersurface that is transverse to the flow, and
$
\varphi_t
$
continue to denote the Hamiltonian flow for $p$ at time $t$. 
Given $A \subset \Sab^{H}$, $\tau>0$, and $r>0$, we define
\begin{equation}\label{e:tube}
\Lambda_{_{\!A}}^\tau(r):=\bigcup_{|t|\leq \tau+r}\varphi_t\big({B\sub{\hyp}}(A,r)\big).%,\qquad A_{_r}:=\{\rho\in \hyp :d(\rho,A)<r\},}
\end{equation}

Let {$\tau\sub{\inj_H}>0$ be small enough so that the map 
\begin{equation}\label{e:tau_inj}
(-\tau\sub{\inj_H},\tau\sub{\inj_H})\times \hyp \to T^*M, \qquad (t,q) \mapsto \varphi_t(q),
\end{equation} 
is injective.} Given $r>0$, $0<\tau<\tau\sub{\inj_H}$, and a collection of points $\{\rho_j\}_{j\in \J(r)}$, we 
will work with the tubes $$\mc{T}_j=\mc{T}_j(r):=\Lambda_{\rho_j}^\tau(r).$$
A \emph{$(\tau, r)$-cover for $A\subset T^*M$} is a collection of tubes 
$\{\mc{T}_j(r)\}_{j\in \J(r)}$ for which
$$\Lambda_A^\tau(\tfrac{1}{2}r) \subset \bigcup_{j\in \J(r)}\mc{T}_j(r),\qquad \text{and}\qquad\, \T_j(r)\cap \Lambda^\tau_A(\tfrac{1}{2}r)\neq \emptyset,\quad \text{for all }{j\in \J(r)}.$$
{Let $\mathfrak{D}>0$.} We say a $(\tau,r)$-cover is a $(\mathfrak{D},\tau,r)$-\emph{good cover}, if there is a splitting $\J(r)=\sqcup_{i=1}^{\mathfrak D}\J_i(r)$ such that for all $1\leq i\leq \mathfrak{D}$ and $k{\neq}\ell\in \J_i(r)$, 
\begin{equation}\label{e:good cover}
\mc{T}_k(3r)\cap \mc{T}_\ell(3r)=\emptyset.
\end{equation}
For $E\in \re$ and $r>0$, we adopt the notation 
\begin{equation}\label{e:J_E}
 \J\sub{E}(r):=\Big\{ j\in \J(r):\; \mc{T}_j(r)\, \cap \hyp\,\cap B(\SE^{H},r)\neq \emptyset \Big\}. %\bigcap\, p^{-1}\big(E-r, E+r\big)\neq \emptyset\Big\}.
\end{equation}

We are now ready to introduce the definitions via coverings of our dynamical assumptions. First, for $0<t_0<T_0$, we say $A\subset T^*\!M$ is \emph{$[t_0,T_0]$ non-self looping} if 
\begin{equation}\label{e:nonsl}
\bigcup_{t=t_0}^{T_0}\varphi_t(A)\cap A=\emptyset\qquad \text{ or }\qquad \bigcup_{t=-T_0}^{-t_0}\varphi_t(A)\cap A=\emptyset.
\end{equation}

 \begin{definition}[non looping pairs via coverings]\label{d:non loop cov}
Let $t_0>0$, ${\tau_0}>0$, ${\mathfrak{D}}>0$, and $\Ti$ be a {\resfun}.
 Let $H_1, H_2$ be two submanifolds and $U_1\subset {N^*H_1}$, $U_2\subset N^*H_2$. We say \emph{$(U_1,U_2)$ is a $(t_0, \Ti)$ non-looping pair in the window $[a,b]$ via {${\tau_0}$-coverings} {with constant ${\Cnl}$}} provided {for all $0<\tau<\tau_0$ there exists $r_0>0$ such that }for $0<r<{r_0}$, any two $({\mathfrak{D}},\tau, r)$-{good} covers of $U_1\cap \Sab^{H_1}$ and $U_2\cap \Sab^{H_2}$,
 $\{\mc{T}_j^1(r)\}_{j\in \J^1(r)}$ and $\{\mc{T}_j^2(r)\}_{j\in \J^2(r)}$ respectively, and every $E\in [a,b]$, there is splittings of indices 
 $$\begin{gathered}\J^1\sub{E}(r)={\BE^1(r) \cup \GE^1(r)},\qquad \J\sub{E}^2(r)=\BE^2(r) \cup \GE^2(r),
 \end{gathered}$$
satisfying
\begin{enumerate}
 \item for each ${i,k\in\{1,2\}}$, {$i\neq k$} every $\ell \in \GE^i(r)$, 
 $$\Bigg(\bigcup_{{t_0+\tau\leq|t|\leq \Ti(r)-\tau}} \varphi_t \Big(\mc{T}_\ell^i(r)\Big) \Bigg) \bigcap \Bigg( \bigcup_{j \in \J^k\sub{E}(r)} \mc{T}_j^k(r) \Bigg)=\emptyset,$$
 \item 
 $r^{2({n-1})}|\BE^1(r)||\BE^2(r)|\Ti(r)^2\leq {\mathfrak{D}^2}\Cnl.$
\end{enumerate}\smallskip
We will say \emph{$(H_1,H_2)$ is a $(t_0,\Ti)$ non-looping pair in the window $[a,b]$ via {$\tau$-coverings}} 
if $(N^*H_1,N^*H_2)$ is. We will also say \emph{$H$ is $(t_0, \Ti)$ non-looping in the window $[a,b]$ via $\tau$ coverings} whenever $(H,H)$ is a non-looping pair.
\end{definition}

In Section~\ref{s:meas-then-dyn}, we prove that
 non looping in the sense of Definition \ref{d:non loop gral} is equivalent to non looping by coverings in the sense of Definition \ref{d:non loop cov}.

\begin{definition}[non-recurrence {via} coverings]\label{d:non rec cov}
Let ${\tau_0}>0$, ${\mathfrak{D}}>0$, and $\Ti$ be a {\resfun}.
We say $H$ is \emph{$\Ti$ non-recurrent in the window $[a,b]$ via {${\tau_0}$-coverings} {with constant $\Cnr$}} provided {for all $0<\tau<\tau_0$} {there exists $r_0>0$ such that } for $0<r<{r_0}$, every $({\mathfrak{D}},\tau, r)$-{good} cover of $\Sab^{H}$, $\{\mc{T}_j(r)\}_{j\in \J(r)}$, and $E\in [a,b]$, there exists a finite collection of sets of indices $\{\mc{G}\sub{E,\ell}(r)\}_{\ell \in {\mc{L}\sub{E}(r)}}$ with 
$
\J\sub{E}(r)=\bigcup_{\ell \in {\mc{L}\sub{E}(r)}}\mc{G}\sub{E,\ell}(r),
$
 and so that for every $\ell \in {\mc{L}\sub{E}(r)}$ there exist functions $t_\ell(r)>0$ and ${T_\ell(r)}>0$, {with $0\leq t_\ell(r)\leq T_\ell(r)\leq \Ti(r)$}, so that 
\begin{enumerate}
\item
$\bigcup_{j\in \mc{G}\sub{E,\ell}(r)}\mc{T}_j(r)\;\;\text{ is }\;\;[t_\ell(r),T_{\ell}(r)]\text{ non-self looping},$
\smallskip
\item
$r^{\frac{n-1}{2}}\sum_{\ell\in {\L\sub{E}(r)}} {\big(|\mc{G}\sub{E,\ell} (r)|t_\ell(r){T_\ell(r)}^{-1}\big)^{\frac{1}{2}}}
\leq {\mathfrak{D}^{\frac{1}{2}}}\Cnr\,\Ti(r)^{-\frac{1}{2}}.$ \smallskip
%\item
%$r^{\frac{n-1}{2}}\sum_{\ell\in\ya{\L\sub{E}(r)}} {\big(|\G_\ell (r)|t_\ell(r){T_\ell(r)}\big)^{\frac{1}{2}}}
%\leq C\,T^{\frac{1}{2}}.$
\end{enumerate}
\end{definition}
%\blue{\cs decide if we keep the third one, and if $t_\ell \geq t_0$ is all the dependence we will ask on $t_0$\cs}

In Lemma \ref{p:non rec implies cov} below we prove that non recurrence in the sense of Definition \ref{d:non rec gral} implies non recurrence by coverings in the sense of Definition \ref{d:non rec cov}. At the moment, we are unable to determine whether these two definitions are equivalent.

%%%%%%%%%%%%%%%%%%%%%%%%%%%%%%%%%%%%%%%%%%%%%%%%%%%%%%%%%%%%%%%%%%%%%%%%%%%%%%%
%%%%%%%%%%%%%%%%%%%%%%%%%%%%%%%%%%%%%%%%%%%%%%%%%%%%%%%%%%%%%%%%%%%%%%%%%%%%%%%
\subsection{Uniformity assumptions}\label{s:uniform}
%%%%%%%%%%%%%%%%%%%%%%%%%%%%%%%%%%%%%%%%%%%%%%%%%%%%%%%%%%%%%%%%%%%%%%%%%%%%%%%
%%%%%%%%%%%%%%%%%%%%%%%%%%%%%%%%%%%%%%%%%%%%%%%%%%%%%%%%%%%%%%%%%%%%%%%%%%%%%%%
Let $H\subset M$ be a smooth submanifold. In practice, we prove estimates on $\{\tilde{H}_h\}_h$, where $\{\tilde{H}_h\}_h$ is a family of submanifolds such that
\begin{equation}
\label{e:conormalClose}
{\sup\Big\{ d\big(\rho ,\Sab^{\tilde H_h}}\big)\,\,\big|\,\,{\rho \in \Sab^H}\Big\}\leq {R(h)} \qquad\qquad h>0,
\end{equation}
 {where $R(h)>0$}
and for every multi-index $\alpha$ there is 
{$\KR_{_{\alpha}}>0$} such that for all $h>0$
\begin{equation}
\label{e:curvature}
|\partial_x^\alpha {\bf R}_{_{\tilde H_h}}|+|\partial_x^\alpha {\bf \Pi}_{_{\tilde H_h}}|\leq \KR_{_{\alpha}}.
\end{equation}
Here ${\bf R}_{_{\tilde H_h}}$ and ${\bf \Pi}_{_{\tilde H_h}}$ denote the sectional curvature and the second fundamental form of $\tilde H_h$. 
Without loss of generality, we will assume $\hyp$ is chosen so that there exist {$N>0$,} $C=C(p,a,b,\{\KR_\alpha\}_{|\alpha|\leq N})>0,$ and $r_0>0$ such that
for all {$E\in [a,b]$,} $A\subset \SE^H$ {and $0<r<r_0$}, 
$$\vol\Big(B_{_{\hyp}}(A,r)\Big)\leq Cr^{n}\mu_{\SE^{H}}\Big(B_{_{\SE^{H}}}\big(A,r\big)\Big).$$
 We may do this since $\dim \hyp=2n-1$, $\dim \SE^{H}=n-1$, and $\SE^{H} \subset \hyp$.

 Note that when $H=\{x_0\}$ is a point, the curvature bounds become trivial, {and so in place of ~\eqref{e:conormalClose} we work with} $d(x_0,\tilde{x}_h)<{R(h)}$ 
and may take ${\KR_{_{\alpha}}}$ to be arbitrarily close to $0$. In what follows, let $r\sub{H}:T^*M\to \re$ be the geodesic distance to $H$, i.e., $r\sub{H}(x,\xi)=d(x,H)$ {for $(x,\xi)\in T^*M$}, {and write $\pi\sub{M}:T^*M \to M$ for the natural projection.}

 \begin{definition}[regular families]
 We will say a family of submanifolds $\{H_h\}_{h}$ is \emph{regular in the window $[a,b]$} if it satisfies ~\eqref{e:curvature} and
 there is $\e>0$ so that for all $h>0$, the map $(-\e,\e)\times {\Sab^H}\to M$,
\begin{equation}
 \label{e:nonSelf}
 (t,\rho)\mapsto\pi\sub{M}({\varphi_t(\rho)})\;\;\;\text{ is a diffeomorphism}.
\end{equation} 

 \end{definition}
Then, define $|H_pr_{_{\!H}}|:\Sab^H\to \re$ by 
\begin{equation}\label{e:HprH}
|H_pr_{_{\!H}}|(\rho):=\lim_{t\to 0}|H_pr_{_{\!H}}(\varphi_t(\rho))|.
\end{equation}

\begin{definition}[uniformly conormally transverse submanifolds]
A family of submanifolds $\{\tilde H_h\}_{h}$ is said to be \emph{uniformly conormally transverse for $p$ in the window $[a,b]$} provided 
\begin{enumerate}
 \item $\tilde H_h$ is conormally transverse {for $p$} in the window $[a,b]$ for all $h>0$,
 \item there exists {$\FR>0$} so that for all $h>0$
\begin{equation}\label{e:FR}
\inf\Big\{|H_pr_{_{\!\tilde H_{\!h}}}|(\rho)\,\,\big|\,\,{\rho \in \Sab^H}\Big\} \geq \FR.
\end{equation}
\end{enumerate}
\end{definition}
\noindent When the constants involved in our estimates depend on $\{\tilde H_h\}_h$, they will do so \emph{only} through finitely many of the $\KR_{_{\alpha}}$ constants and the constant $\FR$.
\begin{remark}
We note that for $p(x, \xi)=|\xi|^2_{g(x)}$, $ {a=b=1,}$ and $\Sab^H=\SNH$, we have $|H_pr_{_{\!H}}|(\rho)=2$ for all $\rho \in \SNH$. It follows that every family of submanifolds is uniformly conormally transverse {and we may take $\FR=2$}.
\end{remark}

%%%%%%%%%%%%%%%%%%%%%%%%%%%%%%%%%%%%%%%%%%%%%%%%%%%%%%%%%%%%%%%%%%%%%%%%%%%%%%%
%%%%%%%%%%%%%%%%%%%%%%%%%%%%%%%%%%%%%%%%%%%%%%%%%%%%%%%%%%%%%%%%%%%%%%%%%%%%%%%
\subsection{Main results}\label{s:gral}
%%%%%%%%%%%%%%%%%%%%%%%%%%%%%%%%%%%%%%%%%%%%%%%%%%%%%%%%%%%%%%%%%%%%%%%%%%%%%%%
%%%%%%%%%%%%%%%%%%%%%%%%%%%%%%%%%%%%%%%%%%%%%%%%%%%%%%%%%%%%%%%%%%%%%%%%%%%%%%%
We now state the main results from which all of our Kuznecov type asymptotics follow. {Throughout the text, the notation $C=C(a_1,\dots, a_k)$ means that the constant $C$ depends \emph{only} on $a_1, \dots, a_k$.}

\begin{theorem}
\label{t:MAIN}
For $j=1,2$, let $k_j\in\{1,\dots ,n\}$, $\FR^j>0$, $A_j\in \Psi^\infty(M)$. Let {$\Cnr^1>0$, $\Cnr^2>0$} and $\Cnl>0$. There is $$C\sub{0}={C\sub{0}(n, k_1, k_2, A_1, A_2, \FR^1, \FR^2,{\Cnr^1, \Cnr^2}, \Cnl)}>0$$
such that the following holds.

Let $P(h) \in \Psi^m(M)$ be a self-adjoint semiclassical pseudodifferential operator, with classically elliptic symbol $p$. {Let {$0<\delta<\frac{1}{2}$, $K>0$,} $a,b \in \re$ with $a\leq b$, and for $j=1,2$ let} $H_j{\subset M}$ be a submanifold with co-dimension $k_j$ that is regular and uniformly conormally transverse for $p$ in the window $[a,b]$ (with constant $\FR^j$ as in~\eqref{e:FR}). {Then, there exists $\tau_0>0$ with the following property. Let $\Lambda>\Lambda_{\max}$, and $t_0>0$. For $j=1,2$ let $\Ti_j$ be a sub-logarithmic {\resfun}\, with $\Lambda\Omega(\Ti_j)<1-2\delta$ and such that the submanifold $H_j$ is $\Ti_j$ non-recurrent in the window $[a,b]$ via ${\tau_0}$-coverings with constant {$\Cnr^j$}. Suppose $(H_1,H_2)$ is a $(t_0,\Ti_{\max})$ non-looping pair in the window $[a,b]$ via ${\tau_0}$-coverings with constant $\Cnl$ where $\Ti_{\max}=\max(\Ti_1,\Ti_2)$. Let $h^\delta\leq R(h)=o(1)$ and for $j=1,2$ let $\{\tilde{H}_{j,h}\}_h$ be a family of submanifolds of codimension $k_j$ that is regular, uniformly conormally transverse for $p$ in the window $[a,b]$, and satisfies
$$
\sup\Big\{d\big(\rho,\Sab^{\tilde{H}_{j,h}}\big)\,\big|\,\rho\in \Sab^{H_j}\Big\}\leq R(h).
$$

Then, there is $h_0>0$ such that for all $0<h\leq h_0$ and ${s}\in[a-Kh,b+Kh]$,
$$
\Big|\RAst(t_0,h; {s})\Big|\leq C\sub{0}\, h^{1-\frac{k_1+k_2}{2}}\Big/\!\sqrt{ \Ti_1(R(h))\Ti_2(R(h))}.
$$}
\end{theorem}

We also have the following corollary involving the definitions of non-looping (Definition ~\ref{d:non loop gral}) and non-recurrence (Definition~\ref{d:non rec gral}).
\begin{theorem}
\label{t:MAIN2}
%Let $P(h) \in \Psi^m(M)$, $A_1,A_2\in \Psi^\infty(M)$, and $H_1,H_2$ as in Theorem~\ref{t:MAIN}. Let $\Cnl,\Cnr, K>0$, $\Lambda>\Lambda_{\max}$, $0\leq \delta<\frac{1}{2}$, $R(h)\geq \ya{h^\delta}$, and $\tilde{H}_{i,h}$ be as in Theorem~\ref{t:MAIN}. Then there is $C\sub{0}>0$ such that the following holds. 
Let {$\ti$ be a {\resfun},} $\Lambda>\Lambda_{\max}$, $K>0$, ${\e}>0$, $R_0>0$, $0< \delta<\frac{1}{2}$, {and for $j=1,2$} let $\Ti_j$ be a sub-logarithmic {\resfun} {with} $\Lambda \Omega(\Ti_j)<1-2\delta$ {and let $\Ti_{\max}=\max(\Ti_1,\Ti_2)$}. {Suppose the same assumptions as Theorem~\ref{t:MAIN}, but assume instead that {for $j=1,2$ the submanifold $H_j$ is ${(\ti,\Ti_j)}$ non-recurrent in the window $[a,b]$ {at scale $R_0$}, and $(H_1,H_2)$ is a $(t_0, {\Ti_{\max}})$} non-looping pair in the window $[a,b]$ with constant $\Cnl$.} 
{Then, there exist $C\sub{0}={C\sub{0}(n, k_1, k_2, A_1, A_2, \FR^1, \FR^2, {\ti}, \Cnl)}$ and $h_0>0$ such that for all $0<h\leq h_0$ and ${\energy}\in [a-Kh,b+Kh]$}
$$
\Big|\RAst(t_0+{\e},h; {\energy})\Big|\leq C\sub{0}\, h^{1-\frac{k_1+k_2}{2}}\Big/\!\sqrt{\Ti_1(R(h))\Ti_2(R(h))}.
$$
\end{theorem}

For the proof of Theorem~\ref{t:MAIN}, see \S\ref{s:imTheMAIN} {and for the proof of Theorem~\ref{t:MAIN2} see \S~\ref{t:MAIN2}.}

\subsection{Application to the Laplacian}
\label{s:itAllFollows}
{In this section we} show how to obtain Theorems~\ref{t:laplaceDiag},~\ref{t:laplaceoff}, and~\ref{t:laplaceOff for H} from Theorem~\ref{t:MAIN2}. To do this, we {work with} an operator $Q$ such that { $\sigma(Q)(x,\xi)=|\xi|_{g(x)}$ near $\{(x,\xi):\,|\xi|_{g(x)}=1\}$, Theorem~\ref{t:MAIN2} applies with $P=Q$, and for $\lambda=h^{-1}$ and all $N>0$}
\begin{equation}
\label{e:approxSquareRoot}
\begin{gathered}
\1_{(-\infty,1]}(Q)=\Pi_{\lambda},\qquad \qquad
(\rho\sub{h,t_0}*\1_{(-\infty,s]}(Q))(1)=\rho\sub{t_0}*\Pi_{\lambda}+O(h^\infty)_{\smooth}.
 \end{gathered}
\end{equation}
{To build $Q$,} let $\psi_1,\psi_2\in C_c^\infty(\mathbb{R};[0,1])$ with $\supp \psi_1\subset (-1/4,1/4)$, $\supp \psi_2\subset [-16,16]$, $\psi_1 \equiv 1$ on $[-1/16,1/16]$ and $\psi_2 \equiv 1$ on $[-4,4]$. We claim
\begin{equation}
\label{e:goodApprox}
Q=(1-\psi_1(-h^2\Delta_g))\psi_2(-h^2\Delta_g)\sqrt{-h^2\Delta_g}-h^2\Delta_g(1-\psi_2(-h^2\Delta_g))
\end{equation}
satisfies the desired properties. 
 {Observe that the second term in \eqref{e:equivalentConvolve} is added to make $Q$ classically elliptic, and that we use $-h^2\Delta_g$ rather than $\sqrt{-h^2\Delta_g}$ in order to apply \cite[Theorem 14.9]{EZB} to obtain $Q\in \Psi^2(M)$. Note also that $Q$ is self-adjoint and $\sigma(Q)=|\xi|_g$ on $\{\frac{1}{2}\leq |\xi|_g\leq 2\}$,}
\begin{gather}
\label{e:equivalentConvolve}
\rho\sub{t_0}*\Pi_{\lambda}=\Big(\rho\sub{t_0,h}*\1_{(-\infty,s]}\Big(\sqrt{-h^2\Delta_g}\Big)\Big)(1),\qquad \Pi_\lambda=\1_{(-\infty,1]}\Big(\sqrt{-h^2\Delta_g}\Big)\\
\label{e:exact}
\1_{(-\infty,s]}(Q)=\1_{(-\infty,s]}\Big(\sqrt{-h^2\Delta_g}\Big),\qquad s\in[\tfrac{1}{2},2]
\end{gather}
and $\1_{(-\infty,s]}(Q)=\1_{(-\infty,s]}(\sqrt{-h^2\Delta_g})=0$ for $s<0$.
Finally, we use the ellipticity of both $Q$ and $-h^2\Delta_g$ to obtain that for $N\geq 0$
$$
\1_{(-\infty,s]}(Q)=O\sub{N}(\langle s\rangle ^{N})_{\smooth},\qquad \1_{(-\infty,s]}\Big(\sqrt{-h^2\Delta_g}\Big)=O\sub{N}(\langle s\rangle ^{2N})_{\smooth}.
$$
Now, for all $N>0$ and $L>1$ there is $C\sub{N,L}>0$ so that $|\rho\Big(\frac{t_0}{h}(1-s)\Big)|\leq C\sub{N,L} h^{2N+L}\langle s\rangle^{-2N-L}$ { on $|s-1|>\frac{1}{2}$}. Therefore
\begin{equation}
\label{e:wellApproximated}
\begin{aligned}
&\Big[\rho_{t_0,h}*\Big(\1_{(-\infty,s]}(Q)-\1_{(-\infty,s]}(\sqrt{-h^2\Delta_g})\Big)\Big](1)\\
&=\int_{\substack{s\notin [1/2,2]\\s\geq 0}}\frac{t_0}{h}\rho\Big(\frac{t_0}{h}(1-s)\Big)\Big(\1_{(-\infty,s]}(Q)-\1_{(-\infty,s]}\Big(\sqrt{-h^2\Delta_g}\Big)\Big)ds
%&=\int_{\substack{s\notin [1/2,2]\\s\geq 0}}C\sub{N,L} h^{2N+L}\langle s\rangle^{-2N-L}O_N(\langle s\rangle^{2N})_{\smooth}ds\\
=O\sub{N}(h^{2N+L-1})_{\smooth}.
\end{aligned}
\end{equation}
Combining~\eqref{e:equivalentConvolve} with~\eqref{e:exact} and~\eqref{e:wellApproximated}, we obtain~\eqref{e:approxSquareRoot}.

Now, every submanifold is conormally transverse for $p(x,\xi)=|\xi|_{g(x)}$ at {$p^{-1}(1)$ with constant} $\FR=1$. Therefore, Theorems~\ref{t:laplaceDiag},~\ref{t:laplaceoff}, and~\ref{t:laplaceOff for H} follow from Theorem~\ref{t:MAIN2}. {To see this, we set} $P=Q$, $a=b=1$, and observe that the Hamiltonian flow for $\sigma(Q)$ near $S^*_xM$ is equal to the geodesic flow. {In particular, the dynamical definitions~\ref{d:non loop gral} and~\ref{d:non rec gral} applied to $Q$ at $E=1$ are exactly Definitions~\ref{d:non looping through} and~\ref{d:non rec} with $S^*_xM$ replaced by $SN^*H$.} {This is true because Definitions~\ref{d:non looping through} and~\ref{d:non rec} are stated with $\varphi_t$ being the \emph{homogeneous} geodesic flow, i.e., the flow generated by $|\xi|_{g(x)}$.} {Next, we apply Theorem~\ref{t:laplaceOff for H} with} $\Lambda=2\Lambda_{\max}{+1}$, $h=\lambda^{-1}$, {and work with the \resfuns\, $\widetilde{\Ti}_j=(\Lambda{\Omega_0})^{-1}(1-2\delta)\Ti_j$ for $j=1,2$}.

%%%%%%%%%%%%%%%%%%%%%%%%%%%%%%%%%%%%%%%%%%%%%%%%%%%%%%%%%%%%%%%%%%%%%%%%%%%%%%%
%%%%%%%%%%%%%%%%%%%%%%%%%%%%%%%%%%%%%%%%%%%%%%%%%%%%%%%%%%%%%%%%%%%%%%%%%%%%%%%
\section{Dynamical assumptions and coverings}\label{s:meas-then-dyn}
%%%%%%%%%%%%%%%%%%%%%%%%%%%%%%%%%%%%%%%%%%%%%%%%%%%%%%%%%%%%%%%%%%%%%%%%%%%%%%%
%%%%%%%%%%%%%%%%%%%%%%%%%%%%%%%%%%%%%%%%%%%%%%%%%%%%%%%%%%%%%%%%%%%%%%%%%%%%%%%
In this section we relate the non-looping and non-recurrence concepts introduced in Definitions \ref{d:non loop gral} , \ref{d:non rec gral}, to their analogues via coverings given in Definitions \ref{d:non loop cov}, \ref{d:non rec cov}.

%%%%%%%%%%%%%%%%%%%%%%%%%%%%%%%%%%%%%%%%%%%%%%%%%%%%%%%%%%%%%%%%%%%%%%%%%%%%%%%
%%%%%%%%%%%%%%%%%%%%%%%%%%%%%%%%%%%%%%%%%%%%%%%%%%%%%%%%%%%%%%%%%%%%%%%%%%%%%%%
\begin{proposition}
\label{p:non loop implies cov}
Let $H_1, H_2 \subset M$ be smooth submanifolds. Let $a,b \in \re$ be such that $H_1, H_2$ are conormally transverse for $p$ in the window $[a,b]$, and $\tau_0>0$. Let $t_0>0$, {$\Ti$ a {\resfun}}, and suppose $(H_1,H_2)$ is a $(t_0,\Ti)$
non-looping pair in the window $[a,b]$ {with constant $\Cnl$}. Then, {there is $\widetilde{\Cnl}=\widetilde{\Cnl}(p,a,b,n, \Cnl,{H_1,
H_2})>0$} such that $(H_1,H_2)$ is a $(t_0{+{3}\tau_0}, {\widetilde{\Ti}})$ non-looping pair in the window $[a,b]$ via {$\tau_0$-coverings} with {constant $\widetilde{\Cnl}$} {and with $\widetilde{\Ti}(R)=\Ti(4R){-{3}\tau_0}$}.

%\ya{\ \\ \cs actually, I think we get $\widetilde{\Ti}(R)=\Ti(4R){-\tau_0}-2R$ and $\tilde t_0=t_0+\tau_0+2R$\cs }

%\blue{In addition, if} $\Ti$ is sub-logarithmic, then the following hold.
%\begin{enumerate}
%\item There is $h_0=h_0(\{R(h)\}_{h>0})$ such that $\Ti(4R(h))\geq \tfrac{1}{2} \Ti(R(h))$ for $0<h<h_0$, and if $R(h)\leq h^\e$ for some $\ep>0$, then
%$
%\Ti(4R(h)) \geq \ep \tfrac{1}{2} \Ti(h).
%$
%\item For all $\Lambda>\Lambda_{\max}$, $\e>0$, there is $h_0=h_0(\{R(h)\}_{h>0},\Lambda,\e)$ such that for $0<h<h_0$,
%$$
%\Ti(R(h))\leq (\Omega(\Ti)\Lambda +\e)T_e(h).
%$$
%\end{enumerate}

\end{proposition}
%%%%%%%%%%%%%%%%%%%%%%%%%%%%%%%%%%%%%%%%%%%%%%%%%%%%%%%%%%%%%%%%%%%%%%%%%%%%%%%

{Before proving the proposition, we record some facts about sub-logarithmic \resfuns.}
\begin{lemma} 
\label{l:sublog}
Suppose $\Ti$ is a sub-logarithmic {\resfun}. 
\begin{enumerate} \item For $0<a<b<1$, 
$$
{\Ti(b)\leq \Ti(a)}\leq \frac{\log a}{\log b}\,{\Ti(b)}.
$$
In particular, $\Ti(R)\leq \frac{\log R}{\log \mu+\log R} \Ti(\mu R)$ for {$0<\mu<R^{-1}$}.
\item Let $f(s):=-\log (\Ti^{-1}(s))$. Then, $f$ extends to a differentiable function on $[0,\infty)$, $f(0)=0$, and $f(a)\leq \frac{a}{b}f(b)$ for $0<a<b$.
\item Let $0<\delta<\frac{1}{2}$, and $R(h)\geq h^\delta$ with $R(h)=o(1)$. Then for all $\Lambda>\Lambda_{\max}$, $\e>0$, there is $h_0>0$ such that for $0<h<h_0$
$$
\Ti(R(h))\leq (\Omega(\Ti)\Lambda+\e)T_e(h).
$$
\end{enumerate}
\end{lemma}
\begin{proof}
Note that 
$$
0\leq \log {\frac{\Ti(a)}{\Ti(b)}}={-}\int^{b}_{a}\frac{\Ti'(s)}{\Ti(s)}ds\leq \int_{a}^{b}\frac{1}{s\log s^{-1}}ds= \log \Big(\frac{\log a^{-1}}{\log b^{-1}}\Big),
$$
and hence the first claim holds.
For the second claim, observe that since $\Ti$ is sub-logarithmic,
$
f'(s)\geq -\frac{\log (\Ti^{-1}(s))}{s}=\frac{f(s)}{s}.
$

To prove the last claim, observe that since $R(h)=o(1)$, for all $\Lambda>\Lambda_{\max}$ and $\e>0$, there is $h_0>0$ such that for $0<h<h_0$, 
$$
\Ti(R(h))\leq (\Omega(\Ti)+\e{\Lambda^{-1}})\log R(h)^{-1}\leq (\Omega(\Ti)\Lambda+\e) T_e(h).
$$
{The second inequality follows from definitions {\eqref{e:Omega}, \eqref{e:Lmax}, and} $R(h)\geq h^\delta$ with $0<\delta<\frac{1}{2}$.}
\end{proof}

{In the following lemma we explain how to partition a $(\mathfrak{D},\tau,r)$-good cover for $\SE^{H_1}$ into tubes that do not loop through $\SE^{H_2}$ for times in $(t_0, T)$, and tubes that are `bad' in the sense that they do loop through $\SE^{H_2}$. We do this while controlling the number of `bad` tubes in terms of the size of the set $ {\mathcal{L}\sub{{H_1,H_2}}^{{S},E}(t_0, T)}$ for $S>4r$.}
%%%%%%%%%%%%%%%%%%%%%%%%%%%%%%%%%%%%%%%%%%%%%%%%%%%%%%%%%%%%%%%%%%%%%%%%%%%%%%%
\begin{lemma}\label{l:non loop implies cov}
Let $a,b\in \re$, $H_1, H_2 \subset M$ be smooth submanifolds such that $H_1, H_2$ are conormally transverse for $p$ in the window $[a,b]$. Then there is $C\sub{0}=C\sub{0}(p,a,b,n,{H_1,H_2})$ such that the following holds.
{Let $\tau_0>0$, $r>0$, and $0<\tau<\tau_0$. 
For $i=1,2$ let $\{\mc{T}_{j}^i(r)\}_{j\in \mathcal{J}^i(r)}$ be a $(\mathfrak{D},\tau,r)$-good cover of $\Sigma^{H_i}_{[a,b]}$. Let $t_0>0$, $T>0$. Then, for all $E\in[a,b]$ and
$
S\geq 4r
$}
there is a splitting $\mc{J}\sub{E}^1(r)=\BE^1(r)\cup\GE^1(r)$ such that
\begin{enumerate}
\item for $j\in \GE^1(r)$ and $k\in \mc{J}\sub{E}^2(r)$ 
 $$
 \bigcup_{t_0+{{2}(\tau+r)}\leq |t|\leq T-{{2}(\tau+r)}}\varphi_t(\mc{T}_{j}^1(r))\cap \mc{T}_{k}^2(r)=\emptyset,
 $$
 \item $|\BE^1(r)|\leq {\mathfrak{D}}C\sub{0}r^{1-n}\,\mu\!\sub{\SE^{H_1}}\!\!\Big( B\!\!\sub{\SE^{H_1}}\!\Big(\smalleq{ \mathcal{L}\sub{{H_1,H_2}}^{{S,E}}(t_0, T)}, S\Big) \Big)$. 
\end{enumerate}
\end{lemma}
%%%%%%%%%%%%%%%%%%%%%%%%%%%%%%%%%%%%%%%%%%%%%%%%%%%%%%%%%%%%%%%%%%%%%%%%%%%%%%%
\begin{proof} {For $j=1,2$ let} $\hyp_j\subset T^*M$ be the hypersurface transverse to the flow, with $\Sab^{H_j} \subset \hyp_j$, used to build the tubes of the cover, {as explained in \eqref{e:mc H}}. Let $E\in [a,b]$ and for $S>0$ set
$$
\BE^1(r):=\big\{ j\in \J\sub{E}^1(r)\,:\, \T_j^1(r)\cap B_{_{\hyp_{1}}}(\mc{L}\sub{H_1,H_2}^{{S,E}}(t_0,T),2r)\neq \emptyset\big\}.
$$
Then, for $j\in \BE^{1}(r)$,
$$
{\T_j^1(r)}\cap \hyp_{1}\;\subset B_{_{\hyp_{1}}}(\mc{L}\sub{H_1,H_2}^{{S,E}}(t_0,T),4r).
$$
In particular, there exists $C\sub{0}={C\sub{0}(p,a,b,n,{H_1,H_2})}>0$ such that for all $S\geq 4r$ 
$$
|\BE^1(r)|\leq {\mathfrak{D}}r^{1-2n}\vol\Big(B_{_{\hyp_{1}}}(\smalleq{\mc{L}\sub{H_1,H_2}^{{S,E}}(t_0,T)},4r)\Big)\leq C\sub{0}{\mathfrak{D}}r^{1-n}\mu_{\SE^{H_1}}\Big(B_{_{\SE^{H_1}}}\big(\smalleq{\mc{L}^{{S,E}}\sub{H_1,H_2}(t_0,T)},S\big)\Big).
$$
This proves the claim in (2).

To see the claim in (1), let $j\in \GE^1(r):=\J\sub{E}^1(r)\setminus \BE^1(r)$. Then, ${\T_j^{{1}}}(r)={\Lambda_{\rho_j}^\tau}(r)$ for some $\rho_j\in \hyp_{1}$ with $d(\rho_j,\SE^{H_1})<2r$ and $d(\rho_j, \smalleq{\mc{L}\sub{H_1,H_2}^{{S,E}}(t_0,T)})>3r$. This yields that there is $\rho_0\in \SE^{H_1}\setminus \mc{L}\sub{H_1,H_2}^{{S,E}}(t_0,T)$ such that $d(\rho_0,\rho_j)<2r$. In particular, since
$
\underset{t_0\leq |t|\leq T}{\bigcup}\varphi_t(B(\rho_0,S))\cap B(\SE^{H_2},S)=\emptyset
$ 
and 
$
\T_j^1(r)\subset \underset{|t|\leq {\tau+r}}{\bigcup}\varphi_t(B(\rho_0,3r)),
$
this yields 
\begin{equation}\label{e:salami}
\bigcup_{t_0+{\tau+r}\leq |t|\leq T-{(\tau+r)}}\varphi_t(\T_j^1(r))\cap B(\SE^{H_2},S)=\emptyset
\end{equation}
for $S\geq 4r$. 
On the other hand,
since for all {$k \in \mc{J}\sub{E}^2(r)$}, we have 
$
\T_k^2(r)\cap \hyp_{2}\subset B(\SE^{H_2},3r),
$
\begin{equation}\label{e:bread}
\T_k^2(r)\subset \bigcup_{|t|\leq {\tau+r}}\varphi_t(B(\SE^{H_2},3r))
\end{equation}
In particular, {combining \eqref{e:salami} and \eqref{e:bread} we have }
$$
\bigcup_{t_0+{2}{(\tau+r)}\leq |t|\leq T-{{2}(\tau+r)}}\varphi_t(\T_j^1(r))\cap B(\SE^{H_2},S)=\emptyset.
$$
{Thus, the} claim (1) holds, provided $S\geq 4r$.
\end{proof}

{With Lemmas \ref{l:sublog} and \ref{l:non loop implies cov} in place, we are now ready to prove Proposition \ref{p:non loop implies cov}.}
%%%%%%%%%%%%%%%%%%%%%%%%%%%%%%%%%%%%%%%%%%%%%%%%%%%%%%%%%%%%%%%%%%%%%%%%%%%%%%%%%

\begin{proof}[Proof of Proposition \ref{p:non loop implies cov}]

Let $C\sub{0}=C\sub{0}(p, a, b,n,{H_1,H_2})$ be as in Lemma \ref{l:non loop implies cov}.
We apply Lemma~\ref{l:non loop implies cov} with $r=R$, {$T=\Ti(S)$}, $S=4R$, {$0<R<\frac{1}{2}\tau_0$}. {This shows that $(H_1,H_2)$ is a $[t_0{+{3}\tau_0}, \widetilde{\Ti}]$ non-looping pair in the window $[a,b]$ via {$\tau$-coverings} with constant $\widetilde{\Cnl}=C_0^2 \Cnl$}.
%
%
% Using Lemma~\ref{l:sublog}, we have
%$$
%\Ti(R(h))\leq \frac{\log R(h)}{\log S(h)} \Ti( S(h))\blue{=} \frac{1}{1+\frac{\log 4}{\log R(h)}}\Ti(S(h)),
%$$
%\blue{proving the first part of the claim in $(1)$.}
%\blue{In addition,} if $R(h)\leq h^\e$, then,
%$$
%\Ti(h)\leq \frac{\log h}{\log R(h)}\Ti(R(h))\leq \e^{-1}\frac{1}{1+\frac{\log 4}{\log R(h)}}\Ti(\blue{S(h)}),
%$$
%\blue{and we finished proving the claim in $(1)$. The claim in $(3)$ follows from $(3)$ in Lemma~\ref{l:sublog}.}
\end{proof}

\begin{lemma}
\label{l:4minutes}
{There is a constant {$C_n>0$}, {depending only on $n$,} such that the following holds.} Let $\tau_0>0$, $t_0>0$, $H_1,H_2\subset M$ be smooth submanifolds such that $H_1$ and $H_2$ are conormally transverse for $p$ in the window $[a,b]$. Let $\Ti$ be a {\resfun}. If $(H_1,H_2)$ is a {$(t_0,\Ti)$} non-looping pair in the window $[a,b]$ via {$\tau_0$-coverings} with constant $\Cnl$, then $(H_1,H_2)$ is a $(t_0,{\widetilde{\Ti}})$ non-looping pair in the window $[a,b]$ with constant $\Cnl {C_n}$ {and $\widetilde{\Ti}(R)=\Ti(2R)$}.
\end{lemma}
%\begin{remark}
%\label{r:4minutes}
%{\sout{Note that if $R(h)=h^\delta$ and $T(h)=a\log h^{-1}$, then $\Ti(R)= \frac{a}{\delta} \log (2R)^{-1}.$}}
%\end{remark}

\begin{proof}
Let $E\in[a,b]$ {and fix $i,j \in \{1,2\}$, $i\neq j$}. {For each $R>0$} consider the non-looping partition $\J\sub{E}^i({R})=\GE^i(R)\sqcup\BE^i({R})$ given by Definition \eqref{d:non loop cov}.
Let $\rho\in \mc{L}\sub{H_i,H_j}^{{{R}/2,E}}(t_0,{\Ti(R)})$. Then, there are $\rho_1\in B(\rho,{R}/2)$ and $t_0\leq |t|\leq {\Ti(R)}$ such that $\varphi_t(\rho_1)\in B(\SE^{H_j},{R}/2)$.
Hence, there is $\ell\in {\BE^i}({R})$ such that $\rho_1\in \T^i_\ell({R})$ and hence $\rho\in \T^i_\ell(2{R})$. This implies
$
B\sub{{\SE^{H_i}}}(\rho,{R}/2)\subset {\T^i_\ell(3{R})}. 
$
Thus, 
$$
B\sub{{\SE^{H_i}}}\big(\smalleq{\mc{L}_{H_i,H_j}^{{{R}/2,E}}(t_0,{\Ti(R)})}, {R}/2\big){\subset}
\bigcup_{\ell\in {\BE^i}({R})}{\T^i_\ell(3{R})}.
$$
In particular, {there exists $C_{{n}}>0$ such that}
$$
{\mu\sub{{\SE^{H_i}}}}\Big(B\sub{{\SE^{H_i}}}\!\big({\smalleq{ \mathcal{L}_{H_i,H_j}^{{{R}/2,E}}(t_0, {\Ti(R)})}}, {R}/2\big)\Big)\leq C_{{n}} {R}^{n-1}|{\BE^i}({R})|.
$$

Therefore,
\begin{align*}
&{\mu\sub{{\SE^{H_1}}}}\Big(B\sub{{\SE^{H_1}}}\!\big({\smalleq{ \mathcal{L}_{H_1,H_2}^{{{R}/2,E}}(t_0, {\Ti(R)})}}, {R}/2\big)\Big){\mu\sub{{\SE^{H_2}}}}\Big(B\sub{{\SE^{H_2}}}\!\big({\smalleq{ \mathcal{L}_{H_2,H_1}^{{{R}/2,E}}(t_0, {\Ti(R)})}}, {R}/2\big)\Big){\Ti(
R)}^2\\&\qquad \qquad\leq {C_n^2} {R}^{2n-2}|{\BE^1}({R})||\BE^2({R})|{\Ti(R)}^2\leq {C_n^2}\mathfrak{D}^2\Cnl .
\end{align*}
{The lemma follows from Definition \ref{d:non loop gral} after taking the limit $R \to 0^+$} {and redefining $C_n$}.
%Then,
%$$
% \limsup_{R\to 0^+} \Bigg({\mu\sub{{\SE^{H_1}}}}\Big(B\sub{{\SE^{H_1}}}\!\big({\smalleq{ \mathcal{L}_{H_1,H_2}^{{R}}(t_0, \Ti(\red{2}R))}}, R\big)\Big){\mu\sub{{\SE^{H_2}}}}\Big(B\sub{{\SE^{H_2}}}\!\big({\smalleq{ \mathcal{L}_{H_2,H_1}^{{R}}(t_0, \Ti(\red{2}R))}}, R\big)\Big) \Ti(\red{2}R)^2 \Bigg)
%$$
%is bounded by $ C\mathfrak{D}^2\Cnl$ as claimed.
\end{proof}

%%%%%%%%%%%%%%%%%%%%%%%%%%%%%%%%%%%%%%%%%%%%%%%%%%%%%%%%%%%%%%%%%%%%%%%%%%%%%%%%
\begin{proposition}\label{p:non rec implies cov}
 Let {$\ti$, $\Ti$ be \resfuns} and $H \subset M$ be a smooth submanifold. Let $a,b \in \re$ be such that $H$ is conormally transverse for $p$ in the window $[a,b]$. Suppose $H$ is ${(\ti,\Ti)}$ non-recurrent in the window $[a,b]$ {{at scale $R_0$}}.

Then, {there exists ${\Cnr=\Cnr}(M,p,{\ti},{R_0})>0$ such that} for all {$\tau_0>0$}, there {is a {\resfun} $\widetilde{\Ti}$} such that the submanifold $H$ is {$\widetilde{\Ti}$} non-recurrent in the window $[a,b]$ via {$\tau_0$-coverings} {with constant {${\Cnr}$}}. Moreover, there is $c>0$ such that if $\Ti$ is sub-logarithmic, then {$\widetilde{\Ti}(R)\geq c\Ti(R)$} {for all $R$}.
\end{proposition}

The proof of this result hinges on two lemmas. To state the first one, we introduce a slight adaptation of ~\cite[Definition 3]{CG19dyn}. Let $\e_0>0$, {$\digamma>0$,} $\ti_0:[\e_0, +\infty) \to [1, +\infty)$, and $f:[0,\infty)\to [0,\infty)$.
We say a set $A_0$ is \emph{$(\e_0,\ti_0,\digamma, f)$ controlled up to time $T$}
provided it is $(\e_0, \ti_0, \digamma)$ controlled up to time $T$ in the sense of ~\cite[Definition 3]{CG19dyn} except that we replace {the condition on $r$ by
\begin{equation}
\label{e:point 1}
0<r<\tfrac{1}{\digamma}e^{-\digamma \Lambda T-f(T)}r_0
\end{equation}
and replace} point (3) by 
\begin{equation}\label{e:point 3}
\inf_{k}R_{1,k} \geq {\tfrac{1}{4}e^{-f(T)}}\inf_i R_{0,i}.
\end{equation}

Next, fix $E \in [a,b]$. 
Since $H$ is $ {(\ti,}\Ti)$ non-recurrent in the window $[a,b]$ {at scale ${R_0}$,} for all $\rho\in \SE^H$ there exists {a choice of $\pm$} such that for all $A \subset {B\sub{\SE^H}(\rho,R_0)}$, {$0<R<R_0$}, {$\e>0$}, and $T>{\ti(\e)}$
\begin{equation}\label{e: epsilon}
\mu\sub{\SE^H} \Big( B\sub{\SE^H}\!\Big({\smalleq{ \mc{R}^{{e^{-f(T)}}R}\sub{A,{\pm}}({\ti(\e)},T)}}\,,\, e^{-f(T)}R\Big) \leq {\e} \,\mu\sub{\SE^H}(B\sub{\SE^H}(A,R) \Big),
\end{equation}
{with $f$ as in Lemma~\ref{l:sublog}.}
 Then, extract a finite cover of $\SE^H$ {by balls $\tilde{B}_\rho=B(\rho,R_{{0}}/2)$} and set
\begin{equation}\label{e:collection A_e2}
{\tilde{\mc{A}}}\sub{E}:=\{{\tilde{B}}_{\rho_i}\}_{i=1}^K, 
\qquad \text{and}\qquad
\mc{A}\sub{E}:=\{{B}_{\rho_i}\}_{i=1}^K, 
\end{equation} 
{where ${B}_\rho=B(\rho,R_{{0}})$.}
{Note that, again using that $H$ is non-recurrent with at scale $R_0$}, we may assume $K\leq C_n R_0^{1-n}$ where $C_n$ is a constant depending only on $n$.

%%%%%%%%%%%%%%%%%%%%%%%%%%%%%%%%%%%%%%%%%%%%%%%%%%%%%%%%%%%%%%%%%%%%%%%%%%%%%%%
\begin{lemma}\label{l:control}
 Let $H$, {$\ti$ and} $\Ti$ be as in Proposition \ref{p:non rec implies cov} and $f(T):= -\log (\Ti^{-1}(T))$. 
 Then, there exist {$c_n>0$ depending only on $n$} and $\digamma>0$ such that for all {$E\in[a,b]$ and} ${T>1}$ every ball in $\mc{A}\sub{E}$ is $(0,{\ti_0},\digamma, f)$ controlled up to time $T$ {with $\ti_0(\e)=\ti(c_n\e)$}.
\end{lemma}
%%%%%%%%%%%%%%%%%%%%%%%%%%%%%%%%%%%%%%%%%%%%%%%%%%%%%%%%%%%%%%%%%%%%%%%%%%%%%%%
\begin{proof}

{Let $E\in[a,b]$.} 
Let $A_0:=B_{\rho_0}$ for some $B_{\rho_0}\in \mc{A}\sub{E}$, {$\e_1>0$}, $\Lambda {>} \Lambda_{\max}$, and $0<\tau<\tfrac{1}{2}\tau\sub{\inj_H}$. 
Let $T> {1}$ and $0\leq {\tilde{R}}_0 \leq \frac{1}{\digamma}e^{-\digamma \Lambda T }$ for $\digamma>{2R_0^{-1}}$ to be determined later. Let $0<r_0<{\tilde{R}}_0$. Suppose $A_1\subset A_0$ and $\{B_{0,i}\}_{i=1}^N$ are balls centered in $A_0$ with radii $R_{0,i}\in [r_0,{\tilde{R}}_0]$ such that 
$
A_1\subset \cup_{i=1}^N B_{0,i}\subset A_0.
$

Let $R:=\tfrac{1}{2}\inf_i R_{0,i}$. There exist $C_n>0$, depending only on $n$, and a collection of balls $\{\tilde{B}_{0,i}\}_{i=1}^{N_0}$ of radius $R$, such that 
\begin{equation}\label{e:hotdog}
A_1\subset \bigcup_{i=1}^{N_0}\tilde{B}_{0,i},\qquad N_0R^{n-1}\leq C_n\sum_{i=1}^NR_{0,i}^{n-1}.
\end{equation}

Fix $0\leq r \leq \frac{1}{\digamma}e^{-\digamma \Lambda T{-f(T)} }r_0$.
Next, let $\{B(q_j, r)\}_{j\in \mc{J}} \subset \SE^H$ be a cover of $\SE^H$ by balls of radius $r$ such that there are at most $\mathfrak{D}_n$ balls over each point in $\SE^H$, {where $\mathfrak{D}_n>0$ depends only on $n$}.
{Assume, without loss of generality, that \eqref{e: epsilon} holds for $\rho_0$ with the choice $\pm=+$.}
Next, set $\mc{J}\sub{\!A_1}:=\{j \in \mc{J}:\; B(q_j, {\frac{1}{2}}e^{-f(T)}R)\cap \mc{R}^{{e^{-f(T)}}R}\sub{{A_1, +}}(\ti{(\e_1)},T)\neq \emptyset\}$.
Defining the collection
$$
\{B_{1,i}\}_{i=1}^{N_1}:=\Big\{ B\sub{{\SE^H}}\big (q_j, \tfrac{1}{2}e^{-f(T)}R\big):\; j \in \mc{J}\sub{\!A_1} \Big\},
$$
we have
$
\bigcup_{i=1}^{N_1}B_{1,i}\subset B\sub{\SE^H}\!\Big({\smalleq{ \mc{R}^{{e^{-f(T)}}R}\sub{{A_1, +}}({\ti(\e_1)},T)}}\,,\, e^{-f(T)}R\Big). 
$
Then, letting $R_{1,i}:=\tfrac{1}{2}e^{-f(T)}R$, we have $R_{1,i}\in [0,\tfrac{1}{4}{\tilde{R}}_0]$, and using that $R<R_0/2$ the bound in \eqref{e: epsilon} {applied to $A_1$} yields
\begin{equation}\label{e:apples}
\sum_{i=1}^{N_1}R_{1,i}^{n-1} \leq {\e_1}\,\mathfrak{D}_n\, \mu\sub{\SE^H}(B\sub{\SE^H}(A_1,R) \Big).
\end{equation}
Next, {by \eqref{e:hotdog}} note that 
$
B\sub{\SE^H}(A_1,R)\subset \bigcup_{i=1}^{N_0}2\tilde{B}_{0,i},
$
where $2\tilde{B}_{0,i}$ denotes the ball with the same center as $\tilde{B}_{0,i}$ but {with radius $2R$}. {Using \eqref{e:hotdog} again} there is $C_n>0$ such that 
\begin{equation}\label{e:oranges}
\mu\sub{\SE^H}(B\sub{\SE^H}(A_1,R))\leq {\mu\sub{\SE^H} \Big(\bigcup_{i=1}^{N_0}2\tilde{B}_{0,i}\Big)} \leq C_n\sum_{i=1}^N R_{0,i}^{n-1}.
\end{equation}

{Let ${\e:=\e_1} C_n \mathfrak{D}_n$.} Combining \eqref{e:apples} and \eqref{e:oranges} yields point $(2)$ of~\cite[Definition 3]{CG19dyn} {with $\ti_0(\e)=\ti(\e/(C_n\mathfrak{D}_n))$}. By the definition of $R$, we also note that point $(3)$, which was replaced by \eqref{e:point 3}, also holds. 

It remains to check point (1) i.e. there is $\digamma>0$ such that $\Lambda^\tau_{A_1\setminus \cup_{k}B_{1,k}}(r)$ is $[\ti_0{(\e)},T]$ non-self looping for
$
0<r<\frac{1}{\digamma}e^{-\digamma \Lambda T{-f(T)}}R.
$
For this, suppose $\rho_1,\rho_2\in \Lambda^{\tau}_{A_1\setminus \cup_{k}B_{1,k}}(r)$ and $t\in[{\ti_0(\e)},T]$ such that 
$
\varphi_t(\rho_1)=\rho_2.
$
Then, there are $s_1,s_2\in{[-\tau-r,\tau+r]}$, $q_1,q_2\in A_1\setminus \cup_k B_{1,k}$ such that $d(\rho_i,\varphi_{s_i}(q_i))<{r}$.
In particular, {there is $C\sub{0}>0$ depending only on $(M,p,a,b,\Lambda)$ such that} 
%\marginpar{\je{seems there is a small error in our other paper here. Note that $C\sub{0}$ should depend on how close $\Lambda$ is to $\Lambda_{\max}$}}
$$
d(\varphi_{s_2-t-s_1}(q_2), A_1)<(1+C\sub{0}e^{\Lambda(|t|+{2\tau+2r})})r.
$$
Finally, let $\digamma>0$ be large enough so that 
$\frac{1}{\digamma}e^{-\digamma \Lambda T}<{\min((1+C\sub{0}e^{\Lambda(|T|+{2\tau+2r})})^{-1},R_0/2)}$. Note that the choice of $\digamma$ does not need to depend on $T$.
Then, since
$
r<(1+C\sub{0}e^{\Lambda(|{T}|+{2\tau+2r})})^{-1}{e^{-f(T)}}R,
$
we have 
$q_2\in \mc{R}^{{e^{-f(T)}}R}\sub{A_1{,+}}({\ti_0(\e)},T)$, which is a contradiction since {$ \mc{R}^{{e^{-f(T)}}R}\sub{A_1{,+}}({\ti_0(\e)},T)\subset \cup_i B_{1,i}$.}
\end{proof}

%%%%%%%%%%%%%%%%%%%%%%%%%%%%%%%%%%%%%%%%%%%%%%%%%%%%%%%%%%%%%%%%%%%%%%%%%%%%%%%
In what follows we fix $1<\beta_0 < \ep_0^{-1}$ and define
$$
{\bf{F}}(T):=\sum_{k=0}^{\log_{\beta_0}T}f\big (\beta_0^{-k}\,T\big ). 
$$

%%%%%%%%%%%%%%%%%%%%%%%%%%%%%%%%%%%%%%%%%%%%%%%%%%%%%%%%%%%%%%%%%%%%%%%%%%%%%%%

\begin{lemma}\label{e:control2}
Let $B\subset \SE^H$ be a ball of radius $\delta>0$. Let ${0<\e_0<1}$, $\ti_0:[\e_0, +\infty) \to [1, +\infty)$, $f:[0,\infty)\to [0,\infty)$ increasing with $f(e^{-x})\in L^1([0,\infty))$, $T_0>0$, and $\digamma>0$, such {that $B$ can be $(\e_0, {\ti_0}, \digamma,f)$-controlled up to time $T_0$.}
Let $0<m<\frac{\log T_0-\log {\ti_0({\e_0})}}{\log \beta_0}$ be a positive integer, {$\Lambda> \Lambda_{\max}$}, 
\[
0< {\tilde{R}}_0\leq {\min}\Big\{{\tfrac{1}{\digamma}}e^{-{\digamma}\Lambda T_0}, {\tfrac{\delta}{10}}\Big\},
\qquad 
0<r_1<{\tfrac{1}{5\digamma}}e^{-({\digamma}\Lambda T_0+\Decay+{f(T_0)})}{\tilde{R}}_0,
\]
and $B_0\subset B$ with $d(B_0, B^c)>{\tilde{R}}_0$. Let $0<\tau<\tau_0$ and suppose $\{\Lambda_{_{\rho_j}}^\tau(r_1)\}_{j=1}^{N_{r_1}}$ is a $({\mathfrak{D}},\tau, r_1)$ good cover of $\SigH$ and set
$
\mc{E}:=\{j \in \{1, \dots, {N_{r_1}}\}: \Lambda_{\rho_j}^\tau(r_1)\cap \Lambda^\tau_{B_0}(\tfrac{r_1}{5})\neq \emptyset\}.
$

 Then, there exist $C_{_{\!M,p}}>0$ depending only on $(M,p)$ and sets $\{\mc{G}\sub{E,\ell}\}_{\ell =0}^m\subset \{1,\dots N_{r_1}\}$, $\BE\subset \{1,\dots N_{r_1}\}$ so that 
 $\mc{E}\;\subset\; \BE\cup \displaystyle\cup_{\ell=0}^m \mc{G}\sub{E,\ell}$ and
\begin{align}
%&\bullet\; \mc{E}\;\subset\; \BE\cup \displaystyle\bigcup_{\ell=0}^m\ya{\mc{G}\sub{E,\ell}},\label{e:union}\\
&\bullet \;\bigcup_{i\in \mc{G}\sub{E,\ell}}\Lambda_{\rho_i}^\tau(r_1)\text{\; is \;} \big[\ti_0(\e_0),\beta_0^{-\ell}T_0\big]\text{\; non-self looping {for $\ell \in\{0, \dots, m\}$},} \label{e:nsl} \\
&\bullet\; |\mc{G}\sub{E,\ell}|\leq C_{_{\!M,p}}{\mathfrak{D}}\e_0^\ell {\delta^{n-1}} r_1^{1-n} \;\;\; {\text{for every}\;\; \ell \in\{0, \dots, m\}}, \label{e:count good}\\ \ \medskip
&\bullet\; |\BE|\leq C_{_{\!M,p}}{\mathfrak{D}} \e_0^{m+1}{\delta^{n-1}} r_1^{1-n}\Big.. \label{e:count bad}
 \end{align}
 \end{lemma}
 
 %%%%%%%%%%%%%%%%%%%%%%%%%%%%%%%%%%%%%%%%%%%%%%%%%%%%%%%%%%%%%%%%%%%%%%%%%%%%%%%
\begin{proof}
The proof is the same as that of~\cite[Lemma 3.2]{CG19dyn}, with a very minor modification. Namely, we {replace $R_0$ by $\tilde{R}_0$ and} put $r_0=e^{-\Decay}{\tilde{R}}_0$ instead of $r_0=e^{2{\bf D}\Lambda T_0}{\tilde{R}}_0$.
We then obtain the following instead of the leftmost equation in~\cite[(3.21)]{CG19dyn}
$$
\inf_{k}R_{2,k}\geq \tfrac{1}{4}e^{-f(T_0)}\inf_iR_{1,i}.
$$
Which in turn changes the leftmost equation in~\cite[(3.22)]{CG19dyn} to 
$$
\inf_{k}R_{\ell,k}\geq e^{-\Decay}{\tilde{R}}_0=r_0.
$$
This follows from the argument below~\cite[Remark 8]{CG19dyn}, 
that yields, since $\ell \leq m$,
$$
\inf_{k}R_{\ell,k}\geq \frac{1}{4^\ell}\prod_{j=0}^\ell e^{-f(\beta_0^{-j}T_0)}R_0=\frac{1}{4^\ell} e^{-\sum_{j=0}^\ell f(\beta_0^{-j}T_0)}{\tilde{R}}_0\geq e^{-\Decay}{\tilde{R}}_0.
$$
\vspace{-1cm}

\end{proof}
%%%%%%%%%%%%%%%%%%%%%%%%%%%%%%%%%%%%%%%%%%%%%%%%%%%%%%%%%%%%%%%%%%%%%%%%%%%%%
{With Lemmas \ref{l:control} and \ref{e:control2} in place, we are now ready to prove Proposition \ref{p:non rec implies cov}.}
\begin{proof}[Proof of Proposition \ref{p:non rec implies cov}]
Let $\{\T_j(R)\}_{j \in \mc{J}(h)}=\{\Lambda_{_{\rho_j}}^\tau(R)\}_{j\in \mc{J}(h)}$ be a $(\mathfrak D, \tau, R)$ good covering of $\Sab^H$.
Let $E \in [a,b]$ and $\mc{A}\sub{E}:=\{B_{\rho_i}\}_{i=1}^K$ be the covering of $\SE^H$ as described in \eqref{e:collection A_e2}. 
{Let $\ti_0$ be as in Lemma~\ref{l:control} and fix $0<\e_0<\frac{1}{2}$}. There exists $\digamma>0$ such that each ball in $\mc{A}\sub{E}$ can be $(\e_0,\ti_0, \digamma, f)$ controlled for time $T>1$. 

We then apply Lemma \ref{e:control2} to each ball in $\mc{A}\sub{E}$. {Let $\delta_0:={R_0/2}$ be the radius of the balls in $\mc{A}\sub{E}$, and
 ${\Ti_0=\Ti_0(R)}$ such that $\Ti_0>\ti_0(\e_0)$ and}
\begin{equation}
\label{e:condR}
{ R\leq \frac{1}{10\digamma^2}e^{-\big(2\digamma \Lambda \Ti_0(R)+{\bf{F}}(\Ti_0(R))+{f(\Ti_0(R))}\big)}.}
\end{equation}
Without loss of generality, we may assume $\digamma$ is large enough so that ${\tfrac{1}{\digamma}}e^{-{\digamma}\Lambda {\ti_0(\e_0)}}\leq {\tfrac{\delta_0}{10}}$.
Then, putting $ {\tilde{R}}_0= \tfrac{1}{\digamma}e^{-{\digamma}\Lambda T_0}$ in Lemma~\ref{e:control2}, and using condition~\eqref{e:condR} allows us to set $r_1={R}$ in Lemma \ref{e:control2} and apply it to each ball $B_{\rho_0}$ in $\mc{A}\sub{E}$. Let $\tilde B_{\rho_0}$ be the ball with the same center as $B_{\rho_0}$ but with a radius ${R_{{0}}/2}$ {so} that $d(\tilde B_{\rho_0}, B_{\rho_0}^c){=}R_0/2>{\tilde{R}_0}$. {Let $\tau_0>0$, $0<\tau<\tau_0$, and set} 
$\J\sub{E}^{\rho_0}({R})=\{j \in \J\sub{E}({R}):\; \Lambda_{_{\rho_j}}^\tau(R) \cap \Lambda_{_{\tilde B_{\rho_0}}}^\tau(\tfrac{1}{5}R)\neq \emptyset\}$, there is $C_{_{\!M,p}}>0$ and sets $\{{\mc{G}\sub{E,\ell}}\}_{\ell =0}^m\subset \J\sub{E}({R})$, $\BE\subset \J\sub{E}({R})$ so that 
 $\J\sub{E}^{\rho_0}({R})\;\subset\; \BE\cup \displaystyle\cup_{\ell=0}^m{\mc{G}\sub{E,\ell}}$, and \eqref{e:nsl}, \eqref{e:count good}, \eqref{e:count bad} hold. 
 
 Therefore, letting $T_\ell= \beta_0^{-\ell} {\Ti_0}$ and $t_\ell=\ti_0(\e_0)$ for $1 \leq \ell \leq m$, and setting $\mc{G}_{m+1}:=\BE$, $T_{m+1}=t_{m+1}=1$, yields that there exists ${\Cnr=\Cnr(M,p{,\ti})>0}$ such that 
$$
R^{\frac{n-1}{2}}\sum_{\ell=0}^{m+1}\bigg(\frac{|\G_\ell| t_\ell}{T_\ell}\bigg)^{\!\!1/2}
\leq \bigg(\frac{C\sub{M,p}\mathfrak{D} \delta_0^{{n-1}}}{{\Ti_0(R)}} \sum_{\ell=0}^{m+1}{(\beta_0\e_0)^\ell}\bigg)^{\tfrac{1}{2}}\leq \frac{\Cnr \mathfrak{D}^{\frac{1}{2}}}{\sqrt{\Ti_0(R)}}.
$$
The existence of $\Cnr>0$ is justified since $\beta_0\e_0<1$.
Repeating for each ball $B_{\rho_i} \in \mc{A}\sub{E}$ {and using $K\leq C_nR_0^{1-n}$,} proves that $H$ is $\Ti_0$ non-recurrent in the window $[a,b]$ via {$\tau_0$-coverings} {with constant $\Cnr C_nR_0^{1-n}$}.

By Lemma~\ref{l:sublog}, when $\Ti$ is sub-logarithmic and $0<a<b$ we have
$
f(b)\geq \frac{b}{a}f(a).
$
In particular,
$$
\Decay=\sum_j f(2^{-j}T_0)\leq \sum_j 2^{-j}f(T_0)\leq 2f(T_0).
$$
Therefore, using $f(T)=-\log (\Ti^{-1}(T))$, there exists $c>0$ such that we may define
$$
\Ti_0(R)=c f^{-1}(\log R)\geq c \Ti(R).
$$
\end{proof}

\begin{remark}
We note that our definition of recurrence (Definition~\ref{d:non rec gral}) is equivalent to the following. There is $\digamma>0$ such that for all $\rho\in \SE^{H}$ there is $R_0>0$ such that {$B(\rho,R_0)$} is $(\e_0,\ti_0,\digamma,f)$ controlled with an additional small modification of the definition of $(\e_0,\ti_0,\digamma,f)$ controlled (see~\eqref{e:point 1} and~\eqref{e:point 3}): One needs to replace (1) by 
$$
\bigcup_{t_0\leq \pm t\leq T}\Lambda_{A_1\setminus \cup \tilde{B}_{1,k}}^\tau(r)\cap \Lambda_{A_1}^\tau(r)=\emptyset.
$$
To see these are equivalent, we identify ${B(\rho,R_0)}$ with $A_0$ and $A$ with $A_1$.

One can check that all of the proofs of being $(\e_0,\ti_0,\digamma, f)$ controlled in~\cite{CG19dyn} actually prove this slightly stronger condition with $f(T)=CT$ for some $C>0$.
\end{remark}

%%%%%%%%%%%%%%%%%%%%%%%%%%%%%%%%%%%%%%%%%%%%%%%%%%%%%%%%%%%%%%%%%%%%%%%%%%%%%%%
%%%%%%%%%%%%%%%%%%%%%%%%%%%%%%%%%%%%%%%%%%%%%%%%%%%%%%%%%%%%%%%%%%%%%%%%%%%%%%%
%%%%%%%%%%%%%%%%%%%%%%%%%%%%%%%%%%%%%%%%%%%%%%%%%%%%%%%%%%%%%%%%%%%%%%%%%%%%%%%
%%%%%%%%%%%%%%%%%%%%%%%%%%%%%%%%%%%%%%%%%%%%%%%%%%%%%%%%%%%%%%%%%%%%%%%%%%%%%%%
\section{Basic Estimates for averages over submanifolds}
\label{s:basic}
%%%%%%%%%%%%%%%%%%%%%%%%%%%%%%%%%%%%%%%%%%%%%%%%%%%%%%%%%%%%%%%%%%%%%%%%%%%%%%%
%%%%%%%%%%%%%%%%%%%%%%%%%%%%%%%%%%%%%%%%%%%%%%%%%%%%%%%%%%%%%%%%%%%%%%%%%%%%%%%
%%%%%%%%%%%%%%%%%%%%%%%%%%%%%%%%%%%%%%%%%%%%%%%%%%%%%%%%%%%%%%%%%%%%%%%%%%%%%%%
%%%%%%%%%%%%%%%%%%%%%%%%%%%%%%%%%%%%%%%%%%%%%%%%%%%%%%%%%%%%%%%%%%%%%%%%%%%%%%%
Let $P(h) \in \Psi^m(M)$ be a self-adjoint semiclassical pseudodifferential operator, with classically elliptic symbol $p$.
Throughout this section we assume $H \subset M$ is a smooth submanifold of co-dimension $k$, and $a, b\in \re$ are such that $H$ is conormally transverse for $p$ in the window $[a,b]$. 

{As explained in \S\ref{s:outline}, we crucially view the kernel of the spectral projector $\1_{[t-s,t]}(P)$ as a quasimode for $P$. We are then able to use estimates from~\cite{CG19a} to {estimate} the error when the projector is smoothed at very small scales. This section is dedicated to adapting the estimates from~\cite{CG19a} to the current setup. }

All our estimates are made in terms of $(\mathfrak{D},\tau, R(h))$-good covers and $\delta$-partitions associated to them. {For the definition of a good cover see \eqref{e:good cover}.} {Note, in addition, that there is a constant $\mathfrak{D}_n$ depending only on $n$ such that we may work with a $(\mathfrak{D}_n,\tau, R(h))$ good cover~\cite[Lemma 2.2]{CG20Lp}~\cite[Proposition 3.3]{CG19a}.} 

We now define the concept of $\delta$-partitions.
{
Let $\tau>0$, $0< \delta <\tfrac{1}{2}$, and $R(h)\geq {h^\delta}$. Let $\{\T_j\}\sub{j\in\J(h)}$ be a $(\tau,R(h))$-cover for $\Sab^H$ {with $\T_j=\Lambda_{\rho_j}^\tau (R(h))$}, and for $E\in [a,b]$ 
let $\JE(h):=\J\sub{E}(R(h))$ as defined in \eqref{e:J_E}.}
We say 
\begin{equation}\label{e: chi H}
 \{\chi\sub{\T_j}\}_{j\in \JE(h)} \subset S_\delta(T^*M;[0,1])
\end{equation}
is a \emph{$\delta$-partition for $\SE^H$} associated to $\{\T_j\}_{j\in \J(h)}$ provided the families 
 $\{\chi_j\}_{j\in \JE(h)}$ and {$ \{h^{-1}[P,\chi_j]\}_{j\in \JE(h)}$ are} bounded in $S_\delta(T^*M;[0,1])$ and 
 \begin{gather*}
 \text{(1)} \supp \chi_j \subset \Lambda_{\rho_j}^\tau(R(h)),\text{ for all }j \in \JE(h),\qquad\text{(2)}
 \underset{j\in \JE(h)}{\sum} \chi_j \geq 1\text{ on }\Lambda_{\SE^H}^{\tau/2}(\tfrac{1}{2}R(h)).
\end{gather*}
{For the construction of such a partition we refer the reader to~\cite[Proposition 3.4]{CG19a}.}

%%%%%%%%%%%%%%%%%%%%%%%%%%%%%%%%%%%%%%%%%%%%%%%%%%%%%%%%%%%%%%%%%%%%%%%%%%%%%%%

{The next lemma controls the average of $Au$ over a submanifold $H$ {in terms of the $L^2$ masses of the bicharacteristic beams intersecting the microsupport of $A$.} Here, $u$ is a quasimode for $P$ and $A$ is a pseudodifferential operator. {When we apply this lemma, $u$ will be the kernel of the spectral projector onto a small window, and $A$ will either represent a localizer to a family of tubes or differentiation in one of the coordinates}. } 

To ease notation, for ${E} \in \re$ we write $P\sub{E}=P\sub{E}(h)$
\begin{equation}\label{e:pE}P\sub{E}:=P-E.\end{equation}
In addition, given $A \in \Psi_\delta^\infty(M)$, $\psi \in C^\infty_0(\re;[0,1])$, $E\in \re$, $h>0$, $C>0$, $C\sub{N}>0$, and $u \in \mc{D}'(M)$ we set $\alpha:=\frac{k-2m+1}{2}$ and
\begin{align}\label{e:Q remainder}
 Q^{A,\psi}_{E,h}(C, C\sub{N}, u)
 := Ch^{-\frac{1}{2}-\delta}\big\|\big(1-\psi \big(\tfrac{P\sub{E}}{h^\delta}\big)\big)P\sub{E}Au\big\|\sub{H_{\scl}^\alpha} \!\!\!+C\sub{N}h^N\Big(\|u\|_\LM+ \|P\sub{E}u\|\sub{H_{\scl}^\alpha}\Big).
\end{align}

We fix $\e_0>0$ and a continuous family $[a-\e_0,b+\e_0]\ni E\mapsto B\sub{E}\in \Psi_\delta^0(M)$ such that 
\begin{equation}\label{e:B_E}
\MSh(B\sub{E})\subset \Lambda^{\tau_0+{\e_0}}_{\SE^H}(3R(h)) \qquad \text{and}\qquad \MSh(I-B\sub{E})\cap \Lambda^{\tau_0+{\e_0}}_{\SE^H}(2R(h)))=\emptyset.
\end{equation}
This will serve as a microlocalizer to the region of interest. {We {recall the} constants $\mc{K}_0$, $\tau_{\inj}$, $\FR$ defined in \eqref{e:curvature}, \eqref{e:tau_inj}, and \eqref{e:FR} respectively.}

%%%%%%%%%%%%%%%%%%%%%%%%%%%%%%%%%%%%%%%%%%%%%%%%%%%%%%%%%%%%%%%%%%%%%%%%%%%%%%%
%%%%%%%%%%%%%%%%%%%%%%%%%%%%%%%%%%%%%%%%%%%%%%%%%%%%%%%%%%%%%%%%%%%%%%%%%%%%%%%
%%%%%%%%%%%%%%%%%%%%%%%%%%%%%%%%%%%%%%%%%%%%%%%%%%%%%%%%%%%%%%%%%%%%%%%%%%%%%%%

\begin{lemma} \label{l:mainEst} 
{There exist 
$
\tau_0=\tau_0(M,p,\tau_{\inj},\FR)>0)$ and$ R_0=R_0(M,p,k, \mc{K}_0, \tau_{\inj}, \FR)>0,
$
such that the following holds. 

Let $0<\tau<\tau_0$, $0<\delta<\tfrac{1}{2}$ and ${h^\delta} \leq R(h) \leq R_0$. For $h>0$ let $\{\T_j\}_{j\in \J(h)}$ be a $(\mathfrak{D}_n, \tau, R(h))$ good cover of $\Sab^H$. Let $\mc{V}\subset S_\delta(T^*M;[0,1])$ be bounded. Let $\psi \in C^\infty_0(\re;[0,1])$ with $\psi(t)=1$ for $|t| \leq \tfrac{1}{4}$ and $\psi(t)=0$ for $|t| \geq 1$.}
 Let $\ell\in \mathbb{R}$, $\mc{W}$ and $\widetilde{\mc{W}}$ be bounded subsets of $\Psi_\delta(M)$ and $\Psi^{\ell}_\delta(M)$ respectively, and $B\sub{E}$ {be as in \eqref{e:B_E}.} 
 
 Then, there exist $C\sub{0}=C\sub{0}(n,k,\FR,\mc{V},\mc{W},\widetilde{\mc{W}})$, $C>0$, and {for all $K>0$ there is} $h_0>0$, such that for all $N>0$ there exists $C\sub{N}>0$, with the following properties. For all $u\in \mc{D}'(M)$, $0<h<h_0$, $E\in[a-Kh,b+Kh]$, every $\delta$-partition $\{\chi\sub{\T_j}\}_{j\in \JE(h)}\subset \mathcal{V}$ associated to $\{\T_j\}_{j\in \JE(h)}$, and every $A\in \widetilde{\mc{W}}$ such that $B\sub{E}\frac{1}{h}[P,A]\in \mc{W}$,
\begin{align}
\label{e:mainEst}
h^{\frac{k-1}{2}}\Big|\int_{H}Au\,\ds{H}\Big|
&\leq C\sub{0}{R(h)^{\frac{n-1}{2}}}
\sum_{j \in \IE(h)} \bigg(\frac{\|Op_h(\tilde\chi\sub{\T_j})u\|_\LM}{\tau^{\frac{1}{2}}}+\frac{C}{h}\|Op_h(\tilde \chi\sub{\T_j})P\sub{E}u\|_\LM\bigg) \notag\\
&\hspace{1cm}+ Q^{A,\psi}_{E,h}(C, C\sub{N}, u).
%&\qquad +Ch^{-\frac{1}{2}-\delta}\big\|\big(1-\psi \big(\tfrac{P\sub{E}}{h^\delta}\big)\big)P\sub{E}Au\big\|_\Hm\\
%&\qquad+C\sub{N}h^N\Big(\|u\|_\LM+ \|P\sub{E}u\|_\Hm\Big).
\end{align}
Here, $\IE(h):=\{j \in \JE(h): \T_j \cap \MSh(A) \cap \Lambda_{\SE^H}^\tau(R(h)/2) \neq \emptyset\}$, $\psi \in S_\delta \cap C^\infty_c(T^*M;[0,1])$ is any symbol with $\supp \psi \subset \big(\Lambda^\tau_{\Sigma^H\sub{E}}(2h^\delta)\big)^c$, 
and for each $j \in \JE(h)$ we let $\tilde{\chi}\sub{\T_j}$ be any symbol in $S_\delta(T^*M;[0,1])\cap C^\infty_c(T^*M;[0,1])$ such that
$\tilde{\chi}\sub{\T_j}\equiv 1 $ on $\supp\chi\sub{\T_j}$ and $\supp \tilde{\chi}\sub{\T_j}\subset \T_j.$ In addition, if ${\widetilde{\mc{W}}}\subset \Psi_0^{\ell}(M)$, then $C\sub{0}=C\sub{0}(n,k,\FR,\mc{V},\widetilde{\mc{W}})$.
\end{lemma}
\begin{proof}
{First, we prove the statement for the case $A=I$.} {Note that in this case the sets $\mc{W}$ and $\widetilde{\mc{W}}$ play no role.}
{The result for $A=I$} is a direct combination of the estimate in \cite[(3.16)]{CG19a} and \cite[Proposition 3.2]{CG19a}. 
Indeed, \cite[Proposition 3.2]{CG19a} yields the existence of $\tau_0, R_0, h_0>0$ as claimed, and the estimate \cite[(3.16)]{CG19a} yields the same bound as above, but with three modifications. 

First, the constant { $C\sub{0}=C\sub{0}(n,k, \FR)>0$} is the constant $C_{n,k}$ divided by $\FR$, because we absorb the $|H_pr_H(\rho_j)|$ factors in \cite[(3.16)]{CG19a}. 
Second, the estimate in \cite[(3.16)]{CG19a} is given for
for $\Big|\int_{H} Op_h(\beta_{\delta}) u\,\ds{H}\Big|$, where $\beta_\delta:T^*H \to \R$ is a localizer to near conormal directions defined by $\beta_\delta(x', \xi')=\chi \big(h^{-\delta}|\xi'|\sub{H} \big)$ where $\chi \in C^\infty_0(\re;[0,1])$ is a smooth cut-off with $\chi(t)=1$ for $t \leq \tfrac{1}{2}$ and $\chi(t)=0$ for $t \geq 1$. It turns out that this estimate is all we need since \cite[Proposition 3.2]{CG19a} yields that for every $N>0$ there exists $c\sub{N}>0$ such that for all $u \in\mc{D}'(M)$
\begin{equation}\label{e:lemma3.1}
\Big|\int_{H} (1-Op_h(\beta_{\delta})) u\,\ds{H}\Big| \leq c\sub{N}h^N \Big(\|u\|_{L^2(H)}+ \|P\sub{E}u\|_\Hm\Big).
\end{equation}

The third modification is that in \cite[(3.16)]{CG19a} the first error term is $Ch^{-\frac{1}{2}-\delta}\big\|P\sub{E}u\big\|\Hm$ instead of $Ch^{-\frac{1}{2}-\delta}\big\|\big(1-\psi \big(\tfrac{P\sub{E}}{h^\delta}\big)\big)P\sub{E}u\big\|\Hm$. The operator $\big(1-\psi \big(\tfrac{P\sub{E}}{h^\delta}\big)\big)$ can be added since the error term is a consequence of the bound in \cite[(3.10)]{CG19a}, and that bound is for $Op_h(\chi) u$ where $\chi$ is supported in $\{(x, \xi):\; |p\sub{E}(x, \xi)|\geq \tfrac{1}{3}h^\delta\}$. One then uses $\supp \chi \subset \supp \big(1-\psi \big(\tfrac{p\sub{E}}{h^\delta}\big) \big)$.

We note that the desired bound holds for every $\delta$-partition $\{\chi\sub{\T_j}\}_{j\in \JE(h)}\subset \mathcal{V}$ associated to $\{\T_j\}_{j\in \JE(h)}$, since the constants $C, C\sub{N}, h_0$ provided by \cite[Proposition 3.5]{CG19a} are uniform for $\chi\sub{\T_j}$ in bounded subsets of $S_\delta$.

{Given $\ep_0>0$ we note that the statement holds for every $E \in [a-\ep_0, b+\ep_0]$ since the constants $C, C\sub{N}, h_0$ provided by \cite[Proposition 3.5]{CG19a} depend on $P\sub{E}$ only through $P$. Therefore, given $K>0$, the statement for $A=I$ holds for $E \in [a-Kh, b+Kh]$ provided $h_0$ depends on $K$.}

{We now treat the case $A \neq I$.}
 Let $\mc{V},\mc{W}, \widetilde{\mc{W}}$, {and $\{B\sub{E}\}_{E\in [a-\e_0, b+\e_0]}$} be as in the assumptions. {Let $E\in [a-\e_0, b+\e_0]$.} Let $X\in \Psi_\delta(M)$ with $\MSh(I-X)\cap \Lambda_{\SE^H}^\tau(\tfrac{1}{3}R(h))=\emptyset$, {$\MSh(X)\subset \Lambda_{\SE^H}^{\tau_0+{\e_0}}(\tfrac{1}{2}R(h))$} and $B\sub{E}[P,X]\in \Psi_\delta(M)$. Then, for all $N>0$ there is $C\sub{N}>0$ depending on $\mc{V}$ 
 $$
 \Big|\int_H (I-X)Au\ds{H}\Big|\leq C\sub{N}h^N,
 $$
 so we may replace $A$ by $XA$ and assume $\MSh(A)\subset \Lambda_{\SE}^{{\tau_0+{\e_0}}}(R(h)/2)$ from now on.
{Since the estimate holds when $A=I$}, there exist $C\sub{0}=C\sub{0}(n,k,\FR)$, $C>0$, and {for all $K>0$ there is $h_0>0$} such that for all $N>0$ there exists $C\sub{N}>0$ with the following properties. For all $u\in \mc{D}'(M)$, $0<h<h_0$, $E\in[a-Kh,b+Kh]$, 
and every $\delta$-partition $\{\chi\sub{\T_j}\}_{j\in \JE(h)}\subset \mathcal{V}$ associated to $\{\T_j\}_{j\in \JE(h)}$,
the bound in \eqref{e:mainEst} holds with {$I$ in place of $A$, and with $Au$ in place of $u$:}
{\begin{align*}
h^{\frac{k-1}{2}}\Big|\int_{H}Au\,\ds{H}\Big|
&\leq C\sub{0}{R(h)^{\frac{n-1}{2}}}
\sum_{j \in \IE(h)} \bigg(\frac{\|Op_h(\tilde\chi\sub{\T_j})Au\|}{\tau^{\frac{1}{2}}}+Ch^{-1}\|Op_h(\tilde \chi\sub{\T_j})P\sub{E}Au\|\bigg) \notag\\
&\hspace{1cm}+ Q^{I,\psi}_{E,h}(C, C\sub{N}, Au).
\end{align*}}
We may sum over $j \in \IE(h)$ instead of $j \in \JE(h)$ since $\MSh(A){\cap \Lambda_{\SE^H}^\tau(\tfrac{1}{2}R(h))}\subset \cup_{j \in \IE(h)} \T_j$. 

Next, we explain how to write $u$ in place of $Au$ in each of the terms of the sum over $j \in \IE(h)$ in \eqref{e:mainEst}. To replace the term $\|Op_h({ \chi}\sub{\T_j})Au\|_\LM$ with $\|Op_h(\tilde \chi\sub{\T_j})u\|_\LM$, we use $\MSh(Op_h(\chi\sub{\T_j})A) \subset \Ell(Op_h(\tilde \chi\sub{\T_j}))$
and apply the elliptic parametrix construction to find $F_1\in \Psi_\delta(M)$ with
\begin{equation}\label{e:F_1}
 Op_h(\chi\sub{\T_j})A=
F_1Op_h(\tilde{\chi}\sub{\T_j}). 
\end{equation}
Next, to replace the term $\|Op_h({ \chi}\sub{\T_j})P\sub{E}Au\|_\LM$ with $\|Op_h(\tilde \chi\sub{\T_j})P\sub{E}u\|_\LM$,
 we decompose $$Op_h(\chi\sub{\T_j})P\sub{E}A=Op_h(\chi\sub{\T_j})[P\sub{E},A]+Op_h(\chi\sub{\T_j})AP\sub{E}$$ for each $j \in \IE(h)$, and apply the elliptic parametrix construction and find $ F_2\in \Psi_\delta(M)$ with
\begin{equation}\label{e:F_2}
h^{-1}Op_h(\chi\sub{\T_j})[P\sub{E},A]=F_2Op_h(\tilde{\chi}\sub{\T_j}).
\end{equation} 
To do this we used the assumptions: $B\sub{E}$ is microlocally the identity on $\Lambda_{{\SE^H}}^{\tau_0+{\e_0}}({2R(h)})$, $\MSh ({A}) \subset \Lambda_{{\SE^H}}^{\tau_0+{\e_0}}(\tfrac{1}{2}R(h))$,
and $A$ is such that $B\sub{{E}}\tfrac{1}{h}[P,A]\in \mc{W}\subset \Psi_\delta(M)$. This allows us to apply the parametrix construction to $Op_h(\chi\sub{\T_j})B\sub{E}\tfrac{1}{h}[P\sub{E},A]$.

Using \eqref{e:F_1} and \eqref{e:F_2}, we may 
 modify $C\sub{0}$, and having it now also depend on $A$, $\mc{V}$ and $\mc{W}$, to obtain the claim. Note that if $A \in \Psi_0^\infty(M)$, then $\tfrac{1}{h}[P\sub{E},A] \in \Psi_\delta^\infty(M)$ and so we may apply the elliptic parametrix construction to obtain \eqref{e:F_2} without the need of introducing the operator $B\sub{{E}}$ or the set $\mc{W}$. In this case, we have $C\sub{0}=C\sub{0}(n,k,\FR,\mc{V}, {\widetilde{\mc{W}}})$ as claimed.
\end{proof}
%%%%%%%%%%%%%%%%%%%%%%%%%%%%%%%%%%%%%%%%%%%%%%%%%%%%%%%%%%%%%%%%%%%%%%%%%%%%%%%
{
\begin{definition}[low density tubes]
\label{d:cheat}
Let $\{\T_j\}_{j\in \J(h)}$ be a cover by tubes of $\Sab^H$ and $0<\delta<\tfrac{1}{2}$. Let $\G(h)\subset \J(h)$ and for each $j \in {\G(h)}$ let $1<t_j(E,h)\leq T_j(E,h)$, where $h>0$ and $E\in\mathbb{R}$.

 We say $\{\T_j\}_{j \in \G(h)}$ \emph{has $\{(t_j, T_j)\}_{j\in \G(h)}$ density on $[a,b]$} if the following holds. For all $\mc{V}\subset S_\delta$ bounded, $K>0$ there is $h_0>0$ such that for all $0<h<h_0$, $E\in[a-Kh,b+Kh]$, every $\delta$-partition $\{\chi_j\}_{j\in \mc{G}\sub{E}(h)}\subset \mc{V}$ associated to $\{\T_j\}_{j \in \GE(h)}$, and all $u \in \mc{D}'(M)$, 
$$
\sum_{j\in {\GE(h)}}\|Op_h(\chi_j)u\|_\LM^2\frac{T_{j}(E,h)}{t_j(E,h)}\leq 4\|u\|_\LM^2+ 4\max_{j \in {\GE(h)}}\frac{T_{j}(E,h)^2}{h^2} \|P\sub{E}u\|^2_\LM,
$$
where $\GE(h)=\G(h)\cap \J\sub{E}(h)$.
%In addition, we ask that $h_0$ be uniform for $\chi_j$ in bounded subsets of $S_\delta(T^*M,[0,1])$.
\end{definition}

{As a} consequence of~\cite[Lemma 4.1]{CG19a} one has: if a collection of families of tubes is non self-looping for different times, then the tubes have a low density dictated by those times. 
\begin{lemma} \label{l:cheat}
Let $R_0,$ $\tau_0$, $\delta$, $R(h)$, $\tau$, and $\{\T_j\}_{j\in \J(h)}$ be as in Lemma \ref{l:mainEst}. Let $0< \alpha< 1-\limsup_{h\to 0^+} 2\tfrac{\log R(h)}{\log h}$ {and $K>0$}. There exists $h_0>0$ such that the following holds.
Let $0<h<h_0$, $E\in [a-Kh, b+Kh]$, and $\GE(h)\subset \J\sub{E}(h)$ with ${\GE(h)=\sqcup_{\ell \in \LE(h)} \mc{G}\sub{E,\ell}}(h)$. For every $\ell \in \mathcal L\sub{E}(h)$ suppose 
$t_\ell(E,h)>0$, $0<{T_\ell(E,h)}\leq 2\alpha\, T_e(h),$
\vspace{-0.1cm}
and 
$$\bigcup_{j\in {\G\sub{E,\ell}(h)}}\T_j \qquad \text{is} \;\; [t_\ell,T_\ell]\;\;\text{ non-self looping for every}\;\; \ell \in \LE(h).$$ 
Then, $\{\T_j\}_{j \in \G(h)}$ has $\{(t_j, T_j)\}_{j\in \G(h)}$ density on $[a,b]$, where 
for $0<h<h_0$, $j \in \J(h)$, and $E\in [a-Kh, b+Kh]$, we set $(t_j(E,h), T_j(E,h)):=(t_\ell(E,h), T_\ell(E,h))$ whenever $j \in \mc{G}\sub{E,\ell}(h)$.
\end{lemma}
{ We note that the statement of \cite[Lemma 4.1]{CG19a} does not provide the requisite uniformity for $E\in[a-Kh,b+Kh]$}; however, this follows from the same argument.}

{Our next estimate shows that if a family of tubes has low density, then averages of a quasimode over $H$ can be controlled in terms of the density times.}
%%%%%%%%%%%%%%%%%%%%%%%%%%%%%%%%%%%%%%%%%%%%%%%%%%%%%%%%%%%%%%%%%%%%%%%%%%%%%%%

\begin{lemma}
\label{l:basicEst}
{Let $R_0,$ $\tau_0$, $\delta$, $R(h)$, $\tau$, $\{\T_j\}_{j\in \J(h)}$, $\mc{W}$, $\widetilde{\mc{W}}$, and $\psi$ be as in Lemma \ref{l:mainEst}.} Then, there exist $C\sub{0}=C\sub{0}(n,k,p,\FR,\mc{W})$ and $C>0$, and for all $N>0$, {$K>0$ there are $h_0>0$ and} $C\sub{N}>0$, such that the following holds.

 Suppose that for all $0<h<h_0$ and $E\in [a-Kh, b+Kh]$ there exists $\GE(h)\subset \J\sub{E}(h)$ with ${\GE(h)=\sqcup_{\ell \in \mathcal L_{E}(h)} \mc{G}\sub{E,\ell}}(h)$, such that for every $\ell \in \mathcal L\sub{E}(h)$ there exist $t_\ell=t_\ell(E,h)>0$ and ${T_\ell=T_\ell(E,h)}>0$ so that, with $(t_j, T_j):=(t_\ell, T_\ell)$ for every $j \in \mc{G}\sub{E,\ell}(h)$, then
 % with $t_\ell(h)\leq T_\ell \leq {2} \alpha T_e(h)$ for $0< \alpha< 1-\limsup_{h\to 0^+} 2\tfrac{\log R(h)}{\log h}$, so that
\begin{gather*}
 %\item $\bigcup_{j\in \ya{\mc{G}\sub{E,\ell}}}\T_j$ is $[t_\ell,T_{\ell}]$ non-self looping for every $\ell \in \mathcal L_{h,E}$, \medskip
 (1)\{\T_j\}_{j \in \G(h)}\text{ has }\{(t_j, T_j)\}_{j\in \G(h)}\text{ density on $[a,b]$},\qquad
 (2) \MSh(A)\cap \Lambda_{\SE^H}^\tau(\tfrac{1}{2}R(h))\subset \, {\bigcup_{j \in \GE(h)}\T_j}.
\end{gather*}
Then, for all $u\in \mc{D}'(M)$, $0<h<h_0$, $E\in[a-Kh,b+Kh]$, and every $A\in {\widetilde{\mc{W}}}$ with $B\sub{E}\frac{1}{h}[P,A]\in \mc{W}$,
\begin{align*}
h^{\frac{k-1}{2}}\Big|\int_{H}Au\,\ds{H}\Big|
&\leq C\sub{0}R(h)^{\frac{n-1}{2}}
\sum_{\ell\in \LE(h)}\!\!\bigg(\frac{(|\G\sub{E,\ell}|t_\ell)^{\frac{1}{2}}}{\tau^{\frac{1}{2}} T_\ell^{\frac{1}{2}}}\|u\|_\LM +\frac{(|\G\sub{E,\ell}| t_\ell T_\ell)^{\frac{1}{2}}}{h}\|P\sub{E}u\|_\LM\!\!\bigg)\\
&\hspace{1cm}+ Q^{A,\psi}_{E,h}(C, C\sub{N}, u).
%&\qquad\qquad\qquad \qquad +Ch^{-\frac{1}{2}-\delta}\big\|\big(1-\psi \big(\tfrac{P\sub{E}}{h^\delta}\big)\big)P\sub{E}Au\big\|_\Hm\\
%&\qquad\qquad\qquad \qquad+C\sub{N}h^N\Big(\|u\|_\LM+ \|P\sub{E}u\|_\Hm\Big).
\end{align*}
In addition, if ${\widetilde{\mc{W}} \subset} \Psi_0^\infty(M)$, the estimate holds with $C\sub{0}=C\sub{0}(n,k,p,\FR,{\widetilde{\mc{W}}})$.
\end{lemma}
%%%%%%%%%%%%%%%%%%%%%%%%%%%%%%%%%%%%%%%%%%%%%%%%%%%%%%%%%%%%%%%%%%%%%%%%%%%%%%%
\begin{proof}
Let $\mc{V}$ a bounded subset of $S_\delta(T^*M;[0,1])$.
By Lemma~\ref{l:mainEst} 
there exist $C\sub{0}=C\sub{0}(n,k,\FR,\mc{V},\mc{W})$, $C>0$, and $h_0>0$, such that for all $N>0$ there exist $C\sub{N}>0$, with the following properties. For all $u\in \mc{D}'(M)$, $K>0$, $0<h<h_0$, $E\in[a-Kh,b+Kh]$, 
and every $\delta$-partition $\{\chi\sub{\T_j}\}_{j\in \JE(h)}\subset \mathcal{V}$ associated to $\{\T_j\}_{j\in \JE(h)}$,
\begin{align*}
h^{\frac{k-1}{2}}\Big|\int_{H}Auds{H}\Big|
&\leq C\sub{0}{R(h)^{\frac{n-1}{2}}}
\!\!\!\sum_{j \in \IE(h)} \!\!\!\bigg(\frac{\|Op_h(\tilde\chi\sub{\T_j})u\|_{L^2}}{\tau^{\frac{1}{2}}}+\frac{C}{h}\|Op_h(\tilde \chi\sub{\T_j})P\sub{E}u\|_{L^2}\!\!\bigg)%\\
%&\hspace{1cm}+
\!\!+Q^{A,\psi}_{E,h}(C,\! C\sub{N}, \!u),
%&\qquad\qquad +Ch^{-\frac{1}{2}-\delta}\big\|\big(1-\psi \big(\tfrac{P\sub{E}}{h^\delta}\big)\big)P\sub{E}Au\big\|_\Hm\\
%&\qquad\qquad +C\sub{N}h^N\Big(\|u\|_\LM+ \|P\sub{E}u\|_\Hm\Big),
\end{align*}
where $\IE(h):=\bigcup_{\ell \in \mathcal L_{h,E}}{\mc{G}\sub{E,\ell}}$.
Note that if $A\in \Psi_0^\infty(M)$, then the estimate holds with $C\sub{0}=C\sub{0}(n,k,p, \FR, \mc{V}, {\widetilde{\mc{W}}})$.
Next, note that 
$$
\sum_{j\in \IE(h)}\|Op_h(\tilde{\chi}\sub{\T_j})P\sub{E}u\|\leq |\JE(h)|^{\frac{1}{2}} \Big(\sum_{j\in \JE(h)}\|Op_h(\tilde{\chi}\sub{\T_j})P\sub{E}u\|^2\Big)^{\frac{1}{2}},
$$
and so, since {$|\JE(h)| \leq C_n\vol({\SE^H})R(h)^{{1-n}}$} for some $C_n>0$,
we have, after adjusting $C>0$, that for all $0<h<h_0$
\begin{align}\label{e:aadvark}
h^{\frac{k-1}{2}}\Big|\int_{H}Au\,\ds{H}\Big|
&\leq C\sub{0}\frac{R(h)^{\frac{n-1}{2}}}{\tau^{\frac{1}{2}}}\!\!\sum_{j\in \IE(h)} \!\|Op_h(\tilde{\chi}\sub{\T_j}) u\|_\LM \!\!+ \frac{C}{h}\|P\sub{E}u\|_\LM 
%&\hspace{1cm}
\!+ Q^{A,\psi}_{E,h}(C,\! C\sub{N},\! u).
%&\qquad\qquad +Ch^{-\frac{1}{2}-\delta}\big\|\big(1-\psi \big(\tfrac{P\sub{E}}{h^\delta}\big)\big)P\sub{E}Au\big\|_\Hm\notag\\
%&\qquad\qquad +C\sub{N}h^N\Big(\|u\|_\LM+ \|P\sub{E}u\|_\Hm\Big),
\end{align}

Since we are working with a $({\mathfrak D_n},\tau,R(h))$-good cover, we split each ${\mc{G}\sub{E,\ell}}$ into $\mathfrak D_n$ families $\{\mc{G}\sub{E,\ell, i}\}_{i=1}^{\mathfrak{D}_n}$ of disjoint tubes.
Note that 
\begin{equation}\label{e:ground}
\sum_{j\in \IE(h)}\|Op_h(\tilde\chi_j)u\|_\LM
\leq \sum_{\ell\in \mathcal L}{\sum_{i=1}^{\mathfrak{D}_n} \sum_{j\in \mc{G}\sub{E,\ell,i}}}\|Op_h(\tilde\chi_j)u\|_\LM.
\end{equation}

Next, since $\{\T_j\}_{j \in \G(h)}$ has $\{(t_j, T_j)\}_{j\in \G(h)}$ density on $[a,b]$, after possibly shrinking $h_0$ {(depending on the $S_\delta$ bounds for $\tilde{\chi}_j$ and $K>0$),} Cauchy-Schwarz yields that for all $0<h<h_0$
\begin{equation}\label{e:ant}
{\sum_{j \in \mc{G}\sub{E,\ell,i}}}\|Op_h(\tilde \chi_j)u\|_\LM
\leq 2\Big(\frac{t_\ell |{\mc{G}\sub{E,\ell}}|}{T_\ell}\Big)^{\frac{1}{2}}\Big(\|u\|_\LM^2+ \frac{T_{\ell}^2}{h^2}\, \|P\sub{E}u\|^2_\LM\Big)^{\frac{1}{2}}.
\end{equation}
\indent The result follows from combining \eqref{e:ant} and \eqref{e:ground}, and feeding this to \eqref{e:aadvark}. Note that $C\sub{0}$ needs to be modified, but only in a way that depends on $n$ via $\mathfrak D_n$.
\end{proof}

%%%%%%%%%%%%%%%%%%%%%%%%%%%%%%%%%%%%%%%%%%%%%%%%%%%%%%%%%%%%%%%%%%%%%%%%%%%%%%%
We also need the following basic estimate for averages over submanifolds to control averages of $u=\mathbf 1_{(-\infty, s]}(P)$ {when $s$ is large}.

\begin{lemma}
\label{l:SobEst}
Suppose $H\subset M$ is a submanifold of codimension $k$ and $P\in \Psi^m(M)$, with $m>0$, is such that there exists $C>0$ for which
$$
|\sigma(P)(x,\xi)|\geq |\xi|^m/C,\qquad (x,\xi)\in N^*H,\qquad |\xi|\geq C.
$$ 
Let $\psi \in S^0(T^*M; [0,1])$ with $\psi\equiv 1$ on $N^*H$, and let $\ell\in \mathbb{R}$. Let $A\in \Psi_\delta^{\ell}(M)$ and $r>\frac{k+2\ell}{2m}$.
Then, there are $C\sub{0}>0$ and $h_{0}>0$ such that for all $N>0$ there is $C\sub{N}>0$ satisfying 
$$
h^{\frac{k}{2}}\Big|\int_H \!\!\!Aud\sigma\sub{H}\Big|\leq C\sub{0}\Big( \|Op_h(\psi)u\|_{\LM}+\|Op_h(\psi)P\sub{E}^ru\|_{\LM}\Big)+C\sub{N}h^N\|u\|_{H_{\textup{scl}}^{-N}(M)},\qquad 0<h<h_0.
$$
\end{lemma}
%%%%%%%%%%%%%%%%%%%%%%%%%%%%%%%%%%%%%%%%%%%%%%%%%%%%%%%%%%%%%%%%%%%%%%%%%%%%%%%
\begin{proof} 
 Let $\tilde \psi \in S^0(T^*M; [0,1])$ with $\tilde{\psi}\equiv 1$ on $N^*H$, $\supp \tilde{\psi}\subset \{\psi\equiv 1\}$, and such that 
 $$
 |\sigma(P\sub{E})(x,\xi)|\geq \tfrac{1}{C}|\xi|^m,\qquad (x,\xi)\in \supp \tilde{\psi},\qquad |\xi|\geq C.
 $$
 Then, since $\WFh(\delta_{H})=N^*\!H$, for any $N>0$ there is $C\sub{N}>0$ such that 
\begin{equation} 
\label{e:blue}
\Big|\int_HAOp_h(1-\tilde{\psi})ud\sigma\sub{H}\Big|\leq C\sub{N}h^N\|u\|_{H_{\textup{scl}}^{-N}(M)}.
\end{equation}
Next, by the Sobolev embedding theorem, for any $\ep>0$ there exists $C\sub{0}>0$ such that 
\begin{align*}
\Big| \int_{H}AOp_h(\tilde{\psi})u d\sigma\sub{H}\Big|&\leq C\sub{0}h^{-\frac{k}{2}}\|Op_h(\tilde{\psi})u\|_{H_{\textup{scl}}^{\frac{k}{2}+\ep+\ell}(M)}.
\end{align*}
Taking $r$ with $rm>\frac{k}{2}+\ell$ and using an elliptic parametrix, for any $N>0$ there is $C\sub{N}>0$ with 
\begin{align}
h^{\frac{k}{2}}\Big| \int_{H}AOp_h(\psi)u d\sigma\sub{H}\Big|\leq C\sub{0}&\|Op_h(\tilde{\psi})u\|_{H_{\textup{scl}}^{rm}(M)}\leq C\sub{0} \big(\|Op_h(\psi)u\|_{\LM}+\|Op_h(\psi) P\sub{E}^ru\|_{\LM}\big)\notag
 \\&+C\sub{N}h^N\|u\|_{H_{\textup{scl}}^{-N}(M)}.\label{e:brown}
\end{align}
Indeed, this follows from letting $\chi \in S^0(T^*M;[0,1])$ so that $ |\sigma(P\sub{E})(x,\xi)|\geq \tfrac{1}{C}|\xi|^m$ in the support of {$\tilde \psi (1-\chi)$}, and then using the elliptic parametrix construction to find $F_1, F_2 \in \Psi^0(M)$ such that
\begin{gather*}
\langle hD\rangle^{rm} Op_h(\tilde \psi)(1-Op_h(\chi))=F_1 Op_h(\psi)P\sub{E}^r + O(h^\infty)_{\Psi^{-\infty}},\\
\langle hD\rangle^{rm} Op_h(\tilde \psi)Op_h(\chi)=F_2 Op_h(\psi) + O(h^\infty)_{\Psi^{-\infty}}.
\end{gather*}
Combining with~\eqref{e:blue} and~\eqref{e:brown} completes the proof.
\end{proof}

%%%%%%%%%%%%%%%%%%%%%%%%%%%%%%%%%%%%%%%%%%%%%%%%%%%%%%%%%%%%%%%%%%%%%%%%%%%%%%%
%%%%%%%%%%%%%%%%%%%%%%%%%%%%%%%%%%%%%%%%%%%%%%%%%%%%%%%%%%%%%%%%%%%%%%%%%%%%%%%
%%%%%%%%%%%%%%%%%%%%%%%%%%%%%%%%%%%%%%%%%%%%%%%%%%%%%%%%%%%%%%%%%%%%%%%%%%%%%%%
\section{Lipschitz Scale for Spectral Projectors} 
%%%%%%%%%%%%%%%%%%%%%%%%%%%%%%%%%%%%%%%%%%%%%%%%%%%%%%%%%%%%%%%%%%%%%%%%%%%%%%%
%%%%%%%%%%%%%%%%%%%%%%%%%%%%%%%%%%%%%%%%%%%%%%%%%%%%%%%%%%%%%%%%%%%%%%%%%%%%%%%
%%%%%%%%%%%%%%%%%%%%%%%%%%%%%%%%%%%%%%%%%%%%%%%%%%%%%%%%%%%%%%%%%%%%%%%%%%%%%%%
\label{s:lip}
{In this section we estimate the scale at which averages of the spectral projector behave like Lipschitz functions {of the spectral parameter}, and use this to approximate $\Pi_h$ using $\rho_{h,T(h)}*\Pi_h$.}

Throughout this section we assume $H_1,H_2 \subset M$ are two smooth submanifolds of co-dimension $k_1$ and $k_2$ respectively. The goal for this section is to prove the following proposition.
\begin{proposition}
\label{p:error-smooth}
Suppose $a, b\in \re$ such that $H_1, H_2$ are uniformly conormally transverse for $p$ in the window $[a,b]$. Let $\tau_0, R_0$ be as in Lemma \ref{l:mainEst}.
Let $0<\tau<\tau_0$ and $0<\delta<\tfrac{1}{2}$.
For $i=1,2$, let $\Ti_i$ be sub-logarithmic \resfuns\, with $\Omega(\Ti_i)\Lambda<1-2\delta$ {and suppose} $H_i$ is ${\Ti_i}$ non-recurrent in the window $[a,b]$ via {$\tau$-coverings} {with constant $\Cnr^i$}.

Let $A_1,A_2\in \Psi^\infty(M)$, ${K>0}$, $ R(h)\geq {h^\delta}$, and $\Ti:=\sqrt{\Ti_1\Ti_2}$. Then, there exist $h_0>0$ and $$C\sub{0}=C\sub{0}(n, k_1, k_2, \FR^1,\FR^2, A_1, A_2, \Cnr^1, \Cnr^2)>0,$$
such that for all $0<h\leq h_0$ and $E\in[a-Kh,b+Kh]$,
$$
\Big|\HAs(E)-\rho\sub{h,T_{\max}(h)}*\HAs(E)\Big|\leq C\sub{0}h^{\frac{2-k_1-k_2}{2}}\Big/\Ti(R(h)).
$$
\end{proposition}

\begin{remark}
{To ease notation, throughout this section we write $T_i(h):=\Ti_i(R(h))$, $T(h):=\Ti(R(h))$, and
${T_{\max}(h):=\max(\Ti_1(R(h),\Ti_2(R(h))))}.$
}
\end{remark}

%%%%%%%%%%%%%%%%%%%%%%%%%%%%%%%%%%%%%%%%%%%%%%%%%%%%%%%%%%%%%%%%%%%%%%%%%%%%%%%
\begin{proof}
We split the proof into Lemmas \ref{l:taub},~\ref{l:tiramisu}, and~\ref{l:bigS} below. 
 Lemmas \ref{l:tiramisu} and~\ref{l:bigS} show that {there exist $C\sub{0}=C\sub{0}(n, k_1, k_2, \FR^1,\FR^2, A_1, A_2, \Cnr^1, \Cnr^2)>0$, $C_1>0$, and $h_0>0$ such that}
{$w_h(E):=\HAs(E)$ satisfies the hypotheses of Lemma~\ref{l:taub} with $I_h:= [a-Kh,b+Kh]$, $\rho_h:=\rho\sub{h,T_{\max}(h)}$, {$\sigma_h:= T_{\max}(h)/h$},
$$L_h:=C\sub{0}h^{\frac{2-k_1-k_2}{2}}\Big/T(h)\qquad \text{and}\qquad B_h:= C\sub{1}{h^{-\frac{k_1+k_2}{2}}},$$
{and $0<h<h_0$.}
{Next,} let $\{K_j\}_{j=1}^\infty\subset \re_+$ be given by the choice of $\rho$ in \eqref{e:rho def}.
Since
$\Big\langle \frac{T_1(h)s}{h}\Big\rangle^{\frac{1}{2}}\Big\langle \frac{T_2(h) s}{h}\Big\rangle^{\frac{1}{2}}\leq \langle \sigma\sub{h} s\rangle$ 
for all $s\in \re$, Lemma \ref{l:taub} yields that there exists $C\sub{\rho}>0$ and for all $N>0$ there exists $C\sub{N}>0$ such that 
$$
\Big|\HAs(E)-\rho\sub{h,T(h)}*\HAs(E)\Big|\leq C\sub{\rho} C\sub{0} \frac{h^{\frac{2-k_1-k_2}{2}}}{T(h)}+ C\sub{N} C\sub{1}{h^{-\frac{k_1+k_2}{2}}} \Big(\frac{h}{T_{\max}(h)}\Big)^N,
$$
{for all $0<h<h_0$.}
This completes the proof after choosing $h_0$ small enough.}
\end{proof}

{We now present the lemmas used in the proof of Proposition \ref{p:error-smooth}. The first shows that if a family of functions $\{w_h\}_h$ is Lipstchitz {at scale $\sigma_h^{-1}$} with (at most) polynomial growth at infinity, then the family can be well approximated by its convolution $\rho_{h} *w_h$ where $\{\rho_h\}_h$ is a family of Schwartz functions}

%%%%%%%%%%%%%%%%%%%%%%%%%%%%%%%%%%%%%%%%%%%%%%%%%%%%%%%%%%%%%%%%%%%%%%%%%%%%%%%
\begin{lemma}
\label{l:taub}
Let $\{K_j\}_{j={0}}^\infty\subset \re_+$. Then, there exists $C>0$ and for all $N_0\in \mathbb{R}$, $N>0$ there exists $C\sub{N}>0$, such that the following holds. Let $\{\rho_{h}\}_{h>0}\subset \mathcal{S}(\mathbb{R})$ be a family of functions and $\{\sigma_h\}_{h>0}\subset \re_+$ such that for all $j\geq 1$ {and $h>0$}, 
$$
|\rho_h(s)|\leq {\sigma_h K_j}\,{\langle \sigma_h s\rangle^{-j}}\qquad \text{for all}\; s \in \re.
$$
Let $\{L_h\}_{h>0}\subset \re_{+}$, $\{B_h\}_{h>0}\subset \re_{+}$, 
 $\{w_h:\re \to \re\}_{h>0}$, $I_h\subset [-K_0,K_0]$, {$h_0>0$} and $\e_0>0$, be so that for all {$0<h<h_0$}
\begin{itemize}
 \item $|w_h(t-s)-w_h(t)|\leq L_h\langle \sigma_h \,s\rangle$ for all $t\in I_h$ and ${|s|\leq \e_0}$, \smallskip
 \item $|w_h(s)|\leq B_h\langle s\rangle^{N_0}$ for all $s\in \mathbb{R}$.
\end{itemize}
Then, for all {$0<h<h_0$} and $t\in I_h$
$$
\Big|(\rho_{h} *w_h)(t)-w_h(t)\int_{\re} \rho_{h}(s) ds\Big|\leq C L_h+ C\sub{N} B_h \sigma_h^{-N}\e_0^{-N}.
$$
\end{lemma}
%%%%%%%%%%%%%%%%%%%%%%%%%%%%%%%%%%%%%%%%%%%%%%%%%%%%%%%%%%%%%%%%%%%%%%%%%%%%%%%

\begin{proof}
For all {$0<h<h_0$} and $t\in I_h$
\begin{align*}
&\Big|(\rho_h *w_h)(t)-w_h(t)\int_{\re} \rho_h(s)ds\Big|
=\Big|\int_{\re} \rho_h(s)\big(w_h(t-s)-w_h(t)\big)ds\Big|\\
& \leq L_h\int_{|s|\leq \e_0} |\rho_h(s)|\langle \sigma_h s\rangle ds
+B_h\int_{|s|\geq \e_0}|\rho_h(s)|\Big( \langle t-s\rangle^{N_0}+\langle t\rangle^{N_0}\Big)ds\\
&\leq L_h\int_{|s|\leq \e_0}\!\!\! \sigma_h K_{3} \langle \sigma_h s\rangle^{-2} ds
+B_h\!\!\int_{|s|\geq \e_0} \!\!\!K\sub{N_0+2+N}\sigma_h \langle \sigma_h s\rangle^{-(N_0+2+N)}\!\Big( \langle t-s\rangle^{N_0}+\!\langle t\rangle^{N_0}\Big)ds.
\end{align*}
The existence of $C$ and $C\sub{N}$ follows from integrability of each term {and the boundedness of $I_h$}.
\end{proof}
%%%%%%%%%%%%%%%%%%%%%%%%%%%%%%%%%%%%%%%%%%%%%%%%%%%%%%%%%%%%%%%%%%%%%%%%%%%%%%%

{The next lemma shows that the family of functions $w_h(t)=\HAs(t)$ is Lipschitz {at scales dictated by the non-recurrence times for $H_1$ and $H_2$.}}

%%%%%%%%%%%%%%%%%%%%%%%%%%%%%%%%%%%%%%%%%%%%%%%%%%%%%%%%%%%%%%%%%%%%%%%%%%%%%%%%

\begin{lemma}
\label{l:tiramisu}
Suppose $a,b\in \re$, $\e_0>0$ are such that $H_1$, $H_2$ are conormally transverse for $p$ in the window $[a-\e_0,b+\e_0]$. {Let $A_1$, $A_2$, $\tau_0$, $R_0$, $\tau$, $\delta$, $R(h)$, and $\alpha$ be as in Proposition \ref{p:error-smooth}.} Let $\Cnr>$ and $K>0$. Then, there exist $h_0>0$ and 
$$
C\sub{0}=C\sub{0}(n, k_1, k_2, \FR^1,\FR^2, A_1, A_2, \Cnr)>0
$$
such that the following holds.

For $i=1,2$, let {$\Ti_i$ be a sub-logarithmic {\resfun} with $\Omega(\Ti_i)\Lambda<1-2\delta$. } Suppose $H_i$ is $\Ti_i$ non-recurrent in the window $[a,b]$ via {$\tau$-coverings} {with constant $\Cnr^i{\leq \Cnr}$}. Then for all $0<h\leq h_0$, $|s|\leq \e_0$, and $t\in [a-Kh,b+Kh]$, 
$$
\Big|\HAs(t)-\HAs(t-s)\Big|\leq C\sub{0}\frac{h^{\frac{2-k_1-k_2}{2}}}{\sqrt{T_1(h)T_2(h)}}\Big\langle \frac{T_1(h)s}{h}\Big\rangle^{\frac{1}{2}}\Big\langle \frac{T_2(h) s}{h}\Big\rangle^{\frac{1}{2}}.
$$

\end{lemma}

%%%%%%%%%%%%%%%%%%%%%%%%%%%%%%%%%%%%%%%%%%%%%%%%%%%%%%%%%%%%%%%%%%%%%%%%%%%%%%%%%%%%
\begin{proof}
{We first assume the statement for $|s| \leq 2h$.}
Suppose $s\geq 2h$. The case of $s\leq -2h$ being similar.
Define $k_0:=\lfloor\frac{s}{h}\rfloor$ and
{$t_k:=t-s+kh$ for $0\leq k\leq k_0-1,$ and $t_k:=t$ for $k=k_0$.} Then
$$
\HAs(t)-\HAs(t-s)=\sum_{k=0}^{k_0-1}\HAs(t_{k+1})-\HAs(t_{k}).
$$
{Using} $|t_{k+1}-t_k|\leq 2h$, {and putting $t=t_{k+1}$, $s=t_{k+1}-t_k$,} we apply the {case $|s|\leq 2h$} with $T_1=T_2=1$ for each term to obtain
\begin{align*}
\Big|\HAs(t)-\HAs(t-s)\Big|
%&\leq \sum_{k=0}^{\ya{k_0-1}}|\HAs(t_{k+1})-\HAs(t_{k})|\\
&\leq C\sub{0}k_0 h^{\frac{2-k_1-k_2}{2}}\leq C\sub{0}h^{\frac{2-k_1-k_2}{2}} |s/h|,
\end{align*}
{and this proves the claim provided the statement holds for $|s| \leq 2h$.}

{We proceed to prove the statement for $|s| \leq 2h$.}
First, note that by \eqref{e:efxn decomp} and Cauchy-Schwarz
\begin{align} \label{e: purple}
&\Big|\HAs(t)-\HAs(t-s)\Big|^2
%=\Bigg|\sum_{t-s\leq E_j\leq t} \int_{H_1}A_1 \phi\sub{E_j}\, \ds{H_1}\int_{H_2} A_2 \phi\sub{E_j} \ds{H_2}\Bigg|\notag\\
\leq \sum_{t-s\leq E_k\leq t} 
\Big|\int_{H_1}A_1\phi\sub{E_k}\ds{H_1}\Big|^2\,\,\cdot\,\,
\sum_{t-s\leq E_j\leq t}\Big|\int_{H_2} A_2\phi\sub{E_j} \ds{H_2}\Big|^2.
\end{align}
Now, for each $i=1,2,$
 \begin{equation}\label{e:pink}
 \sum_{t-s\leq E_j\leq t}\Big|\int_{H_i}\!\!A_i\phi\sub{E_j}\ds{H_i}\Big|^2
 %=\int_{H_i}\gamma\sub{H_i}A_i\1_{[t-s,t]}(P)\1_{[t-s,t]}(P)\, A_i^*\, \delta\sub{H_i}\ds{H_i}
 =\|\1_{[t-s,t]}(P)\, A_i^*\, \delta\sub{H_i}\|^2\sub{L^2(M)}=\!\!\!\!\!\sup_{\|w\|_\LM=1}\Big| \int_{H_i}\!\!A_i\1_{[t-s,t]}(P) w\,\ds{H_i}\Big|^2,
 \end{equation}
where $\delta\sub{H_i}$ is the delta distribution at $H_i$ and the last equality follows by duality.
%{Observe that {by duality}} 
%\begin{align}\label{e: norm of average} 
%\|\1_{[t-s,t]}(P)\, A_i^*\, \delta\sub{H_i}\|\sub{L^2}
%&=\sup_{\|w\|_\LM=1} |\langle \1_{[t-s,t]}(P)\, A_i^*\, \delta\sub{H_i}, w \rangle_\LM|\notag\\
%&=\sup_{\|w\|_\LM=1}\Big| \int_{H_i}A_i\1_{[t-s,t]}(P) w\,\ds{H_i}\Big|.
%\end{align}
%where $\gamma\sub{H_i}:M \to H_i$ is the restriction function to $H_i$.
 
{We now} use the non-recurrence assumption on $H_1$ and $H_2$.
Since for each $i=1,2$, the submanifold {$H_i$ is $\Ti_i$ non-recurrent in the window $[a,b]$ via {$\tau_0$-coverings}}, {there is $h_0>0$ small enough depending on $R(h),K$ so that} for all {$0<h<h_0$ and} $t \in {[E-Kh,E+Kh]}$ there is a partition of indices $\J\sub{t}^i(h)=\cup_{\ell \in \L\sub{t}^i(h)} \G\sub{t,{\ell}}^i(h)$, and times $\{T_\ell^i(h)\}_{\ell \in \L\sub{t}^i(h)}$, and $\{t_\ell^i(h)\}_{\ell \in \L\sub{t}^i(h)}$ as in Definition \ref{d:non rec cov}.

Note that {we have chosen $h_0$ small enough so that} $\JE^i(h)$ is a $(\tau,R(h))$ good covering of $\Sigma^{H_i}_t$ for $t \in [E-Kh,E+Kh]$. 
In particular, for $i=1, 2$ and $t\in [E-Kh,\,E+Kh]$
\begin{gather}\label{e: non rec conds}
R(h)^{\frac{n-1}{2}}\sum_{\ell\in\LE^i(h)} \frac{(|\G\sub{t,{\ell}}^i|t_\ell^i)^{\frac{1}{2}} }{{(T_\ell^i)}^{\frac{1}{2}}}
\leq \frac{\Cnr^i}{T_i^{\frac{1}{2}}},
\qquad 
R(h)^{\frac{n-1}{2}}\sum_{\ell\in\LE^i(h)} (|\G\sub{t,{\ell}}^i|t_\ell^i)^{\frac{1}{2}}(T_\ell^i)^{\frac{1}{2}}\leq \Cnr^i T_i^{\frac{1}{2}}.
\end{gather}
The first bound is condition $(2)$ in Definition \ref{d:non rec cov}, and the second bound follows from the first one together with the $T_\ell^i \leq T_i$ for all $\ell \in \mc{L}^i_{h,E}$.
Next, for $\ell \in \mc{L}^i_{E}$ let 
\begin{equation}
\label{e:modifiedTs}\tilde T_{\ell}^i(h):=\begin{cases}T_{\ell}^i(h) \big\langle \tfrac{T_i(h) s}{h}\big \rangle^{-1}&t_\ell^i\leq T_\ell^i\big\langle \tfrac{T_i(h)s}{h}\big\rangle^{-1}\\
1&\text{else}
\end{cases},\quad\tilde {t}_{\ell}^i(h):=\begin{cases}t_{\ell}^i(h) &t_\ell^i\leq T_\ell^i\big\langle \tfrac{T_i(h)s}{h}\big\rangle^{-1}\\
1&\text{else}
\end{cases}\end{equation}
and note that 
$
\sum_{\tilde{t}_\ell^i=\tilde{T}_\ell^i=1}|\G\sub{t,{\ell}}^i|^{\frac{1}{2}}\leq \Cnr^i \sqrt{\frac{1}{T_i}\Big\langle \frac{T_i s}{h}\Big\rangle}.
$
In particular, 
\begin{equation}
 \label{e:modifiedTEsts}
\begin{gathered}
\!\!\!\!\sum_{\ell \in \LE^i(h)}\!\!\frac{(|\G\sub{t,{\ell}}^i|\tilde{t}_\ell^i)^{\frac{1}{2}}}{(\tilde{T}_\ell^i)^{\frac{1}{2}}}\leq 2\Cnr^i \sqrt{\frac{1}{T_i}\Big\langle \frac{T_i s}{h}\Big\rangle},
\qquad
\sum_{\ell \in \LE^i(h)}\!\!\sqrt{|\G\sub{t,{\ell}}^i|\tilde{t}_\ell^i\tilde{T}_\ell^i}\leq 2\Cnr^i \Big(\frac{1}{T_i}\Big\langle \frac{T_i s}{h}\Big\rangle\Big)^{-\frac{1}{2}}.
\end{gathered}
\end{equation}
Then, since for each $\ell \in \LE^i(h)$ the union of tubes with indices in ${\mc{G}^i\sub{E,\ell}}$ is also $[\tilde{t}_\ell^i(h),\tilde T_{\ell}^i(h)]$ non-self looping, we may apply Lemma \ref{l:cheat} with the sets $\{\G\sub{t,{\ell}}^i(h)\}_{\ell \in \LE^i(h)}$, $\{\tilde T_\ell^i(h)\}_{\ell \in \LE^i(h)}$, $\{t_\ell^i(h)\}_{\ell \in \LE^i(h)}$ {to see that $\{\T_j\}_{j\in \G\sub{t,{\ell}}^i(h)}$ has $\{(t_j,T_j)\}$ density on $[a,b]$ where $t_j=\tilde{t}_j^i(h)$, $T_j=\tilde{T}_j^i(h)$}. Then, using Lemma~\ref{l:basicEst} with operators $A_i \in \Psi^\infty(M)$, $\psi \in C^\infty_0(\re;[0,1])$ with $\psi(t)=1$ for $|t| \leq \tfrac{1}{4}$ and $\psi(t)=0$ for $|t| \geq 1$, and for $s \in \re$ let $u=\1_{[t-s,t]}(P)w$, where $w$ is any function in $L^2(M)$ with $\|w\|_\LM=1$.
Next, by Lemma~\ref{l:basicEst}, for $i=1,2$, there exist $C\sub{0}^i=C\sub{0}(n,k_i,\FR^i, A_i)$, $C>0$, and for all $N$ there is $C\sub{N}>0$ such that for all $0<h<h_0$, $s\in \mathbb{R}$, and $t \in [E-Kh, E+Kh]$
\begin{align}\label{e:dracula}
&h^{\frac{k_i-1}{2}}\Big|\int_{H_i}A_i\1_{[t-s,t]}(P)w\,\ds{H_i}\Big|\notag
\leq C\sub{0}^iR(h)^{\frac{n-1}{2}}\!\!\!\!
\sum_{\ell\in \LE^i(h)}\!\!\frac{(|\G\sub{t,{\ell}}^i|{\tilde{t}}_\ell^i)^{\frac{1}{2}}}{(\tau \tilde T_\ell^i)^{\frac{1}{2}}}\|\1_{[t-s,t]}(P)w\|_\LM \\
&\,+C\sub{0}^iR(h)^{\frac{n-1}{2}}\!\!\!\!
\sum_{\ell\in \LE^i(h)}\frac{(|\G\sub{t,{\ell}}^i| \tilde{t}_\ell^i \tilde T_\ell^i)^{\frac{1}{2}}}{h}\|P_t \1_{[t-s,t]}(P)w\|_\LM+ Q^{A,\psi}_{t,h}(C, C\sub{N}, \1_{[t-s,t]}(P)w).
%&\qquad\qquad\qquad \qquad +Ch^{-\frac{1}{2}-\delta}\big\|\big(1-\psi \big(\tfrac{P_t}{h^\delta}\big)\big)P_tA_i\1_{[t-s,t]}(P)w\big\|_\Hm \notag\\
%&\qquad\qquad\qquad \qquad+C\sub{N}h^N\Big(\|\1_{[t-s,t]}(P)w\|_\LM+ \|P_t \1_{[t-s,t]}(P)w\|_\Hm\Big).
\end{align}

Note that for all $N$ there is $C\sub{N}>0$ such that for all $t\in[a{-Kh},b{+Kh}]$, $|s|\leq 10$ and $0<h<1$
\begin{equation}\label{e:bat}
\|P_t\1_{[t-s,t]}(P)\|_{L^2\to H_{\textup{scl}}^N}\leq C\sub{N}|s|,\qquad \|\1_{[t-s,t]}(P)\|_{L^2\to L^2}\leq 1.
\end{equation}
In addition, we use the elliptic parametrix construction, together with $|s|\leq 2h$ to obtain
\begin{equation}\label{e:coffin}
\|\big(1-\psi \big(\tfrac{P_t}{h^\delta}\big)\big)P_tA_i\1_{[t-s,t]}(P)\|_{L^2\to H_{\textup{scl}}^N}\leq C\sub{N} h^N.
\end{equation}
%find $F_i\in \Psi_\delta^\infty(M)$ such that 
%$$
%\big(1-\psi \big(\tfrac{P_t}{h^\delta}\big)\big)P_tA_i=F_iP_t,
%\qquad\qquad 
%\|F_iP_t\1_{[t-s,t]}(P)\|_{L^2\to H_{\textup{scl}}^N}\leq C\sub{N}h^\delta|s|.
%$$
We combine these estimates with \eqref{e: non rec conds} and the definition of $\tilde T_\ell^i$ into \eqref{e:pink} to obtain that for all $0<h<h_0$, $|s|\leq 2h$, $K>0$, and $t \in [E-Kh, E+Kh]$,
\begin{align*}
&h^{\frac{k_i-1}{2}}\|\1_{[t-s,t]}(P)\, A_i^*\, \delta\sub{H_i}\|\LM
\leq C\sub{0}^i\Cnr^i
\bigg( \frac{1}{\tau^{\frac{1}{2}}}\Big(\frac{1}{ T_i} \Big\langle \frac{T_i s}{h} \Big\rangle \Big)^\frac{1}{2} + \frac{|s|}{h} \Big(\frac{1}{ T_i} \Big\langle \frac{T_i s}{h} \Big\rangle \Big)^{\!-\frac{1}{2}}\bigg) +C\sub{N}h^N.
\end{align*}

In particular, since $\tau<1$, using this estimate in \eqref{e:pink} we conclude that for all $0<h<h_0$, $|s|\leq 2h$, $K>0$, and $t \in [E-Kh, E+Kh]$
$$h^{\frac{k_i-1}{2}}\Big(\sum_{t-s\leq E_j\leq t}\Big|\int_{H_i}A_i\phi\sub{E_j}\ds{H_i}\Big|^2\Big)^{\tfrac{1}{2}}\leq \frac{C\sub{0}^i\Cnr^i}{\sqrt{\tau T_i(h)}}\Big\langle \frac{T_i(h)s}{h}\Big\rangle^{\tfrac{1}{2}} +C\sub{N}h^N.$$
Combining estimates for $H_1$ and $H_2$ using~\eqref{e: purple}, and $\Cnr^i\leq \Cnr$ completes the proof.
\end{proof}

{The last lemma shows that $w_h(s)=\HAs(s)$ has at most polynomial growth at infinity.}
%%%%%%%%%%%%%%%%%%%%%%%%%%%%%%%%%%%%%%%%%%%%%%%%%%%%%%%%%%%%%%%%%%%%%%%%%%%%%%%
\begin{lemma}
\label{l:bigS}
Let $\ell_1, \ell_2\in \mathbb{R}$. Then, there is {$N_0>0$} such that for all $A_1\in \Psi^{\ell_1}_\delta(M)$, $A_2\in \Psi^{\ell_2}_\delta(M)$, there are $C\sub{1}>0$, {$h_0>0$}, such that for all $0<h<h_0$ and $s\in \mathbb{R}$,
$$
|\HAs(s)|\leq C\sub{1}h^{-\frac{k_1+k_2}{2}}\langle s\rangle^{N_0}.
$$
\end{lemma}
\begin{proof} 
Arguing as in~\eqref{e: purple}, and~\eqref{e:pink}, it is enough to
prove that there is $C\sub{1}>0$ such that for each $i=1,2$ there is $N_i>0$ for which
$$
\sup_{\|w\|_\LM=1}\Big| \int_{H_i}A_i\1_{(-\infty,s]}(P) w\,\ds{H_i}\Big|\leq C\sub{1}h^{-\frac{k_i}{2}}\langle s\rangle^{N_i}.
$$ 
Applying Lemma~\ref{l:SobEst} with $u=\1_{(-\infty,s]}(P) w$ yields that for any $\psi \in S^0(T^*M;[0,1])$ with $\psi\equiv 1$ on $N^*\!H$ and $r_i> \tfrac{k_i+2\ell_i}{2m}$ there exist $C\sub{1}>0$ and $h_0>0$ such that for all $N>0$ there is $C\sub{N}>0$ satisfying for $0<h<h_1$ and $s\in\re$, 
\begin{equation}
\begin{aligned}
&h^{\frac{k_i}{2}}\Big| \int_{H_i}\!A_i\1_{(-\infty,s]}(P) w\,\ds{H_i}\Big| 
\leq C\sub{N}h^N\|\1_{(-\infty,s]}(P) w\|_{H^{-N}_{\scl}(M)}\\
&\qquad\qquad\qquad + C\sub{1} \big(\|{Op_h(\psi)}\1_{(-\infty,s]}(P) w\|_{\LM}\!\!+\|{Op_h(\psi)} P_s^{r_i}\1_{(-\infty,s]}(P) w\|_{\LM}\big).% \\
%&\qquad \qquad +{ C\sub{N}h^N\|\1_{(-\infty,s]}(P) w\|_{H^{-N}_{\scl}(M)}}.
\end{aligned} 
\end{equation}
{Finally, the last term is bounded by $C\sub{1} (1+|s|^{r_i})$ since $\|f(P)\|_{L^2\to L^2}\leq \|f\|_{L^\infty}.$}
\end{proof}

%%%%%%%%%%%%%%%%%%%%%%%%%%%%%%%%%%%%%%%%%%%%%%%%%%%%%%%%%%%%%%%%%%%%%%%%%%%%%%%
%%%%%%%%%%%%%%%%%%%%%%%%%%%%%%%%%%%%%%%%%%%%%%%%%%%%%%%%%%%%%%%%%%%%%%%%%%%%%%%
%%%%%%%%%%%%%%%%%%%%%%%%%%%%%%%%%%%%%%%%%%%%%%%%%%%%%%%%%%%%%%%%%%%%%%%%%%%%%%%
\section{Smoothed projector with non-looping condition}
%%%%%%%%%%%%%%%%%%%%%%%%%%%%%%%%%%%%%%%%%%%%%%%%%%%%%%%%%%%%%%%%%%%%%%%%%%%%%%%
%%%%%%%%%%%%%%%%%%%%%%%%%%%%%%%%%%%%%%%%%%%%%%%%%%%%%%%%%%%%%%%%%%%%%%%%%%%%%%%
%%%%%%%%%%%%%%%%%%%%%%%%%%%%%%%%%%%%%%%%%%%%%%%%%%%%%%%%%%%%%%%%%%%%%%%%%%%%%%%
\label{s:smoothed}
This section is dedicated to the proof of Theorems \ref{t:MAIN} and \ref{t:MAIN2}. {The crucial step, completed in \S\ref{s:fixed time}, is to bound $
(\rho\sub{h,\Tt(h)}-\rho\sub{h,t_0})*\HAs
$ 
when the pair $(H_1,H_2)$ is $(t_0,\Ti)$ non-looping and $\Tt(h)=\frac{1}{2}\Ti(R(h))$. In \S\ref{s:imTheMAIN} we prove Theorem \ref{t:MAIN} by combining the estimates from \S \ref{s:fixed time} with Proposition \ref{p:error-smooth}. In \S \ref{s:Main2} we derive Theorem \ref{t:MAIN2} from Theorem \ref{t:MAIN}.}
%%%%%%%%%%%%%%%%%%%%%%%%%%%%%%%%%%%%%%%%%%
%%%%%%%%%%%%%%%%%%%%%%%%%%%%%%%%%%%%%%%%%%
\subsection{Comparing against a short fixed time} \label{s:fixed time}
%%%%%%%%%%%%%%%%%%%%%%%%%%%%%%%%%%%%%%%%%%
%%%%%%%%%%%%%%%%%%%%%%%%%%%%%%%%%%%%%%%%%%

Throughout this section we continue to assume $H_1 \subset M$ and $H_2\subset M$ are two submanifolds of co-dimension $k_1$ and $k_2$ respectively. 
The goal is to show that, under the assumption {$(H_1,H_2)$ is a $(t_0,\Ti)$ non-looping pair in the window $[a,b]$}, we can control $\rho_{\sigma_{h,\Tt(h)}}*\Pi_h$ by comparing it to $\rho\sub{h,t_0}*\Pi_h$.
For the rest of the section we write 
$$\Tt(h):=\tfrac{1}{2}{\Ti(R(h))},\qquad T(h):=\Ti(R(h)).$$

\begin{proposition}
\label{P:toShortTime}
Suppose $a, b\in \re$ are such that $H_1, H_2$ are conormally transverse for $p$ in the window $[a,b]$. Let $\tau_0, R_0$ be as in Lemma \ref{l:mainEst}.
Let $0<\tau<\tau_0$, $0<\delta<\tfrac{1}{2}$, {and $\Ti$ a sub-logarithmic {resolution function} with $\Omega(\Ti)\Lambda<1-2\delta$.}

Suppose {$(H_1,H_2)$ is a $(t_0, {\Ti})$ non-looping pair} in the window $[a,b]$ via {$\tau$-coverings} {with constant $\Cnl$}.
 Let $A_1,A_2\in \Psi^\infty(M)$, {$ {h^\delta} \leq R(h) \leq R_0$,} {and $K>0$.}
 There exist $$C\sub{0}=C\sub{0} (n,k_1, k_2,\FR^1, \FR^2,A_1,A_2,\Cnl)>0$$ 
 and $h_0>0$ such that for all $0<h<h_0$ and all $E\in[a-Kh,b+Kh]$,
\begin{equation}\label{e:toShortTime 0}
\Big| (\rho\sub{h,\Tt(h)}-\rho\sub{h,t_0})*\HAs(E) \Big|\leq C\sub{0} h^{\frac{2-k_1-k_2}{2}}\Big/{\Ti(R(h))}.
\end{equation}
\end{proposition}

%%%%%%%%%%%%%%%%%%%%%%%%%%%%%%%%%%%%%%%%%%%%%%%%%%%%%%%%%%%%%%%%%%%%%%%%%%%%%%%

We prove the proposition at the end of the section. The proof hinges on four lemmas.
The first one, Lemma \ref{l:toShortTime2}, rewrites the left hand side in \eqref{e:toShortTime 0} in terms of the function
\begin{equation}\label{e:def of f}
{f\sub{S,T,h}(\lambda):=f\sub{S,T}(h^{-1}\lambda)},
\qquad
f\sub{S,T}(\lambda):=\frac{1}{i}\int_\re \frac{1}{\tau}\,\hat{\rho}\big(\tfrac{\tau}{T}\big)\big(1-\hat{\rho}\big(\tfrac{\tau}{S}\big)\big) e^{-i{\tau}\lambda}d\tau,
\end{equation}
where $S,T$ are two positive constants with $S<T$, and $\rho$ is as in \eqref{e:rho def}
%%%%%%%%%%%%%%%%%%%%%%%%%%%%%%%%%%%%%%%%%%%%%%%%%%%%%%%%%%%%%%%%%%%%%%%%%%%%%%%
\begin{remark}We note that for all $N>0$
\begin{gather}\label{e:bound on f}
|f\sub{S, T}(\lambda)|\leq C\sub{N}\langle \lambda S\rangle^{-N},\qquad 
\supp\hat{\rho}\big(\tfrac{\tau}{T}\big)\big(1-\hat{\rho}\big(\tfrac{\tau}{S}\big)\big) \subset \{ \tau \in \R:\; |\tau| \in [S, 2T]\}.
\end{gather}
\end{remark}

\begin{lemma} \label{l:toShortTime2}
Suppose $k>0$ and $P\in \Psi^k(M)$ is self-adjoint with symbol satisfying~\eqref{e:positivity}. Then, for all $N>0$, 
$$
(\rho\sub{h,\Tt}-\rho\sub{h,t_0})*\Pi_h(E)={f\sub{t_0,\Tt,h}\big(P\sub{E}\big)}+O(h^N)_{\smooth}.
$$
\end{lemma}

\begin{proof}
First, we prove that if $P$ is self-adjoint {$E_1, E_2 \in \re$, then}
\begin{equation}\label{e:toShortTime}
\int_{E_1}^{E_2}(\rho\sub{h,\Tt(h)}-\rho\sub{h,t_0})*\partial_s\Pi_h(s)ds= {f\sub{t_0,\Tt(h),h}\big(P\sub{E_2}\big)-f\sub{t_0,\Tt(h),h}\big(P\sub{E_1}\big).}
\end{equation}
 To ease notation write $\Tt$ for $\Tt(h)$. To prove \eqref{e:toShortTime} we write
\begin{align*}
\int_{E_1}^{E_2}(\rho\sub{h,\Tt}-\rho\sub{h,t_0})*\partial_s\Pi_h(s)ds
%&=\int_{E_1}^{E_2} (\rho\sub{h,\Tt}(-P_s)+\rho\sub{h,t_0}(-P_s) )ds\\
%&=\int_{E_1}^{E_2}\int_\re \Big( \hat{\rho}\big(\tfrac{w}{\sigma\sub{h,\Tt}}\big)-\hat{\rho}\big(\tfrac{w}{\sigma\sub{h,t_0}}\big)\Big)e^{-iw(P-s)}dw ds\\
&=\int_{E_1}^{E_2}\int_\re \hat{\rho}\big(\tfrac{w}{\sigma\sub{h,\Tt}}\big)\big[1-\hat{\rho}\big(\tfrac{w}{\sigma\sub{h,t_0}}\big)\big]e^{-iw(P-s)}dw ds,
\end{align*}
where we use $\hat{\rho}\big(\tfrac{w}{\sigma\sub{h,t_0}}\big)=\hat{\rho}\big(\tfrac{w}{\sigma\sub{h,\Tt}}\big)\hat{\rho}\big(\tfrac{w}{\sigma\sub{h,t_0}}\big)$.
Putting $\tau:=hw$,~\eqref{e:toShortTime} follows.
%yields
%\begin{align*}
%\int_{E_1}^{E_2}\partial_s (\rho\sub{h,\Tt}-\rho\sub{h,t_0})*\Pi_h(s)ds
%&=\frac{1}{h}\int_{E_1}^{E_2}\int_\re \hat{\rho}\big(\tfrac{\tau}{\Tt}\big)\big(1-\hat{\rho}\big(\tfrac{\tau}{t_0}\big)\big)e^{-i\frac{\tau}{h}(P-s)}d\tau ds\\
%&=\frac{1}{i}\int_\re \frac{1}{\tau}\,\hat{\rho}\big(\tfrac{\tau}{\Tt}\big)\big(1-\hat{\rho}\big(\tfrac{\tau}{t_0}\big)\big)\Big(e^{-i\frac{\tau}{h}(P-E_2)}-e^{-i\frac{\tau}{h}(P-E_1)}\Big)d\tau\\
%&={f\sub{t_0,\Tt,h}\big(P\sub{E_2}\big)-f\sub{t_0,\Tt,h}\big(P\sub{E_1}\big)}.
%\end{align*}
%This together with 

{Next, let $N>0$. By \eqref{e:toShortTime} it suffices to find $E_1 \in \re$ such that for all ${t>c>0}$
\begin{equation}\label{e:claim with fs}
\big\|{f_{t_0,\Tt,h}\big(P\sub{E_1}\big)}\big\|_{\smooth}\leq C\sub{N} h^{2N}, 
\qquad
\|\rho\sub{h,t}*\Pi_h(E_1)\|_{\smooth} = O(h^N).
\end{equation}
To prove the first claim in \eqref{e:claim with fs},} note that by {\eqref{e:bound on f}} for all $N>0$ there is $C\sub{N}>0$ such that
$$
\big\|P\sub{E_1}^N{f_{t_0,\Tt,h}\big(P\sub{E_1}\big)}P\sub{E_1}^N\big\|_{L^2\to L^2}\leq C\sub{N}h^{2N}. 
$$
Next, since $P$ satisfies~\eqref{e:positivity}, there is $a>0$ such that $p(x,\xi)>-a$ for all $(x, \xi) \in T^*M$. In particular, for $E_1\leq -2a$, $P\sub{E_1}$ is elliptic and we have
$
P\sub{E_1}^{-1}:H_{\textup{scl}}^s(M)\to H_{\textup{scl}}^{s+k}(M)=O_s(1)
$
 for all $s \in \re$.
Then, for $E_1\leq {-2a}$ {the first claim in \eqref{e:claim with fs} follows}.

Next, by the sharp G\aa rding inequality, there is $C>0$ such that $\Pi_h(s)\equiv 0$ for $s\leq -a-Ch$. Thus, for $E_1\leq -3a$ and all $N,M\geq 0$ {there is $C\sub{M,N}>0$ such that}
\begin{gather*}
\|({\rho\sub{h,t}}*\Pi_h)(E_1)\|_{\smooth}\leq \int_{\re}\tfrac{t}{h}\rho\big(\tfrac{t}{h}s\big)\|\Pi_h(E_1-s)\|_{\smooth}ds
\leq C\sub{M,N} \int_{s\leq -a}\tfrac{t}{h}\big\langle\tfrac{t}{h}s\big\rangle^{-M}\langle s\rangle^{2N/k}.
\end{gather*}
{The claim follows after choosing $M$ large enough.}
\end{proof}

%%%%%%%%%%%%%%%%%%%%%%%%%%%%%%%%%%%%%%%%%%%%%%%%%%%%%%%%%%%%%%%%%%%%%%%%%%%%%%%
Let $H_1, H_2$, $t_0, T(h)$, $\tau,$ and $ R(h)$ be as in Proposition \ref{P:toShortTime}. 
Since {$(H_1,H_2)$ is a $(t_0, \Ti)$ non-looping pair} in the window $[a,b]$ via {$\tau_0$-coverings}, 
for $i=1,2$ and $h>0$ we let 
\begin{equation}\label{e:tubesj}
 \{\T_j^i\}_{j \in \J^i(h)}\quad\text{a $(\mathfrak{D}_n,\tau, R(h))$-good cover of $\Sab^{H_i}$ satisfying $(1)$ and $(2)$ in Definition \ref{d:non loop cov}}.
\end{equation} 
{We study $A_1{f\sub{t_0,\Tt,h}\big(P\sub{E}\big)}\,A_2^*$ by understanding the behavior of
\begin{equation}
 \Fjl:= Op_h(\chi\sub{\T_j^1})A_1{f\sub{t_0,\Tt,h}\big(P\sub{E}\big)}\,A_2^*Op_h(\chi\sub{\T_{\ell}^2})
\end{equation}
for $j \in \J^1(h)$ and $k \in \J^2(h)$.
 Next, we study the case when $\T_j^1$ does not loop through $\T_k^2$. }
%%%%%%%%%%%%%%%%%%%%%%%%%%%%%%%%%%%%%%%%%%%%%%%%%%%%%%%%%%%%%%%%%%%%%%%%%%%%%%%
\begin{lemma}
\label{l:prop}
Assume $H_1$ and $H_2$ are conormally transverse for $p$ in the window $[a,b]$.
 For $i=1,2$ let $\{\T_j^i\}_{j\in \J^i(h)}$ as in \eqref{e:tubesj} and $j \in \J^1(h)$, $\ell \in \J^2(h)$ be such that
$$
\varphi_t(\T_j^1)\cap \T_{\ell}^2=\emptyset,\qquad |t|\in [t_0{+\tau},T(h){-\tau}].
$$
 Let {$K>0$ and} $\mc{V}$ be a bounded subset of $S_\delta(T^*M;[0,1])$.
Then, there exists $h_0>0$ and for all $N>0$ there exists $C\sub{N}>0$ such that for all $0<h<h_0$, $E\in[a-Kh,b+Kh]$,
and every $\delta$-partition $\{\chi\sub{\T_j^i}\}_{j\in \JE^i(h)}\subset \mathcal{V}$ associated to $\{\T_j\}_{j\in \JE^i(h)}$, { $i=1,2$},
$$
\|\Fjl\|_{H^{-N}_{\scl}(M) \to H^N_{\scl}(M)}\leq C\sub{N}h^N.
$$
%In particular,
%$$
%\Big|\int_{H_1}\int_{H_2}\Big[Op_h(\chi\sub{\T_j^1})A_1 \,f\sub{t_0,\Tt(h)}\Big(\tfrac{P\sub{E}}{h}\Big)\,A_2^*\, Op_h(\chi\sub{\T_k^2})\Big](x,y)\,\ds{H_2}(y)\ds{H_1}(x)\Big|\leq C\sub{N}h^N.
%$$
\end{lemma}
%%%%%%%%%%%%%%%%%%%%%%%%%%%%%%%%%%%%%%%%%%%%%%%%%%%%%%%%%%%%%%%%%%%%%%%%%%%%%%%
\begin{proof}
By Egorov's theorem, for all $N>0$ there exist $h_0>0$ and $C\sub{N}>0$ such that for all $0<h<h_0$, and $E\in[a-Kh,b+Kh]$,
$$
\big\|Op_h(\chi\sub{\T_j^1})A_1\, e^{-i t \frac{P\sub{E}}{h}}A_2^*\,Op_h(\chi\sub{\T_{\ell}^2})\big\|_{H^{-N}_{\scl}(M) \to H^N_{\scl}(M)}\leq C\sub{N}h^N ,\qquad |t|\in [t_0{+\tau},T(h){-\tau}]
$$
{(see e.g.~\cite[Proposition 3.9]{DyGu:14})}.
The claim follows from the definition \eqref{e:def of f} together with the facts that by \eqref{e:bound on f} the support of its integrand has $\tau \in [t_0, 2 \Tt(h)]$,
and $\Tt(h)=\tfrac{1}{2}T(h)$. 
\end{proof}

{The next lemma provides an estimate for $\Fjl$ based on volumes of tubes. }
%%%%%%%%%%%%%%%%%%%%%%%%%%%%%%%%%%%%%%%%%%%%%%%%%%%%%%%%%%%%%%%%%%%%%%%%%%%%%%%
\begin{lemma}
\label{l:vol}
Assume $H_1$ and $H_2$ are conormally transverse for $p$ in the window $[a,b]$.
{Let $A_1$, $A_2$, $\tau_0$, $R_0$, $\tau$, $\delta$, and $R(h)$ be as in Proposition \ref{P:toShortTime}.} 
For $i=1,2$ let $\{\T_j^i\}_{j\in \J^i(h)}$ be a $(\mathfrak{D}_n, \tau, R(h))$-good covering of $\Sab^{H_i}$. 
 {Let $K>0$} and $\mc{V}$ a bounded subset of $S_\delta(T^*M;[0,1])$. Then, there are $C\sub{0}=C\sub{0}(n,k_1, k_2,\FR^1, \FR^2, A_1,A_2,\mc{V})$ and $h_0>0$, and for all $N>0$ there exists $C\sub{N}>0$
such that the following holds. 
For all $0<h<h_0$, $E\in [a-Kh,b+Kh]$, 
all $\delta$-partitions $\{\chi\sub{\T_j^i}\}_{j\in \JE^i(h)}\subset \mathcal{V}$ and $\mc{I}_i \subset \JE^i(h)$ for $i=1,2$, and all {$t_0, \tilde T$ with $0<t_0<\Tt$},
\begin{align*}
&\bigg|\int_{H_1}\int_{H_2} \sum_{\ell\in \mc{I}_1,j\in \mc{I}_2}\!\!\! \Fjl(x,y)\ds{H_2}(y)\ds{H_1}(x)\bigg|\leq C\sub{0} \tau^{-1} h^{\frac{2-k_1-k_2}{2}}R(h)^{n-1}|\mc{I}_1|^{\frac{1}{2}}|\mc{I} _2|^{\frac{1}{2}} + C\sub{N}h^N.
\end{align*}
\end{lemma}
%%%%%%%%%%%%%%%%%%%%%%%%%%%%%%%%%%%%%%%%%%%%%%%%%%%%%%%%%%%%%%%%%%%%%%%%%%%%%%%
\begin{proof}
The first step in our proof is to define for $0<t_0<\Tt$ the functions
$$
 g^2\sub{t_0,\Tt}(\lambda)g^1\sub{t_0,\Tt}(\lambda):=f\sub{t_0,\Tt}(\lambda), \qquad \qquad g^2\sub{t_0,\Tt}(\lambda):=\langle t_0\lambda \rangle^{-N_0},
$$
where $N_0 \geq 1$ will be chosen later.
Note that by \eqref{e:bound on f} for all $L>0$ there is $C_{L}>0$ such that
\begin{equation}\label{e:bound on g1}
|g^1\sub{t_0,\Tt}(\lambda)|\leq C_{L}\langle t_0 \lambda\rangle^{-L+1}.
\end{equation}
Since
${f\sub{t_0,\Tt,h}\big(P\sub{E}\big)= g^1\sub{t_0,\Tt,h}\big(P\sub{E}\big)g^2\sub{t_0,\Tt,h}\big(P\sub{E}\big)},$
we may use Cauchy-Schwarz to bound
\begin{align*}
&\Big|\int_{H_1}\int_{H_2} \sum_{\ell\in \mc{I}_1,j\in \mc{I}_2}\Big[ \Fjl\Big](x,y)\ds{H_2}(y)\ds{H_1}(x)\Big|\\
&\qquad\leq 
\Big\|\sum_{\ell\in \mc{I}_1} {g^1\sub{t_0,\Tt}\big(P\sub{E}\big)}A_1^*\,Op_h(\chi\sub{\T_\ell^1})\delta\sub{H_1}\Big\|_\LM
\Big\|\sum_{\ell\in \mc{I}_2}{g^2\sub{t_0,\Tt,h}\big(P\sub{E}\big)}A_2^*\,Op_h(\chi\sub{\T_\ell^2})\delta\sub{H_2}\Big\|_\LM.
\end{align*}

Next, we use that for $i=1,2$,
$$
\Big\|\sum_{\ell\in \mc{I}_i} {g^i\sub{t_0,\Tt,h}\big(P\sub{E}\big)}A_i^*\,Op_h(\chi\sub{\T_\ell^i})\delta\sub{H_i}\Big\|_\LM
\!\!\!\leq \sup_{\|w\|=1}
\Big|\int_{H_i}\sum_{\ell\in \mc{I}_i}Op_h(\chi\sub{\T_{\ell}^i})A_i \, {g^i\sub{t_0,\Tt,h}\big(P\sub{E}\big)}w \,\ds{H_i}\Big|.
$$
Thus, let $w \in L^2(M)$ and fix $i \in \{1,2\}$. We next apply Lemma~\ref{l:basicEst} to the function $u={g^i\sub{t_0,\Tt,h}\big(P\sub{E}\big)}w$ and operator $A=\sum_{j\in \mc{I}_i}Op_h(\chi\sub{\T_{j}^i}) A_i\in \Psi_\delta^\infty(M)$. Here, we use that $\MSh(A) \subset \cup_{j \in \mc{I}_i}\T_j^i$ and that $\tfrac{1}{h}[P\sub{E},A] \in \Psi^\infty_\delta(M)$ {(see the definition of a $\delta$-partition~\eqref{e: chi H})}. In particular, we may fix $\mc{W}\subset \Psi^\infty_\delta(M)$ such that $\tfrac{1}{h}[P\sub{E},A] \in \mc{W}$ regardless of the choice of cover and $\delta$-partition {contained in $\mc{V}$}. Then, the constant $C\sub{0}^i$ provided by the Lemma depends on $A_i$ instead of $\mc{W}$. 

Fix $\psi \in C^\infty_0(\re;[0,1])$ with $\psi(t)=1$ for $|t| \leq \tfrac{1}{4}$ and $\psi(t)=0$ for $|t| \geq 1$.
By Lemma~\ref{l:basicEst} with $t_1=t_0$, $T_1=t_0$, and $\G_\ell =\emptyset$ for all $\ell >1$, we obtain that there are $C\sub{0}^i=C\sub{0}^i(n,k_i,\FR^i,A_i)>0$, $C>0$, there exist $h_0>0$ and for all $N>0$ there is $C\sub{N}>0$ such that for all $0<h<h_0$
\begin{align*}
&h^{\frac{k_i-1}{2}}\Big|\int _{H_i} \sum_{j\in \mc{I}_i} Op_h(\chi\sub{\T_j^i})A_i \,{g^i\sub{t_0,\Tt,h}\big(P\sub{E}\big)}w\, \ds{H_i}\Big|= Q^{A,\psi}_{E,h}(C, C\sub{N}, {g^i\sub{t_0,\Tt,h}\big(P\sub{E}\big)}w)\\
&\qquad \qquad \qquad + C\sub{0}^i\,R(h)^{\frac{n-1}{2}}|\mc{I}_i|^{\frac{1}{2}}
\Big(\frac{1}{\tau^{\frac{1}{2}} }\big\|{g^i\sub{t_0,\Tt,h}\big(P\sub{E}\big)}w\big\|_\LM+ \tfrac{t_0}{h}\big\|P\sub{E}{g^i\sub{t_0,\Tt,h}\big(P\sub{E}\big)}w\big\|_\LM\Big).
%&\hspace{3cm}+ Q^{A,\psi}_{E,h}(C, C\sub{N}, {g^i\sub{S,T,h}\big(P\sub{E}\big)}w).
%&\qquad \qquad \qquad \;\;\; +Ch^{-\tfrac{1}{2}-\delta}\Big\|\big(1-\psi\big(\tfrac{P\sub{E}}{h^\delta}\big)\big)P\sub{E}Ag^i\sub{S,T}\Big(\tfrac{P\sub{E}}{h}\Big)w\Big\|_{H\sub{\scl}^{\frac{k_i-2m+1}{2}}}\\
%&\qquad \qquad \qquad \;\;\; +C\sub{N}h^N\Big(\Big\|g^i\sub{S,T}\Big(\tfrac{P\sub{E}}{h}\Big)w\Big\|_\LM+\Big\|P\sub{E}g^i\sub{S,T}\Big(\tfrac{P\sub{E}}{h}\Big)w\Big\|_{H\sub{\scl}^{\frac{k_i-2m+1}{2}}}\Big).
\end{align*}
By the definitions $g^i\sub{t_0,\Tt}$, $i=1,2$ and \eqref{e:bound on g1} there exists $C>0$ such that for all $t_0,\Tt$ with $t_0<\Tt$,
\begin{gather*}
\big\|{g^i\sub{t_0,\Tt,h}\big(P\sub{E}\big)}\big\|_{L^2\to L^2}\leq C,\qquad
 \big\| P\sub{E}{g^i\sub{t_0,\Tt,h}\big(P\sub{E}\big)}\big\|_{L^2\to L^2}\leq C\frac{h}{t_0},\quad i=1,2.
\end{gather*}
In addition, note that for $i=1,2$ there exists $C\sub{N_0}>0$ such that 
$$
\big\|\big(1-\psi\big(\tfrac{P\sub{E}}{h^\delta}\big)\big)P\sub{E}A\,{g^i\sub{t_0,\Tt,h}\big(P\sub{E}\big)}\big\|_{L^2 \to L^2}\leq C\sub{N_0}h^{N_0(1-\delta)+\delta}. 
$$
The claim follows from choosing $N_0$ large enough that $N_0(1-\delta)+\delta \geq N$.
\end{proof}

%%%%%%%%%%%%%%%%%%%%%%%%%%%%%%%%%%%%%%%%%%%%%%%%%%%%%%%%%%%%%%%%%%%%%%%%%%%%%%%%%%%%%%%%%%%%%%%%%%%%%
\begin{lemma} \label{l:Q(h)}
Assume the same assumptions as in Proposition \ref{P:toShortTime}.
 For $i=1,2$ let $\{\T_j^i\}_{j\in \J^i(h)}$ be as in \eqref{e:tubesj},
%For $j \in \J^1(h)$ and $\ell \in \J^2(h)$ define
%\begin{equation}\label{e:I def}
% I_{j,\ell}(h):=\gamma\sub{H_1}A_1Op_h(\chi\sub{\mc{T}^1_j}){f\sub{t_0, \Tt(h),h}\big(P\sub{E}\big)}Op_h(\chi\sub{\mc{T}^2_\ell})A_2^*\delta\sub{H_2},
%\end{equation}
%and 
$\mc{V}$ be a bounded subset of $S_\delta(T^*M;[0,1])$ {and $K>0$}.
There exists $h_0>0$, and for all $N>0$ there exists $C\sub{N}>0$
such that for all $0<h<h_0$, $E\in [a-Kh,b+Kh]$, 
and every $\delta$-partition $\{\chi\sub{\T_j^i}\}_{j\in \JE^i(h)}\subset \mathcal{V}$ associated to $\{\T_j^i\}_{j\in \JE^i(h)}$,
%$$\|Q(h)\|_{H\sub{\scl}^{-N}(H_2)\to H\sub{\scl}^N(H_1)}\leq C\sub{\!N}h^N.$$
{
\begin{equation*}\label{e:onlyNear}
\Big\|\gamma\sub{H_1}A_1{f\sub{t_0, \Tt,h}\big(P\sub{E}\big)}A_2^*\delta\sub{H_2}
-\sum_{j\in \J^1\sub{{E}}(h),\, \ell \in \J^2\sub{{E}}(h)}\!\!\!\!\!\!\!\!\!\!\!\!\!\gamma\sub{H_1} \Fjl \delta\sub{H_2}\Big\|_{H\sub{\scl}^{-N}(H_2)\to H\sub{\scl}^N(H_1)} 
\leq C\sub{N}h^N.
\end{equation*}}
%\vspace{-0.5cm}
\end{lemma}
%%%%%%%%%%%%%%%%%%%%%%%%%%%%%%%%%%%%%%%%%%%%%%%%%%%%%%%%%%%%%%%%%%%%%%%%%%%%%%%%%%%%%%%%%%%%%%%%%%%%%
\begin{proof}
 {Let $K>0$} {and $\psi \in C^\infty_c((-1,1);[0,1])$ with $\psi(t)=1$ for $|t| \leq \tfrac{1}{4}$.}
 {We claim there exists $h_0>0$ such that for all $N>0$ there is $C\sub{N}>0$ so that for $0<h<h_0$,}
\begin{equation}
 \label{e:away3}
\|\big(1-{\psi}\big(\tfrac{P\sub{E}}{h^\delta}\big)\big)f\sub{t_0, \Tt,h}\big(P\sub{E}\big)\|_{H_{\textup{scl}}^{-N}(M)\to H_{\textup{scl}}^N(M)}\leq C\sub{N} h^{N},\qquad\qquad E\in[a-Kh,b+Kh].
\end{equation}
To see this, first note that for ${\tilde \psi}\in C_c^\infty$ with $\supp\tilde \psi \subset \{{\psi}\equiv 1\}$ {and $L>0$},
$$\big(1-{\psi}\big(\tfrac{P\sub{E}}{h^\delta}\big)\big)f\sub{t_0, \Tt,h}\big(P\sub{E}\big)
=P\sub{E}^{-L}\big(1-{\psi}\big(\tfrac{P\sub{E}}{h^\delta}\big)\big)P\sub{E}^L f\sub{t_0, \Tt,h}\big(P\sub{E}\big)P\sub{E}^LP\sub{E}^{-L}\big(1-\tilde \psi\big(\tfrac{P\sub{E}}{h^\delta}\big)\big).$$
Now, since $P\sub{E}$ is classically elliptic in $\Psi^m(M)$, for all $s\in \mathbb{R}$, 
\begin{equation} 
\label{e:imElliptic}
P\sub{E}^{-L}\big(1-\psi\big(\tfrac{P\sub{E}}{h^\delta}\big)\big)=O_{L,s}(h^{-\delta L})_{{H_{\textup{scl}}^{s}(M)\to H_{\textup{scl}}^{s+mL}(M)}}.
\end{equation}
{Note that~\eqref{e:imElliptic} also holds with $\tilde \psi$ in place of $\psi$.}
In addition, by~\eqref{e:bound on f}
\begin{equation}
\label{e:imSmall}
P\sub{E}^L f\sub{t_0, \Tt,h}\big(P\sub{E}\big)P\sub{E}^L=O\sub{L}(h^{2L})_{L^2(M)\to L^2(M)}.
\end{equation}
Taking $L>\max(N/m,N/(2(1-\delta)))$ and combining~\eqref{e:imElliptic} and~\eqref{e:imSmall} we obtain~\eqref{e:away3}.% 

Next, for $i=1,2$ we define 
$
{G_i}:=\Id-\sum_{j \in \J\sub{{E}}^i(h)} \!\!\!Op_h(\chi\sub{\mc{T}^i_j}),
$
and note that {$\MSh(G_i)\cap \Lambda^\tau_{\SE^{H_i}}(R(h)/2)=\emptyset$}. Therefore, combining Lemma~\ref{l:mainEst} together with~\eqref{e:away3}, {there exists $h_0>0$ such that for all $N>0$ there is $C\sub{N}>0$ so that for all $0<h<h_0$,}
{
\begin{equation}\label{e:meta1}
 \big\| \gamma\sub{H_1}A_1G_1 {f\sub{t_0, \Tt,h}}\big(P\sub{E}\big)A_2^*\delta\sub{H_2} \big\|_{H_\text{scl}^{-N}(H_2)\to H_\text{scl}^N(H_1)}
\leq C\sub{N}h^N,\qquad E\in[a-Kh,b+Kh].
\end{equation}
}
In particular, the lemma follows from applying \eqref{e:meta1} and its analogs since
\begin{align*}
&{\gamma\sub{H_1}A_1{f\sub{t_0, \Tt,h}\big(P\sub{E}\big)}A_2^*\delta\sub{H_2}
-\!\!\!\!\!\!\!\!\!\sum_{j\in \J_E^1(h),\ell \in \J_E^2(h)}\!\!\!\!\!\!\!\!\!\!\gamma\sub{H_1} \Fjl \delta\sub{H_2}}\\
&=\gamma\sub{H_1}A_1G_1f\sub{t_0, \Tt,h}\big(P\sub{E}\big)A_2^*\delta\sub{H_2}+\gamma\sub{H_1}A_1 f\sub{t_0, \Tt,h}\big(P\sub{E}\big)G_2A_2^*\delta\sub{H_2}+\gamma\sub{H_1}A_1 G_1 f\sub{t_0, \Tt,h}\big(P\sub{E}\big)G_2A_2^*\delta\sub{H_2}. \qedhere
\end{align*}
\end{proof}

%%%%%%%%%%%%%%%%%%%%%%%%%%%%%%%%%%%%%%%%%%%%%%%%%%%%%%%%%%%%%%%%%%%%%%%%%%%%%%%%%%%%%%%%%%%%%%%%%%%%%
\noindent{\bf Proof of Proposition \ref{P:toShortTime}.}
Since $(H_1,H_2)$ is a $(t_0, \Ti)$ non-looping pair in the window $[a,b]$ via {$\tau_0$-coverings}, 
for $i=1,2$ and $h>0$ we may work with $\{\T_j^i\}_{j\in \J^i(h)}$, as in~\eqref{e:tubesj} {and} $\{\chi\sub{\T_j^i}\}_{j\in \J^i(h)}$ a $\delta$-partition associated $\{\T^i_j\}$
{For each $E\in [a,b]$ and $i=1,2,$ let $\J^i\sub{E,h}=\BE^i(h)\cup\mc{G}^i\sub{E}(h)$ be a partition of indices such that property $(1)$ of Definition \ref{d:non loop cov} with $r=R(h)$.}
%for $i,k=1,2$ with} $i\neq k$ and $\ell\in \mc{G}^i\sub{E}(h)$ then
%$$
%\Big( \bigcup_{|t|\in[t_0,T(h)]}\varphi_t(\supp \chi\sub{\mc{T}^i_\ell})\Big)\bigcap \Big(\bigcup_{j \in \J^k\sub{E,h}}\supp \chi\sub{\mc{T}^k_j}\Big)=\emptyset.
%$$
%Let $I_{j, \ell}(h)$ be as in Lemma \ref{l:Q(h)}.
Then, by Lemma~\ref{l:prop}, 
 {for $K>0$ there exists $h_0>0$ such that the following holds: For all $N>0$ there is $C\sub{N}>0$ so that for all $0<h<h_0$, $E\in [a-Kh,b+Kh]$,}
%for all $N>0$ and ${K}>0$ there exist $h_0>0$ and $C\sub{N}>0$ such that for all $0<h<h_0$, all $E\in[a-Kh,b+Kh]$, 
{and $i,k=1,2$ with $i\neq k$,}
\begin{equation}
 \label{e:nonLoop}
 \begin{gathered}
\Big|\int_{H_1} \int_{H_2}\sum_{j \in {\J^k\sub{E}(h)}}\sum_{\ell\in {\mc{G}^i\sub{E}(h)}}[\Fjl](x,y) \ds{H_2}(y)\ds{H_1}(x)\Big|
\leq C\sub{N}h^N.%\\
%\red{\Big|\int_{H_1} \int_{H_2}\sum_{j \in \G^1\sub{E}(h)}\sum_{\ell\in \mc{J}_h^2}I_{j,\ell}(h)(x,y) \ds{H_2}(y)\ds{H_1}(x)\Big|\leq C\sub{N}h^N}
\end{gathered}
\end{equation}
Therefore, considering the remaining term, and applying Lemma~\ref{l:vol} we obtain the following. There is $C\sub{0}=C\sub{0} (n,k_1, k_2,\FR^1, \FR^2,A_1,A_2)>0$ {and for $K>0$ there exists $h_0>0$ such that the following holds: For all $N>0$ there is $C\sub{N}>0$ so that for all $0<h<h_0$, $E\in [a-Kh,b+Kh]$,}
\begin{align}
 \label{e:finalVol}
&\Big|\int_{H_1} \int_{H_2}\sum_{j \in \B^1\sub{E}(h)}\sum_{\ell\in \BE^2(h)}[\Fjl](x,y)\ds{H_2}(y)\ds{H_1}(x)\Big| \notag\\
&\qquad\quad \leq C\sub{0}h^{\frac{2-k_1-k_2}{2}}R(h)^{n-1}{|\B^1\sub{E}(h)|^{\frac{1}{2}}}|\B^2\sub{E}(h)|^{\frac{1}{2}}+C\sub{N}h^N \leq C\sub{0}\Cnl h^{\frac{2-k_1-k_2}{2}}\big/T(h).
\end{align}
To get the last line we used that our covering satisfies property $(2)$ of Definition \ref{d:non loop cov}.
Combining Lemma \ref{l:Q(h)} with ~\eqref{e:onlyNear}, ~\eqref{e:nonLoop}, and~\eqref{e:finalVol}, we obtain the claim.

%%%%%%%%%%%%%%%%%%%%%%%%%%%%%%%%%%%%%%%%%%%%%%
%%%%%%%%%%%%%%%%%%%%%%%%%%%%%%%%%%%%%%%%%%%%%%
\subsection{Proof of Theorem \ref{t:MAIN} }
%%%%%%%%%%%%%%%%%%%%%%%%%%%%%%%%%%%%%%%%%%%%%%
%%%%%%%%%%%%%%%%%%%%%%%%%%%%%%%%%%%%%%%%%%%%%%
\label{s:imTheMAIN}
 {Since for $i=1,2$ the submanifold $H_i$ is $ T_i(h)$ non-recurrent} in the window $[a,b]$ via ${\tau_0}$-coverings {with constant $\Cnr^i$}, we may apply Proposition \ref{p:error-smooth} {to obtain the existence of $C\sub{0}=C\sub{0}(n,k_1, k_2,\FR^1, \FR^2,A_1,A_2,\Cnr^1, \Cnr^2)$ and for all $K>0$ obtain $h_0>0$ such that} for all $0<h\leq h_0$ and ${s}\in[a-Kh,b+Kh]$,
\begin{equation}\label{e:MAIN1}
\Big|\HAs(s)-\rho\sub{h,{\widetilde{T}_{\max}(h)}}*\HAs(s)\Big|\leq C\sub{0}\,h^{\frac{2-k_1-k_2}{2}}\big/\,T(h),
\end{equation}
{where $T(h)=(T_1(h)T_2(h))^{\tfrac{1}{2}}$ {and $T_{\max}(h)=\max(T_1(h),T_2(h))$.} Note that we are actually applying the proposition only using that $H_i$ is $\tfrac{1}{2} T_i(h)$ non-recurrent.}

On the other hand, since {$(H_1,H_2)$ is a $(t_0, {\Ti_{\max}})$ non-looping pair} in the window $[a,b]$ via {$\tau_0$ coverings}, we may apply Proposition \ref{P:toShortTime} to obtain that there exist $C\sub{1}=C\sub{1}(n,k_1, k_2,\FR^1, \FR^2,A_1,A_2,\Cnl)>0$ and {for all $K>0$} {there is} $h_0>0$ such that for all $0<h<h_0$ and all $s\in[a-Kh,b+Kh]$
\begin{equation}\label{e:MAIN2}
\Big| (\rho\sub{h,{\widetilde{T}_{\max}}(h)}-\rho\sub{h,t_0})*\HAs(s) \Big|\leq C\sub{1} \,h^{\frac{2-k_1-k_2}{2}}\big/\,T(h).
\end{equation}
{The result follows from combining \eqref{e:MAIN1} with \eqref{e:MAIN2}.
%we conclude that {for all $K>0$} there exists $h_0>0$ such that for all $0<h<h_0$ and all $s\in[a-Kh,b+Kh]$,
%\begin{equation*}
%\Big|\RAs(t_0,h; s)\Big|\leq (C\sub{0}+C\sub{1})\frac{h^{\frac{2-k_1-k_2}{2}}}{T(h)}.
%\end{equation*}
%
We note that $H_1$ and $H_2$ may be replaced by $\tilde H_{1,h}$ and $\tilde H_{2,h}$ since $\Cnl$, $\Cnr^1$, and $\Cnr^2$ are uniform for $\{\tilde H_{1,h}\}_h$ and $\{\tilde H_{2,h}\}_h$.}
%%%%%%%%%%%%%%%%%%%%%%%%%%%%%%%%%%%%%%%%%%%%%%
%%%%%%%%%%%%%%%%%%%%%%%%%%%%%%%%%%%%%%%%%%%%%%
\subsection{Proof of Theorem \ref{t:MAIN2} }\label{s:Main2}
%%%%%%%%%%%%%%%%%%%%%%%%%%%%%%%%%%%%%%%%%%%%%%
%%%%%%%%%%%%%%%%%%%%%%%%%%%%%%%%%%%%%%%%%%%%%%

{Let $0<\tau{<\min(\tau_0,\e/3)}$. By Proposition \ref{p:non rec implies cov} there exists $c_0>0$, $\Cnr=\Cnr(M,p,\ti,R_0)>0$ such that for $j=1,2$, the submanifold $H_j$ is $c \Ti_i(R)$ non-recurrent in the window $[a,b]$ via $\tau$ coverings with constant $\Cnr$. }

{Now, since $(H_1,H_2)$ is a $(t_0,{\Ti}_{\max})$ non-looping pair in the window $[a,b]$ with constant $\Cnl$. Proposition~\ref{p:non loop implies cov} implies there is $\widetilde{\Cnl}=\widetilde{\Cnl}(p,a,b,n,\Cnl, H_1,H_2)$ such that $(H_1,H_2)$ is a $(t_0+3\tau_0,\tilde{\Ti})$ non-looping pair in the window $[a,b]$ via $\tau_0$-coverings with constant $\widetilde{\Cnl}$ where $\tilde{\Ti}(R)=\Ti_{\max}(4R)-3\tau_0$. Since $\Ti_j$ are sub-logarithmic, there is $c_1>0$ such that $\tilde{\Ti}(R)\geq c_1 \Ti_{\max}(R)$. The proof now follows from a direct application of Theorem~\ref{t:MAIN} with $\Ti_j$ replaced by $\min(c_0,c_1)\Ti_j$ and $t_0$ by $t_0+\e$.}

%%%%%%%%%%%%%%%%%%%%%%%%%%%%%%%%%%%%%%%%%%%%%%%%%%%%%%%%%%%%%%%%%%%%%%%%%%%%%%%
%%%%%%%%%%%%%%%%%%%%%%%%%%%%%%%%%%%%%%%%%%%%%%%%%%%%%%%%%%%%%%%%%%%%%%%%%%%%%%%
%%%%%%%%%%%%%%%%%%%%%%%%%%%%%%%%%%%%%%%%%%%%%%%%%%%%%%%%%%%%%%%%%%%%%%%%%%%%%%%
%%%%%%%%%%%%%%%%%%%%%%%%%%%%%%%%%%%%%%%%%%%%%%%%%%%%%%%%%%%%%%%%%%%%%%%%%%%%%%%
\section{The Weyl law}
%%%%%%%%%%%%%%%%%%%%%%%%%%%%%%%%%%%%%%%%%%%%%%%%%%%%%%%%%%%%%%%%%%%%%%%%%%%%%%%
%%%%%%%%%%%%%%%%%%%%%%%%%%%%%%%%%%%%%%%%%%%%%%%%%%%%%%%%%%%%%%%%%%%%%%%%%%%%%%%
%%%%%%%%%%%%%%%%%%%%%%%%%%%%%%%%%%%%%%%%%%%%%%%%%%%%%%%%%%%%%%%%%%%%%%%%%%%%%%%
%%%%%%%%%%%%%%%%%%%%%%%%%%%%%%%%%%%%%%%%%%%%%%%%%%%%%%%%%%%%%%%%%%%%%%%%%%%%%%%

\label{s:weyl}
In order to improve remainders in the Weyl law itself, we let $\Delta \subset M \times M$ be the diagonal, and for $A_1, A_2\in \Psi^\infty(M)$ consider the integral
$$
\int_{M}[A_1\1_{(-\infty,s]}(P)A_2](x,x)\dv_g(x)= \int_{\Delta}\Big((A_1\otimes A_2^*)\1_{(-\infty,s]}(P)\Big)(x,y) d\sigma\sub{\!\Delta}(x,y),
$$
{where $d\sigma\sub{\!\Delta}$ is the Riemannian volume form induced on $\Delta$ by the product metric on $M \times M$.}
To ease notation, we write $\Pt=(P-t)\otimes 1={P\otimes 1-t\Id}.$
We will view $\Delta$ as a hypersurface of codimension $n$ in $M\times M$, and the kernel of $\1_{[t-s,t]}(P)$ as a quasimode for $\Pt$. In particular, observe that {for any operator $B:L^2(M)\to L^2(M)$}
\begin{equation}\label{e:bee}
\|\Pt \1_{[t-s,t]}(P){B}\|\sub{L^2(M\times M)}\leq |s|\|\1_{[t-s,t]}(P){B}\|\sub{L^2(M\times M)}.
\end{equation}
In addition, note that for $(x,\xi,y,\eta) \in T^*M \times T^*M$
$$
\sigma(\Pt)(x,\xi,y,\eta)=p(x,\xi)-t=:{\p}(x,\xi,y,\eta)-t=:\pt(x,\xi,y,\eta).
$$
Therefore, for all $c>0$, there is $C>0$ such that if $c|\eta|\leq |\xi|$ and $|\xi|\geq C$, then
$$
|\sigma(\Pt)(x,\xi,y,\eta)|\geq \tfrac{1}{C}|(\xi,\eta)|^m.
$$
In particular, since we work near the $\p$ flow-out of
$
N^*\Delta\cap \{\p=t\}
$
where $t\in[a,b]$, and
$$
N^*\Delta=\{ (x,\xi,x,-\xi):\; (x,\xi)\in T^*M\},
$$
we may work as though $\Pt$ were elliptic in $\Psi^m(M\times M)$, and apply the results of the previous sections by accepting $O(h^\infty)$ errors. We will do this without further comment. 

We next describe the tubes relevant in this section. We will work microlocally near a point $\rho_0\in N^*\Delta\cap \p^{-1}([a,b])$. Let $\pi\sub{R},\pi\sub{L}:T^*(M\times M)\to T^*M$ denote the projections to the right and left factor, and let $\hyp_{\pi\sub{L}(\rho_0)}\subset T^*M$ be a transversal to the flow for $p$ containing $\pi\sub{L}(\rho_0)$. (Such a hypersurface exists since $dp(\rho)\neq 0$ on $p^{-1}([a,b])$.) Define a transversal to the flow for $\p$ by 
$$
\hyp_{\rho_0}:=\hyp_{\pi\sub{L}(\rho_0)}\times T^*M,
$$
and let $U$ be a neighborhood of $\rho_0$ in $N^*\Delta$ such that 
$
U\cap \p^{-1}([a,b])\subset \hyp_{\rho_0}.
$
We will use the metric $\tilde d$ on $T^*M\times M$ defined by
$
\tilde d\Big((\rho\sub{L},\rho\sub{R}),(q\sub{L},q\sub{R})\Big):=\max\Big( d(\rho\sub{L},q\sub{L}), d(\rho\sub{R},q\sub{R})\Big),
$
for $(\rho\sub{L},\rho\sub{R}),(q\sub{L},q\sub{R})\in T^*M \times M.$
With this definition, for $\rho=(\rho\sub{L},\rho\sub{R})\in N^*\Delta\cap \{\pt=0\}$, 
$$
\T_\rho:=\Lambda_\rho^\tau (r)= \tilde{\Lambda}_{\rho\sub{L}}^\tau(r)\times B(\rho\sub{R},r)
$$
where $\Lambda_A^\tau(r)$ is defined by~\eqref{e:tube} with $\varphi_t$ the Hamiltonian flow for $\p$ and {$\tilde \T=\tilde{\Lambda}_{\rho\sub{L}}^\tau(r)$} denotes a tube with respect to $p$ and the hypersurface $\hyp_{\pi\sub{L}(\rho_0)}$.
In particular, when we use cutoffs with respect to a tube $\T$, we will always work with cutoffs of the form
$$
\chi\sub{\T}(x,\xi, y, \eta)=\chi\sub{\tilde{\T}}(x,\xi)\chi_{\rho\sub{R}}(y,{\eta}),\qquad\qquad
\supp \chi_{\rho\sub{R}}\subset B(\rho\sub{R},{r}).
$$
We will refer only to this tube in $T^*M$, leaving the other implicit and will think of the kernel of $A_1\1_{[a,b]}(P)A_2$ as that of $\1_{[a,b]}(P)$ acted on by $A_1\otimes A_2^{{t}}.$
Before we start our proof of the improved Weyl remainder, we need a dynamical lemma.
%%%%%%%%%%%%%%%%%%%%%%%%%%%%%%%%%%%%%%%%%%%%%%%%%%%%%%%%%%%%%%%%%%%%%%%%%%%%%%%
\begin{lemma}
Let $\Cnp>0$, $a\leq b$, and $U\subset T^*M$ satisfying $d\pi\sub{M}H_p\neq 0$ on $p^{-1}([a,b])\cap \overline{U}$. Then there are $\tau_0>0$ and $\widetilde{\Cnp}=\widetilde{\Cnp}(p,U,\Cnp)$ such that the following holds. If $U$ is $(t_0,\Ti)$ non-periodic for $p$ in the window $[a,b]$ with constant $\Cnp$, then $N^*\Delta\cap (U\times T^*M)$ is $(t_0+{3}\tau_0,\Ti({16}R)-{3}\tau_0)$ non-looping for $\p$ via {$\tau_0$-coverings} in the window $[a,b]$ with constant $\widetilde{\Cnp}$.
\end{lemma}

%%%%%%%%%%%%%%%%%%%%%%%%%%%%%%%%%%%%%%%%%%%%%%%%%%%%%%%%%%%%%%%%%%%%%%%%%%%%%%%
%\blue{\cs rewrote the proof, you'll want to read it all\cs}
\begin{proof}
Let $E\in [a,b]$. We work with $\mc{L}\sub{\Delta,\Delta}^{{R,E}}(t_0,T)$ as defined in Definition~\ref{d:non loop gral} but with $p$ replaced by $\p$, $\varphi_t^{\p}:=\exp(t H_{\bf p})$, and $\Sigma\sub{E}^{\Delta}=N^*\!\Delta \cap \{\p=\!E\}$.
First, we claim 
\begin{equation}\label{e:wild claim}
 \pi\sub{L} \Big(B\sub{\Sigma\sub{E}^{\Delta}}\!\big({\scriptstyle \mc{L}\sub{\Delta_U,\Delta_U}^{{R,E}}\!(t_0,T)}, R\big) \Big) \subset B\sub{\Sigma\sub{E}^{\Delta}}\!\big({\scriptstyle \mc{P}_{{U}}^{R}(t_0,T)}, {2}R\big).
\end{equation}
Here, through a slight abuse of notation, we write 
$\mc{L}^{R,E}\sub{\Delta_U,\Delta_U}$ for~\eqref{e:L} with $S^*_xM$ and $S^*_yM$ replaced by $\Delta_U:=N^*\Delta \cap (U\times T^*M)$ and $\varphi_t=\exp(tH_\p)$.
To {prove \eqref{e:wild claim}} suppose $\rho_0\in B\sub{\Sigma\sub{E}^{\Delta}}\!\big({\scriptstyle \mc{L}\sub{\Delta_U,\Delta_U}^{{R,E}}\!(t_0,T)}, R\big)$.
Then, there {are} $\rho_1 \in \Sigma\sub{E}^{\Delta}{\cap\Delta_U}$ and $\rho_1'\in T^*(M\times M)$ such that 
$$
\tilde{d}(\rho_0, \rho_1)<R,\qquad {\tilde{d}(\rho_1,\rho_1')<R}, \qquad \text{and} \qquad
\bigcup_{t_0\leq |t|\leq T}\varphi_t^\p({\rho_1'})\cap B({\scriptstyle \Sigma\sub{E}^{\Delta}},R)\neq \emptyset.
$$
Therefore, there is $\rho_2 \in \Sigma\sub{E}^{\Delta}$ such that 
$\tilde d\big(\varphi_t^\p(\rho_1'), \rho_2\big)<R$ for some $t_0\leq |t|\leq T$.
{Let $\rho_1'=(x',\xi',y',-\eta')$ with $(x',\xi'),(y',\eta')\in T^*M$. Then}, since $\rho_1=(x,\xi,x,-\xi)$ and $\rho_2=(y,\eta,y,-\eta)$ for some $(x,\xi)\in T^*M$ and $(y,\eta)\in T^*M$,
{we have
%\begin{gather*}
%\max\Big(d \Big(\varphi_t(x',\xi'),(y,\eta)\Big),\; d\Big((y',-\eta'),(y,-\eta)\Big),\;
%d\Big((x',\xi'),(x,\xi)\Big), \; d\Big((y',-\eta'),(x,-\xi)\Big)\Big)<R.
%\end{gather*}
%This implies 
$d(\varphi_t(x',\xi'),(x',\xi'))<4R$
and 
$\pi\sub{L}({\rho_1'})=({x',\xi'})\in \mc{P}^{{4}R}\sub{U}(t_0,T).$
On the other hand, since
%since $\tilde{d}(\rho_0,\rho_1)<R$ and {$\tilde{d}(\rho_1,\rho'_1)<R$} we have 
$d(\pi\sub{L}(\rho_0),\pi\sub{L}({\rho'_1}))<{2}R$
we obtain $\pi\sub{L}(\rho_0)\in B\sub{S^*\!M}\!\big({\scriptstyle \mc{P}_U^{{4}R}(t_0,T)}, {2}R\big)$. }
This proves claim \eqref{e:wild claim}.

Next, note that since $\pi\sub{L}:{\Delta_U}\cap {\SE^{\Delta} \to \{p=E\}\cap U}$ is a diffeomorphism {for $E\in[a,b]$}, it follows that there exists $C={C(p)}>0$ such that {for all $E \in [a,b]$}
$$
\mu\sub{E}\Big( B\sub{\Sigma\sub{E}^{\Delta}}\!\big({\scriptstyle \mc{L}\sub{\Delta_U,\Delta_U}^{{R,E}}\!(t_0,T)}, R\big)\Big)\leq C \mu\sub{{S^*M}}\Big(B\sub{S^*\!M}\!\big({\scriptstyle \mc{P}_U^{{4}R}(t_0,T)}, {2}R\big)\Big).
$$
Hence, if $U$ is $(t_0,{\Ti})$ non-periodic for $p$ at energy $E$, we have
$$
\mu\sub{E}\Big( B\sub{\Sigma\sub{E}^{\Delta}}\!\big({\scriptstyle \mc{L}\sub{\Delta_U,\Delta_U}^{{R,E}}\!(t_0,{\Ti({4}R))}}, R\big)\Big)\,{\Ti}({4}R)\leq C\mu\sub{{S^*M}}\Big(B\sub{S^*\!M}\!\big({\scriptstyle \mc{P}_U^{{4}R}(t_0,{\Ti({4}R)})}, {4}R\big)\Big)\,{{\Ti}({4}R)}\leq C\Cnp,
$$
and so $\Delta_U$ is $(t_0,\Ti({4}R))$ non-looping for ${\bf p}$ at energy $E$.
The result follows from Corollary~\ref{p:non loop implies cov}.
\end{proof}

%%%%%%%%%%%%%%%%%%%%%%%%%%%%%%%%%%%%%%%%%%%%%%%%%%%%%%%%%%%%%%%%%%%%%%%%%%%%%%%

%We next study the $L^2$ norm of the kernel of the operator $\1_{[a,b]}(P)$.
In what follows, we write $\|\cdot\|\sub{HS}$ for the Hilbert-Schmidt norm.% since $\1_{[a,b]}(P)$ is an orthogonal projector and trace class, 
%$$
%\|\1_{[a,b]}(P)\|\sub{L^2(M\times M)}^2\!\!\!\!\!\!=\|\1_{[a,b]}(P)\|^2\sub{HS}=\tr %\1_{[a,b]}(P)^2=\tr \1_{[a,b]}(P)=\!\int_\Delta\!\1_{[a,b]}(P)d\sigma\sub{\!\Delta}.
%$$}
%%%%%%%%%%%%%%%%%%%%%%%%%%%%%%%%%%%%%%%%%%%%%%%%%%%%%%%%%%%%%%%%%%%%%%%%%%%%%%%
%%%%%%%%%%%%%%%%%%%%%%%%%%%%%%%%%%%%%%%%%%%%%%%%%%%%%%%%%%%%%%%%%%%%%%%%%%%%%%%
%\begin{lemma}
%\label{l:basicTrace}
%{For all $a,b \in \re$} we have 
%$
%\int_{\Delta}\1_{[a,b]}(P)d\sigma\sub{\!\Delta}=\|\1_{[a,b]}{(P)}\|\sub{L^2(M\times M)}^2.
%$
%\end{lemma}
%%%%%%%%%%%%%%%%%%%%%%%%%%%%%%%%%%%%%%%%%%%%%%%%%%%%%%%%%%%%%%%%%%%%%%%%%%%%%%%
%\begin{proof}
%Observe that since $\1_{[a,b]}(P)$ is an orthogonal projector and trace class,
%$$
%\|\1_{[a,b]}(P)\|\sub{L^2(M\times M)}^2\!\!\!\!\!\!=\|\1_{[a,b]}(P)\|^2\sub{HS}=\tr \1_{[a,b]}(P)^2=\tr %\1_{[a,b]}(P)=\!\int_\Delta\!\1_{[a,b]}(P)d\sigma\sub{\!\Delta}.
%$$\vspace{-1.1cm}
%\end{proof}

%%%%%%%%%%%%%%%%%%%%%%%%%%%%%%%%%%%%%%%%%%%%%%%%%%%%%%%%%%%%%%%%%%%%%%%%%%%%%%%
\begin{lemma}\label{l:HS}
Let $\mc{V}\subset S_\delta(T^*M;[0,1])$ be a bounded subset. Then, there are $C>0$ {and $h_0>0$}, and for all $N>0$ there exists $C\sub{N}>0$, such that for all {$t \in [a,b]$,} $\chi \in \mc{V}$, {$0<h<h_0$,} and $|s|\leq 2h$, 
\begin{gather} 
\label{e:purplePeopleEater}
\|\1_{[t-s,t]}(P)Op_h(\chi)\|^2\sub{HS}\leq C h^{1-n} {\mu\sub{{p^{-1}(t)}}}(\supp \chi\cap p^{-1}(t)) +C\sub{N}h^N,\\
\label{e:tomato}
h^{-2}\|P_t\1_{[t-s,t]}(P)Op_h(\chi)\|^2\sub{HS}\leq C h^{1-n}{\mu\sub{{p^{-1}(t)}}}(\supp \chi\cap p^{-1}(t)) +C\sub{N}h^N.
\end{gather}
%{where $\mu_t$ is the Liouville measure on $p^{-1}(t)$.}
\end{lemma}
%%%%%%%%%%%%%%%%%%%%%%%%%%%%%%%%%%%%%%%%%%%%%%%%%%%%%%%%%%%%%%%%%%%%%%%%%%%%%%%
\begin{proof}
We follow the proof of~\cite[Lemma 3.11]{DyGu:14}.
Let $\psi \in \mc{S}(\mathbb{R};\mathbb{R})$ with $\psi(0)=1$ and $\supp \hat{\psi}\subset[-1,1]$. Define $\psi_\e(s):=\psi(\e s).$
Then, there is $\e_0>0$ small enough so that $\psi_{\e_0}(s)>\frac{1}{2}$ on $[-2,2]$. Abusing notation slightly, put $\psi=\psi_{\e_0}$. 
Then, {there exists an operator $Z_s$ such that }
$
\1_{[t-s,t]}(P)={Z_s}\psi\big(\tfrac{P_t}{h}\big),
$
{$[Z_s,P]=0$, and
$\|Z_s\|_{L^2\to L^2}\leq 3$
for $|s|\leq 2h$.}
Therefore,
$
\|\1_{[t-s,t]}(P)Op_h(\chi)\|\sub{HS}\leq {3}\big\|\psi\big(\tfrac{P_t}{h}\big)Op_h(\chi)\big\|\sub{HS}
$
and the Hilbert--Schmidt norm is the $L^2$ norm of the kernel. 
Next, we recall that after application of a microlocal partition of unity, we may write 
$$
\psi\Big(\tfrac{P_t}{h}\Big)(x,y)=h^{-n}\int_\re\int_{\re^n} \hat{\psi}(\tau)e^{\frac{i}{h}(\varphi(\tau,x,\eta)-\langle y,\eta\rangle-t\tau)}a(\tau,x,y,\eta)d\eta d\tau +O(h^\infty)\sub{HS}
$$
for a symbol $a\sim \sum_j h^ja_j$ and phase $\varphi$ solving
$
\partial_t\varphi=p(x,\partial_x\varphi)
$
and
$
\varphi(0,x,\eta)=\langle x,\eta\rangle.
$
At this point the proof of~\eqref{e:purplePeopleEater} follows exactly as in~\cite[Lemma 3.11]{DyGu:14}.

To obtain~\eqref{e:tomato}, we write
$
P_t\1_{[t-s,t]}(P)=Z_sP_t\psi\big(\tfrac{P_t}{h}\big)
$
and note that 
$
\tfrac{P_t}{h}\psi\big(\tfrac{P_t}{h}\big)=(t\psi)\big(\tfrac{P_t}{h}\big).
$
Hence the same argument applies with $\widehat{t\psi}(\tau)=-i\partial_\tau\hat{\psi}(\tau)$ replacing $\hat{\psi}(\tau)$.
\end{proof}
%%%%%%%%%%%%%%%%%%%%%%%%%%%%%%%%%%%%%%%%%%%%%%%%%%%%%%%%%%%%%%%%%%%%%%%%%%%%%%%
We will also need the following trace bound for $\1_{[t-s,t]}$. 
%%%%%%%%%%%%%%%%%%%%%%%%%%%%%%%%%%%%%%%%%%%%%%%%%%%%%%%%%%%%%%%%%%%%%%%%%%%%%%%
\begin{lemma}
\label{l:advancedTrace}
Suppose $a,b\in \mathbb{R}$, $\e_0>0$, $\ell_1,\ell_2\in \mathbb{R}$, {$\mc{V}_1\subset \Psi^{\ell_1}(M)$, and $\mc{V}_2\subset \Psi_\delta^{\ell_2}(M)$} bounded subsets, {$U\subset T^*M$ open} such that {$|d\pi\sub{M}H_p|>c >0$} on $p^{-1}([a-\e_0,b+\e_0])\cap U$. {Let $\tau_0$,$R_0$, $\delta$, $R(h)$, and $\tau$ be as in Lemma~\ref{l:mainEst}.} Let $\{\T_j\}_{j\in \J(h)}$ be a $(\mathfrak{D},\tau,R(h))$ good covering of ${\bf{p}}^{-1}([a,b])\cap N^*\Delta{\cap (U\times T^*M)}$ and $\mc{V}\subset S_\delta(T^*M\times T^*M; [0,1])$ bounded. Then, there is $C\sub{0}>0$ such that for all $\{\chi\sub{\T_j}\}_{j\in \J(h)}\subset \mc{V}$ partitions for $\{\T_j\}_{j\in \J(h)}$, $j \in \J(h)$,  {$A_1\in \mc{V}_1$, $A_2\in \mc{V}_2$}, and $|s|\leq \e_0$
$$
\Big|\int_{\Delta} Op_h(\chi\sub{\T_j})A_1\1_{[t-s,t]}(P)A_2 d\sigma\sub{\!\Delta}\Big|\leq C\sub{0}h^{1-n}R(h)^{2n-1}\Big\langle \frac{s}{h}\Big\rangle.
$$
\end{lemma}
%%%%%%%%%%%%%%%%%%%%%%%%%%%%%%%%%%%%%%%%%%%%%%%%%%%%%%%%%%%%%%%%%%%%%%%%%%%%%%%
\begin{proof}
We first note that it suffices to prove the statement for $|s|\leq 2h$. Indeed, this is because we may apply the arguments from Lemma~\ref{l:tiramisu} and decompose
$
\1_{[t-s,t]}(P)=\sum_{k=0}^{k_0-1}\1_{[t_k,t_{k+1}]}(P),
$
with $|t_{k+1}-t_k|\leq 2h$. This allows us to obtain the result for $|s|\leq \e_0$.

Suppose $|s|\leq 2h$. Let $\tilde U \supset {B(U,2R(h))}$,
 $j \in \J(h)$, and $A:=Op_h(\chi\sub{\T_j})(A_1 \otimes A_2)$.
 Note that
 \begin{equation} 
\label{e:checkCommutator}
[{\bf{P_t}},A]=[{\bf{P_t}},Op_h(\chi\sub{\T_j})](A_1\otimes A_2)+Op_h(\chi\sub{\T_j})[P-t,A_1]\otimes A_2\in \Psi_\delta(M)
\end{equation}
with seminorms bounded by those of $\chi\sub{\T_j}$, $A_1$, and $A_2$. 
 {We next apply Lemma~\ref{l:mainEst} with $A:=Op_h(\chi\sub{\T_j})(A_1 \otimes A_2)$, ${\bf P_t}$ in place of $P_t$, $k=n$, $M\times M$ in place of $M$, and 
$
u:=\1_{[t-s,t]}(P){Op_h(\chi\sub{\tilde{U}})},
$
where the latter is viewed as a kernel on $M \times M$}.
%Therefore, we may apply Lemmas~\ref{l:mainEst} and~\ref{l:basicEst}}
Here, $\chi\sub{\tilde{U}}\in S_\delta(T^*M)$ with $\chi\sub{\tilde{U}}\equiv 1$ on $B(U,R(h))$, $\supp \chi\sub{\tilde{U}}\subset \tilde{U}$. 
Let $\tilde{\chi}\sub{\T_j}\in \mc{V}$ with $\supp \tilde{\chi}\sub{\T_j}\subset \T_j$ and $\tilde{\chi}\sub{\T_j}\equiv 1$ on $\supp \chi\sub{\T_j}$. 
Then, since $\MSh(A) \subset \T_j$, by Lemma~\ref{l:mainEst} there exist $C\sub{0}>0$ and $C>0$, such that 
$$
h^{\frac{n-1}{2}}\Big|\int_{\Delta} Op_h(\chi\sub{\T_j})A_1\1_{[t-s,t]}(P)A_2 d\sigma\sub{\!\Delta}\Big|
\leq C\sub{0}{R(h)^{\frac{2n-1}{2}}}
 \Big({\|Op_h(\tilde\chi\sub{\T_j})u\|_\LM}+\frac{C}{h}\|Op_h(\tilde \chi\sub{\T_j}){\bf{P}}_{t}u\|_\LM\Big).
%&\qquad +Ch^{-\frac{1}{2}-\delta}\big\|\big(1-\psi \big(\tfrac{{\bf{P}}_{t}}{h^\delta}\big)\big){\bf{P}}_{t}A\1_{[t-s,t]}(P)\big\|_\Hm
%+C\sub{N}h^N\Big(\|\1_{[t-s,t]}(P)\|_\LM+ \|{\bf{P}}_{t}\1_{[t-s,t]}(P)\|_\Hm\Big).
$$
{Note that we omit the analogous error terms appearing in the estimate of Lemma~\ref{l:mainEst} since these error terms can be dealt with by applying the bounds in \eqref{e:bat} and \eqref{e:coffin} in combination with \eqref{e:bee}.}

Next, since $Op_h(\tilde{\chi}\sub{\T_j})=Op_h(\tilde{\chi}\sub{\widetilde{\T}_j})\otimes Op_h(\tilde{\chi}_{\rho_j})$, where $\tilde{\chi}_{\rho_j}$ and $\tilde{\chi}\sub{\widetilde{\T}_j}$ are bounded in $S_\delta(T^*M;[0,1])$ by the seminorms in the set $\mc{V}$, we obtain
\begin{align*} 
&h^{\frac{n-1}{2}}R(h)^{-\frac{2n-1}{2}}\Big|\int_{\Delta} Op_h(\chi\sub{\T_j})A_1\1_{[t-s,t]}(P)A_2d\sigma\sub{\!\Delta}\Big|\\
&\leq C\sub{0}\|Op_h(\tilde{\chi}\sub{\widetilde{\T}_j})u Op_h(\tilde{\chi}_{\rho_j})\|\sub{HS}+ C\sub{0}Ch^{-1}\|Op_h(\tilde{\chi}\sub{\widetilde{\T}_j}) {P}_{t}uOp_h(\tilde{\chi}_{\rho_j})\|\sub{HS}
%&\leq C\sub{0}\Big(\|1_{[t-s,t]}(P)Op_h(\tilde{\chi}_{\rho_j})\|\sub{HS}+ Ch^{-1}\|{P}_{t}\1_{[t-s,t]}(P)Op_h( \chi\sub{\tilde{U}})Op_h(\tilde{\chi}_{\rho_j})\|\sub{HS}\Big)\red{+C\sub{N}h^N\|\1_{[t-s,t]}(P)\|\sub{HS}}\\
\leq C\sub{0}h^{\frac{1-n}{2}}R(h)^{\frac{2n-1}{2}},
\end{align*}
where $u$ is now viewed as an operator. In the last line we used Lemma~\ref{l:HS} and the existence of $C>0$ such that 
$
\mu_t\Big((\supp\tilde{\chi}_{\rho_j})\cap p^{-1}(t)\Big)\leq CR(h)^{2n-1}.
$
This finishes the proof when $|s|\leq 2h$.
\end{proof}
%%%%%%%%%%%%%%%%%%%%%%%%%%%%%%%%%%%%%%%%%%%%%%%%%%%%%%%%%%%%%%%%%%%%%%%%%%%%%%%

\begin{lemma}
\label{l:weylDyn}
Let $a,b,\e_0$, $\tau_0$, $\mc{V}_1,\mc{V}_2$ $R_0$, $\tau$, $\delta$, $R(h)$ and $\alpha$ as in Lemma \ref{l:basicEst}. 
Let ${N^*\Delta\cap (U\times T^*M)}$ be $\Ti$ non-looping for $\p$ in the window $[a,b]$ via {$\tau$-coverings} and let $\Cnp$ be the constant $\Cnl$ in Definition~\ref{d:non loop cov}. Then, there is $C\sub{0}=C\sub{0}(n,P, {\mc{V}_1, \mc{V}_2}, \Cnp,\e_0)>0$ {and for all $K>0$ there is $h_0>0$}
such that
for all $0<h\leq h_0$, {$A_1\in \mc{V}_1$, $A_2 \in \mc{V}_2$ with ${\MSh(A_2)\subset U}$}, $|s|\leq \e_0$, and $t\in [a-Kh,b+Kh]$, 
$$
h^{n-1}\Big|\int_{\Delta}A_1\1_{[t-s,t]}(P)A_2d\sigma\sub{\!\Delta} \Big|^2\leq C\sub{0} \frac{1}{T(h)}\Big\langle \frac{T(h)s}{h}\Big\rangle\|\1_{[t-s,t]}(P){Op_h(\chi\sub{\tilde{U}})}\|^2_{L^2},
$$
{where $\tilde{U}(h)\supset B(U,2R(h))$, $\chi\sub{\tilde{U}}\in S_\delta$, $\chi\sub{\tilde{U}}\equiv 1$ on $B(U,R(h))$, and $\supp \chi\sub{\tilde{U}}\subset \tilde{U}$.}
\end{lemma}
%%%%%%%%%%%%%%%%%%%%%%%%%%%%%%%%%%%%%%%%%%%%%%%%%%%%%%%%%%%%%%%%%%%%%%%%%%%%%%%

\begin{proof}
Decomposing
$\1_{[t-s,t]}(P)=\sum_{k=0}^{k_0-1}\1_{[t_k,t_{k+1}]}(P),$
with $|t_{k+1}-t_k|\leq 2h$ and using the proof of Lemma~\ref{l:tiramisu} to obtain the result for $|s|\leq \e_0$,
it suffices to prove the statement for $|s|\leq 2h$.

From now on we assume $|s|\leq 2h$. Since $N^*\Delta \cap (U\times T^*M)$ is $\Ti$ non-looping in the window $[a,b]$ via {$\tau_0$-coverings}, for all $t\in [a-Kh,b+Kh]$, there is a partition of indices $\J\sub{t}(h)=\G\sub{t,0}(h)\sqcup \G\sub{t,1}(h)$ as described in Definition~\ref{d:non loop cov} {(with $H=\Delta$)}. Let $t_0=t_0$, $t_1=1$, $T_0(h)=T(h)$ and $T_1(h)=1$. Then, there is $\Cnp>0$ such that for all $t\in [a-Kh,b+Kh]$
\begin{equation}\label{e:tired}
\sum_{\ell=0}^1\sqrt{\frac{|\G\sub{t,\ell}(h)|t_\ell}{T_\ell}}\leq \frac{\Cnp R(h)^{\frac{1-2n}{2}}} {\sqrt{T(h)}}, \qquad\qquad \sum_{\ell=0}^1\sqrt{|\G\sub{t,\ell}(h)| t_\ell T_\ell}\leq \Cnp R(h)^{\frac{1-2n}{2}}\sqrt{T(h)}.
\end{equation}

Next, we argue as in~\eqref{e:modifiedTEsts}, and then {apply a combination of Lemma \ref{l:cheat}} and Lemma~\ref{l:basicEst} with { $A:=A_1 \otimes A_2$, ${\bf P_t}$ in place of $P\sub{E}$, $2n$ in place of $n$, $M\times M$ in place of $M$, $k=n$, and $u:=\1_{[t-s,t]}(P){Op_h(\chi\sub{\tilde{U}})}$, where $u$ is viewed as a kernel on $M \times M$}. 
Then, there is $C\sub{0}>0$ so that 
\begin{align*}
&h^{\frac{n-1}{2}}\Big|\int_{\Delta} \!\!\!A_1\1_{[t-s,t]}(P)A_2 d\sigma\sub{\!\Delta}\Big|\leq C\sub{0}R(h)^{\frac{2n-1}{2}}
\bigg(\sum_{\ell=0}^1\!\!\frac{(|\G\sub{t,\ell}(h)|{\tilde{t}}_\ell)^{\frac{1}{2}}}{(\tau \tilde T_\ell)^{\frac{1}{2}}}\|u\|_{L^2} +\!\!\sum_{\ell}\!\!\frac{(|\G\sub{t,\ell}(h)| \tilde{t}_\ell \tilde{T}_\ell)^{\frac{1}{2}}}{h}\|\Pt u\|_{L^2}\!\bigg),
%&\qquad\qquad\qquad \qquad +Ch^{-\frac{1}{2}-\delta}\big\|\big(1-\psi \big(\tfrac{\Pt}{h^\delta}\big)\big)\Pt A_1\otimes A_2^*\1_{[t-s,t]}(P)\big\|_{H_{\textup{scl}}^{\frac{n-2m+1}{2}}}\\
%&\qquad\qquad\qquad \qquad+C\sub{N}h^N\Big(\|\1_{[t-s,t]}(P)\|_{L^2}+ \|P_t \1_{[t-s,t]}(P)w\|_{H_{\textup{scl}}^{\frac{n-2m+1}{2}}}\Big).
\end{align*}
{where $\tilde{t}_\ell$ and $\tilde{T}_\ell$ are as in~\eqref{e:modifiedTs}}.
%As in the proof of Lemma~\ref{l:Lip}, for all $N>0$, there is $C\sub{N}>0$ such that for $t\in[a,b]$, $|s|\leq 1$,
%$$\|\Pt \1_{[t-s,t]}(P)\|_{L^2\to H_{\textup{scl}}^N}\leq C_N|s|,$$
%Furthermore, since $|s|\leq \e_0^{-1}h$, 
%$$\|(1-\psi(h^{-\delta}\Pt))P_tA_1\otimes \1_{[t-s,t]}(P)A_2\|_{L^2\to H^N_{\text{scl}}}\leq C\sub{N}h^N.$$
%Using these estimates above yields the desired result.
%
{We have used that, since $\MSh(A)\subset U \times T^*M$ and the tubes are a covering for ${\bf{p}}^{-1}([a,b])\cap N^*\Delta{\cap (U\times T^*M)}$, then $\MSh(A)\cap \Lambda_{\Sigma^\Delta_t}^\tau(R(h)/2) \subset \bigcup_{j \in \J\sub{t}(h)}\T_j$.
Also, note that we omit the analogous error terms appearing in the estimate of Lemma~\ref{l:basicEst} since these error terms can be dealt with by applying the bounds in \eqref{e:bat} and \eqref{e:coffin} in combination with \eqref{e:bee}.}

{The proof follows from applying the bounds in {\eqref{e:modifiedTEsts}} in combination with \eqref{e:bee}.}
\end{proof}

%%%%%%%%%%%%%%%%%%%%%%%%%%%%%%%%%%%%%%%%%%%%%%%%%%%%%%%%%%%%%%%%%%%%%%%%%%%%%%%

\begin{lemma}
\label{l:weylSobEst}
Let $\ell_i \in \mathbb{R}$, {$\mc{V}_i\subset \Psi_\delta^{\ell_i}(M)$ bounded for $i=1,2$. Then, there are $N_0>0$, $C>0$, $h_0>0$ such that for all $A_1\in \mc{V}_1$ and $A_2\in \mc{V}_2$, $s\in \mathbb{R}$ and $0<h<h_0$}
$$
\Big|\int A_1\1_{(-\infty,s]}(P)A_2 d\sigma\sub{\!\Delta}\Big|\leq Ch^{-\frac{n}{2}}\langle s\rangle^{N_0}\|\1_{(-\infty,s]}(P)\|_{L^2}.
$$
\end{lemma}
%%%%%%%%%%%%%%%%%%%%%%%%%%%%%%%%%%%%%%%%%%%%%%%%%%%%%%%%%%%%%%%%%%%%%%%%%%%%%%%
\begin{proof}
{We apply Lemma~\ref{l:SobEst} with $H=\Delta$, $A=A_1 \otimes A_2$, and $u=\1_{(-\infty,s]}(P)$.} Then, for $r>\frac{n+2(\ell_1+\ell_2)}{2m}$, there is $C>0$ such that for all $N>0$ there is $C_N>0$ such that 
\begin{align*}
h^{\frac{n}{2}}\Big|\int_{\Delta}A_1\1_{(-\infty,s]}(P)A_2 d\sigma\sub{\!\Delta}\Big|&\leq C(\|\1_{(-\infty,s]}(P)\|_{L^2}+\|{\bf{P}}^r \1_{(-\infty,s]}(P)\|_{L^2})+C_Nh^N\|\1_{(-\infty,s]}(P)\|_{L^2}.
\end{align*}
It follows from \eqref{e:bee} that the last term can be bounded by $C(1+|s|^r)\|\1_{(-\infty,s]}(P)\|_{L^2}$.
\end{proof}
%%%%%%%%%%%%%%%%%%%%%%%%%%%%%%%%%%%%%%%%%%%%%%%%%%%%%%%%%%%%%%%%%%%%%%%%%%%%%%%

\subsection{Proofs of Theorems~\ref{t:laplaceWeyl} and \ref{t:genWeyl}}

%%%%%%%%%%%%%

We claim that for ${E}\in[a-Kh,b+Kh]$ and {$A_1\in \mc{V}_1$, and $A_2\in \mc{V}_2$ with $\MSh(A_2)\subset U$},
\begin{equation}\label{e:JeffWillCookForYaicita}
h^{n-1}\Big|\int_{\Delta}A_1\Big(\1_{(-\infty,E]}(P)-\big(\rho_{h,t_0}*\1_{(-\infty,\cdot]}(P)\big)(E)\Big)A_2d\sigma\sub{\!\Delta}\Big|\leq C_0\big/T(h).
\end{equation}

We start by showing under the same assumptions that
\begin{gather}\label{e:ribs}
h^{n-1}\Big|\int_{\Delta}A_1\Big(\big(\rho\sub{h,T(h)}*\1_{(-\infty,\,\cdot\,]}(P) \big)({E})-\1_{(-\infty,{E}]}(P)\Big)A_2 d\sigma\sub{\!\Delta}\Big|\leq C_0\big/T(h),\\
\label{e:sauce}
h^{n-1}\Big|\int_{\Delta} A_1\Big( \big(\rho\sub{h,T(h)}*\1_{(-\infty,\,\cdot\,]}(P)\big){(E)}
-
 \big(\rho\sub{h,t_0}*\1_{(-\infty,\,\cdot\,]}(P)\big){(E)}\Big)A_2d\sigma\sub{\!\Delta} \Big| 
\leq C_0\big/T(h).
\end{gather}
for some $t_0$ independent of $h$. At the end of the section we will derive Theorems~\ref{t:laplaceWeyl} and~\ref{t:genWeyl} from~\eqref{e:JeffWillCookForYaicita}.

\subsubsection{Proof of \eqref{e:ribs}.}

Let $\tilde U,U_0 \subset T^*M$ with {$B(U_0,2R(h))\subset U\subset B(U_0,4R(h)) \subset \tilde U$}. Then, let $\chi\sub{\tilde{U}},\chi\sub{U_0},\chi\sub{\tilde{U}\setminus{U_0}}\in S_\delta(T^*M;[0,1])$ with $\chi\sub{\tilde{U}}\equiv 1$ on $U$, $\supp \chi\sub{\tilde{U}}\subset B(U_0,3R(h))$, $\chi\sub{U_0}\equiv 1$ on $B(U_0, R(h))$, $\supp \chi\sub{U_0}\subset U$, $\chi\sub{\tilde{U}\setminus{U_0}}\equiv 1$ on $\supp \chi\sub{\tilde{U}}(1-\chi\sub{U_0})$, $\supp \chi\sub{\tilde{U}\setminus {U_0}}\subset \tilde{U}\setminus U_0$.
By {Lemma~\ref{l:HS}} {and~\eqref{e:thePoachedChicken}} {there exists $C_0>0$ such that} for {$|s|\leq 2h$},
\begin{equation}
\label{e:volEstMe}
\begin{aligned}
h^{n-1}\big\|\1_{[t-s,t]}(P)Op_h(\chi\sub{\tilde{U}\setminus U_0})\big\|\sub{HS}^2
\leq C_0\mu\sub{{p^{-1}(t)}}(p^{-1}(t)\cap (\tilde{U}\setminus U_0))
\leq C_0 C\sub{U}\big/T(h).
\end{aligned}
\end{equation}
Note that when $U=T^*M$ this is an empty statement.
Then, for ${|s|\leq 2h}$, by Lemma~\ref{l:weylDyn}
\begin{align*}
&h^{n-1}\tr\Big(\1_{[t-s,t]}(P){Op_h(\chi\sub{U_0})}\Big)^2 \Big(\frac{1}{T(h)}\Big\langle \frac{T(h)s}{h}\Big\rangle\Big)^{-1}\leq C\sub{0}\| \1_{[t-s,t]}(P)Op_h(\chi\sub{\tilde{U}})\|^2_{L^2}\\
&\leq 
C\sub{0}\tr \1_{[t-s,t]}(P)Op_h(\chi\sub{U_0})+C\sub{0}\|\1_{[t-s,t]}(P)Op_h(\chi\sub{\tilde{U}\setminus U_0})\|\sub{HS}^2 +C\sub{N}h^N.
%&\leq 
%C\sub{0}\tr\Big( \1_{[t-s,t]}(P){Op_h(\chi\sub{U_0})}\Big)+C_0 C\sub{U}T(h)^{-1} +C\sub{N}h^N.
\end{align*}
Then, applying the quadratic formula with $x=\tr \1_{[t-s,t]}(P){Op_h(\chi\sub{U_0})}$, {for $|s|\leq 2h$} we have 
\begin{align*}
0\leq h^{n-1}\tr \1_{[t-s,t]}(P){Op_h(\chi\sub{U_0})}
%&\leq \frac{\frac{C_0}{T(h)}\big\langle \frac{T(h)s}{h}\big\rangle+\sqrt{\frac{C_0^2}{T^2(h)}\big\langle \frac{T(h)s}{h}\big\rangle^2+4\frac{C_0}{T(h)}\big\langle \frac{T(h)s}{h}\big\rangle h^{n-1}C_0\langle \frac{s}{h}\rangle \vol_M(B(\partial U,R(h)))}}{2h^{n-1}}\\
&\leq \frac{C_0}{T(h)}\Big \langle \frac{T(h)s}{h}\Big\rangle +
\frac{C\sub{U}C_0}{T(h)}+c_Nh^N.
\end{align*}
%Therefore, for $\red{|s|\leq 2h}$,
%%$$
%%\tr\1_{[t-s,t]}(P)1_U\leq h^{1-n}\frac{2C_0}{T(h)}\Big\langle \frac{T(h)s}{h}\Big\rangle.
%%$$
%%In particular, 
%$$
%\red{\|\1_{[t-s,t]}(P)Op_h(\chi\sub{\tilde{U}})\|_{HS}^2\leq h^{1-n}\frac{C_0}{T(h)}\Big\langle \frac{T(h)s}{h}\Big\rangle}.
%$$

{Next, for $|s|\leq \e_0$, splitting $\1_{[t-s,t]}(P)=\sum_{k=0}^{k_0-1}\1_{[t_k,t_{k+1}]}(P)$ as before, } we have by Lemma~\ref{l:weylDyn} and Lemma~\ref{l:weylSobEst} that there exists $N_0>0$ such that 
\begin{gather}
\label{e:diffL2b}
h^{n-1}\Big|\int_{\Delta}A_1\1_{[t-s,t]}(P)A_2d\sigma\sub{\Delta}\Big|\leq C\sub{0}\frac{1}{T(h)}\Big\langle \frac{T(h)s}{h}\Big\rangle,\\
\label{e:totalL2b}
h^{\frac{n}{2}}\Big|\int_{{\Delta}}A_1\1_{(-\infty,s]}(P)A_2d\sigma\sub{\Delta}\Big|\leq C\langle s \rangle^{N_0}\|\1_{(-\infty,s]}(P)\|_{L^2}\leq Ch^{-\frac{n}{2}}(1+|s|^{2N_0}),
\end{gather}
where to get the last inequality, we use Lemma~\ref{l:weylSobEst} with $U=M$, $A_1=A_2=\Id$.

In particular, combining ~\eqref{e:diffL2b} and~\eqref{e:totalL2b} together with Lemma~\ref{l:taub} implies~\eqref{e:ribs} holds.

%for $E\in[a-Kh,b+Kh]$ and \red{$A_1,A_2\in \Psi^\infty$ with $\WFh(A_2)\subset U$},
%$$
%\Big|\int_{\Delta}A_1\Big(\big(\rho\sub{h,T(h)}*\1_{(-\infty,\,\cdot\,]}(P) \big)(E)-\1_{(-\infty,E]}(P)\Big)A_2 d\sigma\sub{\!\Delta}\Big|\leq C_0\frac{h^{1-n}}{T(h)}.
%$$

\subsubsection{Proof of \eqref{e:sauce}.}

 Using Lemma~\ref{l:toShortTime2}, {the proof of \eqref{e:sauce}} amounts to understanding 
$$
 A_1\big(({\rho_{\sub{h,{\Tt(h)}}}}-\rho\sub{h,t_0})*\1_{(-\infty,\,\cdot\,]}(P)\big){(E)}A_2 = A_1 {f_{t_0,\Tt(h),h}}\big(P\sub{E}\big)A_2+O(h^\infty)_{\smooth},
$$
where $f_{S,T,h}$ is given by~\eqref{e:def of f}, and $\tilde{T}(h)=\frac{T(h)}{2}$.
In particular, for $E\in [a-Kh,b+Kh]$, we consider
$
\tr A_1{f_{t_0,\tilde{T},h}\big(P_E\big)}A_2.$
For this, we let $\{\T_j\}_{j\in \J(h)}$ be a $(\mathfrak{D},\tau,R(h))$-good covering of ${\bf{p}}^{-1}([a,b])\cap N^*\Delta\cap (U\times T^*M)$ and $\mc{V}\subset S_\delta(T^*M\times M;[0,1])$ a bounded subset. Let $\{\chi\sub{\T_j}\}_{j\in \J(h)}\subset \mc{V}$ be a partition associated to $\{\T_j\}_{j\in \JE(h)}$. 

%%%%%%%%%%%%%%%%%%%%%%%%%%%%%%%%%%%%%%%%%%%%%%%%%%%%%%%%%%%%%%%%%%%%%%%%%%%%%%%

\begin{lemma}
\label{l:weylVol} 
{Let $\mc{I}\subset \JE(h)$, {$\mc{V}_1\subset \Psi^{\ell_1}(M)$, $\mc{V}_2\subset \Psi_\delta^{\ell_2}(M)$ bounded subsets.} Then, there exist $C\sub{0}>0$ and $h_0>0$ such that for all {$A_1\in \mc{V}_1$, $A_2\in \mc{V}_2$}, $0<h<h_0$}
$$
\Big|\int_{\Delta} \sum_{j\in \mc{I}} Op_h(\chi\sub{\T_j})A_1 {f_{t_0,\tilde{T},h}\big(P\sub{E}\big)}A_2d\sigma\sub{\!\Delta}\Big|\leq C\sub{0} h^{1-n}R(h)^{2n-1}|\mc{I}|.
$$
\end{lemma}
%%%%%%%%%%%%%%%%%%%%%%%%%%%%%%%%%%%%%%%%%%%%%%%%%%%%%%%%%%%%%%%%%%%%%%%%%%%%%%%
\begin{proof}
We first note that 
$
{f_{t_0,\Tt(h),h}(P\sub{E})}= {\varrho_h}*\partial_s \1_{(-\infty,\,\cdot\,]}({P})({E}), 
$
where
$
{\varrho_h(s)}:={f_{t_0,\Tt(h),h}(-s)}.
$
Then, since $\widehat{f_{t_0,\Tt(h)}}(0)=0$, we have
$\int_\re \partial_s\varrho_h(s)ds=0$.
In particular, by the estimates~\eqref{e:bound on f}, Lemma~\ref{l:taub} applies with $\sigma_h=h^{-1}$. Note that by Lemma~\ref{l:advancedTrace}, for $t\in[a-Kh,b+Kh]$, and $|s|\leq 1$,
\begin{equation}\label{e:peach}
\Big|\int_{\Delta}Op_h(\chi\sub{\T_j})A_1(\1_{(-\infty,t]}-\1_{(-\infty,t-s]})A_2 d\sigma\sub{\!\Delta}\Big|\leq Ch^{1-n}R(h)^{2n-1}\Big\langle \frac{s}{h}\Big\rangle.
\end{equation}
Also, by Lemma~\ref{l:weylSobEst}, there exists $N_0$ such that for $s\in \re$, 
\begin{equation}\label{e:pear}
\Big|\int_{\Delta}Op_h(\chi\sub{\T_j})A_1\1_{(-\infty,s]}(P)A_2 d\sigma\sub{\!\Delta}\Big|\leq Ch^{-n}\langle s\rangle^{N_0}.
\end{equation}
{The proof follows from Lemma~\ref{l:taub} using \eqref{e:peach} and \eqref{e:pear}, and by summing in $j\in \mathcal{I}$.}
%applies and the proof follows from summing in $j$ the bounds 
%\begin{equation*}
% \Big|\int_{\Delta}\!\! Op_h(\chi\sub{\T_j})A_1{f_{t_0,\Tt(h),h}\big(P\sub{E}\big)}A_2 d\sigma\sub{\!\Delta}\Big|
%% &=\Big|\int_{\Delta}\!\! Op_h(\chi\sub{\T_j})A_1(\partial_s\rho_h*\1_{(-\infty,s]})A_2 d\sigma\sub{\!\Delta}\Big|
% \leq Ch^{1-n}R(h)^{2n-1}. \qedhere
%\end{equation*}
\end{proof}
%%%%%%%%%%%%%%%%%%%%%%%%%%%%%%%%%%%%%%%%%%%%%%%%%%%%%%%%%%%%%%%%%%%%%%%%%%%%%%%

\begin{lemma}
\label{l:weylLoop}
{Let $\mc{V}_1, \mc{V}_2$ as in Lemma~\ref{l:weylVol}} and suppose $\T_j$ is a tube such that {$\tilde{\T}_j$}, its corresponding tube in $T^*M$, satisfies
$
\varphi_t(\tilde{\T}_j)\cap \tilde{\T}_j=\emptyset
$
for $|t|\in [t_0,T(h)].$
Then for all $N>0$ there is $C\sub{N}>0$ such that for all $A_1\in \mc{V}_1$, and $A_2\in \mc{V}_2$,
$$
\Big|\int_{\Delta} Op_h(\chi\sub{\T_j})A_1{f_{t_0,\tilde{T},h}\big(P\sub{E}\big)}A_2 d\sigma\sub{\!\Delta}\Big|\leq C\sub{N}h^N.
$$
\end{lemma}
%%%%%%%%%%%%%%%%%%%%%%%%%%%%%%%%%%%%%%%%%%%%%%%%%%%%%%%%%%%%%%%%%%%%%%%%%%%%%%%
\begin{proof}
Note that the assumption on $\tilde{\T}_j$ implies 
$
\exp(tH_{\bf{p}})(\T_j)\cap N^*\Delta =\emptyset
$
for
$
|t|\in [t_0,T(h)].
$
Therefore, the same application of Egorov's theorem as in~Lemma~\ref{l:prop}, completes the proof.
\end{proof}
%%%%%%%%%%%%%%%%%%%%%%%%%%%%%%%%%%%%%%%%%%%%%%%%%%%%%%%%%%%%%%%%%%%%%%%%%%%%%%%
 Since ${U}$ is $\Ti$ non-periodic in the window $[a,b]$ via {$\tau$-coverings}, for all $E\in [a-Kh,b+Kh]$, there is a splitting 
$
\JE(h)=\BE(h)\cup\GE(h)
$
such that $\varphi_t(\tilde{\T}_j)\cap \tilde{\T}_j=\emptyset$
for $|t|\in [t_0,T(h)]$ for $j\in \GE(h)$,
and 
$
|\BE(h)|R(h)^{2n-1}\leq T^{-1}(h).
$
We write, using $\MSh(A_1\otimes A_2)\cap \Lambda_{\Sigma_t^\Delta}^\tau(R(h)/2)\subset \bigcup_{j\in \J_{h,E}}\T_j,$
\begin{align*}
\int_{\Delta} A_1{f_{t_0,\tilde{T},h}\big(P\sub{E}\big)}A_2 d\sigma\sub{\!\Delta}
&=\sum_{{j\in \GE(h) \cup \BE(h)}}\int_{\Delta} Op_h(\chi\sub{\T_j})A_1{f_{t_0,\tilde{T},h}\big(P\sub{E}\big)}A_2 d\sigma\sub{\!\Delta}%\\
%&\;\;\;+\sum_{j\in \BE(h)}\int_{\Delta} Op_h(\chi\sub{\T_j})A_1{f_{t_0,\tilde{T},h}\big(P\sub{E}\big)}A_2 d\sigma\sub{\!\Delta}
+O(h^\infty).
\end{align*}
{Applying Lemma~\ref{l:weylLoop} to the sum over $\GE(h)$ and Lemma~\ref{l:weylVol} to the sum over $\BE(h)$}, we have
$$
\Big|\int_{\Delta} A_1{f_{t_0,\tilde{T},h}\big(P\sub{E}\big)}A_2 d\sigma\sub{\!\Delta}\Big|\leq C h^{1-n}|\BE(h)|R(h)^{2n-1}+O(h^\infty)\leq C\big/T(h)
$$
for any $E\in[a-Kh,b+Kh]$. In particular~\eqref{e:sauce} holds.

\subsubsection{Completion of the proof of Theorem~\ref{t:genWeyl}}
\label{s:laplaceGen}

In order to complete the proof of Theorem~\ref{t:genWeyl}, we take $A_1=\Id$ and $A_2=A^{{t}}$ and apply~\eqref{e:JeffWillCookForYaicita} to obtain the theorem. \qed

\subsubsection{Proof of Theorem~\ref{t:laplaceWeyl}}
 \label{s:laplaceWeyl} 
{We assume $\subM\subset M$ is $\Ti$ non-periodic and let $P=Q$ as in~\eqref{e:goodApprox}. Then $|d\pi\sub{M}H_p|>c>0$ on $|\xi|_g>\frac{1}{2}>0$ so we may apply~\eqref{e:JeffWillCookForYaicita} for $E>\frac{1}{2}$.}
Let $0<\delta<\tfrac{1}{2}$.
Let $\chi_h\in C_c^\infty(M)$ as in~\cite[(19)]{CG17} i.e. such that $\chi_h\equiv 1$ in a neighborhood of $\partial \subM$, 
$\supp \chi_h\subset \{ d(x,\partial \subM)<2h^\delta\}$, 
$|\partial_x^\alpha \chi|\leq C_\alpha h^{-|\alpha|\delta},$ $\vol\sub{\!M}(\supp \chi_h)\leq C h^{\delta(n-\dim_{\textup{box}}\partial \subM)}.$

Let $R(h)\geq {h^\delta}$, and $T(h)=\Ti(R(h))$.
Then, put $A_1=1$ and $A_2=(1-\chi_h)1_{\subM}$ in~\eqref{e:JeffWillCookForYaicita} to obtain
$$
\Big|\int_{\Delta}\Big(\1_{(-\infty,1]}(P)-\rho_{h,t_0}*\1_{(-\infty,\cdot]}(P)({1})\Big)(1-\chi_h)1_{\subM}d\sigma\sub{\!\Delta}\Big|\leq C_0h^{1-n}\big/T(h).
$$
Next, since
$
\rho_{h,t_0}*\1_{(-\infty,\cdot]}(P)({1})(x,x)=\frac{\vol_{\mathbb{R}^n}(B^n)}{(2\pi h)^{n}}+O(h^{-n+2})
$
(apply Theorem~\ref{t:laplaceDiag} with $\Ti=1$),
$$
\Big|\int_{\subM} (1-\chi_h(x))\Big({\Pi_h(1,x,x)
}-(2\pi h)^{-n}\vol_{\mathbb{R}^n}(B^n)\Big)\dv_g(x)\Big|\leq C_0h^{1-n}\big/T(h).
$$
Also, since 
$\Pi_h(1,x,x)=(2\pi h)^{-n}\vol_{\mathbb{R}^n}(B^n)|= O(h^{1-n})$
(apply Theorem~\ref{t:laplaceDiag} with $\Ti=\inj M$),
$$
\Big|\int_{\subM}\chi_h(x)\Big(\Pi_h(1,x,x)-(2\pi h)^{-n}\vol_{\mathbb{R}^n}(B^n)\Big)\dv_g(x)\Big|\leq Ch^{1-n+\delta(n-\dim_{\textup{box}}(\partial \subM))},
$$
where we used $ \vol(\supp \chi_h)\leq h^{\delta(n-\dim_{\textup{box}}(\partial \subM))}$.
In particular, 
\begin{equation*}
\Big|\int_{\subM}\Pi_h(1,x,x)\dv_g(x)-(2\pi h)^{-n}\vol_{\mathbb{R}^n}(B^n)\vol_M(\subM)\Big|\leq Ch^{1-n}\Big(T(h)^{-1}+Ch^{\delta(n-\dim_{\textup{box}}\partial \subM)}\Big).\qed
\end{equation*}

%\begin{lemma}[From semiclassical asymptotics to homogeneous asymptotics]
%Let $L:C^\infty(M\times M)\to \mathbb{C}$ be a linear operator and suppose that there are $k> 1$, $h_0>0$, $T:(0,\infty)\to (0,\infty)$ decreasing with $T(h)\geq 1$, and $C_{L,0},C_{L,1}\in \mathbb{R}$ such that for $0<h<h_0$,
%$$
%\Big|L(\Pi_h(2,\cdot,\cdot)-\Pi_h(1,\cdot,\cdot))-(2^k-1)h^{-k}C_{L,0}-(2^{k-1}-1)h^{1-k}C_{L,1}\Big|\leq C\sub{0}\frac{h^{1-k}}{T(h)}.
%$$
%Then, there are $\lambda_0>0$ and $C>0$ such that for $\lambda>\lambda_0$,
%$$
%\big|L\Pi_\lambda -\lambda^{k}C_{L,0}-\lambda^{k-1}C_{L,1}\big|\leq C \frac{\lambda^{k-1}}{T(2\lambda^{-1})}.
%$$
%\end{lemma}
%\begin{proof}
%Let $\lambda>0$ and put $h_j=\lambda^{-1}2^{j}$, $j=1,2,\dots$. Then, let $J=J(\lambda)$ such that $\lambda^{-1}2^{J}>h_0$, and $\lambda^{-1}2^{J-1}<h_0$. Then, 
%\begin{align*}
%&\Big|L\Pi_\lambda(x,y)-\lambda^kC_{L,0}-\lambda^{k-1}C_{L,1}\Big|\\
%&=\Big|L\Pi_{h_{J}}(1,x,y)+\sum_{j=1}^{J(\lambda)-1}\Big(L(\Pi_{h_j}(2,x,y)-\Pi_{h_j}(1,x,y))\Big)-\lambda^kC_{L,0}-\lambda^{k-1}C_{L,1}\Big|\\
%&\leq C_{h_0}+2^{J(\lambda)k}C_{L,0}\lambda^k+2^{J(\lambda)(k-1)}C_{L,1}\lambda^{k-1} +\sum_{j=1}^{J(\lambda)-1}C\sub{0}\frac{\lambda^{k-1}2^{-j(k-1)}}{T(2^{j}\lambda^{-1})}\\
%&\leq C_{h_0}+\frac{C\sub{0}\lambda^{k-1}}{T(2\lambda^{-1})(1-2^{1-k})}.
%\end{align*}
%\end{proof}

%%%%%%%%%%%%%%%%%%%%%%%%%%%%%%%%%%%%%%%%%%%%%%%%%%%%%%%%%%%
%%%%%%%%%%%%%%%%%%%%%%%%%%%%%%%%%%%%%%%%%%%%%%%%%%%%%%%%%%%
%%%%%%%%%%%%%%%%%%%%%%%%%%%%%%%%%%%%%%%%%%%%%%%%%%%%%%%%%%%
\appendix
%%%%%%%%%%%%%%%%%%%%%%%%%%%%%%%%%%%%%%%%%%%%%%%%%%%%%%%%%%%
%%%%%%%%%%%%%%%%%%%%%%%%%%%%%%%%%%%%%%%%%%%%%%%%%%%%%%%%%%%
%%%%%%%%%%%%%%%%%%%%%%%%%%%%%%%%%%%%%%%%%%%%%%%%%%%%%%%%%%%

\section{Index of Notation}\label{s:index}
{In general we denote points in $\TM$ by $\rho$. When position and momentum need to be distinguished we write $\rho=(x,\xi)$ for $x\in M$ and $\xi \in T_x^*M$. The natural projection is $\pi\sub{M}:T^*M \to M$. Sets of indices are denoted in calligraphic font (e.g., $\mc{J}$). Next, we list symbols that are used repeatedly in the text along with the location where they are first defined.
\begin{tabular*}{\textwidth}{l@{\extracolsep{\fill}}lllll}
$\rho_\sigma$ & \eqref{e:rho def lambda}&$E\sub{H_1,H_2}^{A_1,A_2}$ &\eqref{e:gralremainder}&$\mc{K}_\alpha$ &\eqref{e:curvature}\\
$E_\lambda^{t_0}$& \eqref{e:smoothie}&$\Lambda_{_{\!A}}^\tau(r)$ & \eqref{e:tube}&$|H_p r\sub{H}|$ & \eqref{e:HprH}\\
$\Lambda_{\max}$ &\eqref{e:Lmax}&$\mc{Z}$& \eqref{e:mc H}&$\FR$ & \eqref{e:FR}\\
$T_e(h)$ &\eqref{e:Lmax}&$\tau_{\inj}$ & \eqref{e:tau_inj} &$\rho\sub{h, T}$ & \eqref{e:rho def}\\
 $\Sab$ & \eqref{e:SigmaAB}& $\mc{J}\sub{E}(h)$ & \eqref{e:J_E}&$P\sub{E}$ & \eqref{e:pE}
\end{tabular*}
For $U\subset V \subset T^*M$ we write $B_V(U, R)=\{\rho \in V: d(U, \rho)<R\}$ and $B(U,R)=B\sub{T^*M}(U,R)$. 
For $A \subset T^*M$ we write $\mu_A$ for the Liouville measure induced on $A$. The injectivity radius of $M$ is denoted by $\inj M$. For the definitions of the semiclassical objects
$\Psi^\ell(M)$, $\Psi^\ell_\delta(M)$, $S^\ell(T^*M)$, $S^\ell_\delta(T^*M)$, $\WFh$, $\MSh$, $H^N_{\text{scl}}(M)$, 
we refer the reader to \cite[Appendix A.2]{CG19a}. For the definition of $[t,T]$\text{ non-self looping}, see~\eqref{e:nonsl}, that of $(\mathfrak{D},\tau,r)$ good covers, see \eqref{e:good cover}.
} {Non-periodic, non-looping, and non-recurrent are defined in Definitions~\ref{d:non period gral},~\ref{d:non loop gral}, and ~\ref{d:non rec gral} respectively. For non-looping via coverings and non-recurrent via coverings, see Definitions~\ref{d:non loop cov} and~\ref{d:non rec cov}.}

%%%%%%%%%%%%%%%%%%%%%%%%%%%%%%%%%%%%%%%%%%%%%%%%%%%%%%%%%%%
%%%%%%%%%%%%%%%%%%%%%%%%%%%%%%%%%%%%%%%%%%%%%%%%%%%%%%%%%%%
%%%%%%%%%%%%%%%%%%%%%%%%%%%%%%%%%%%%%%%%%%%%%%%%%%%%%%%%%%%

\section{Examples}
\label{a:concrete}

\numberwithin{lemma}{subsection}
In this section, we verify our dynamical conditions in some concrete examples (some of which are displayed in Tables~\ref{ta:exWeyl} and~\ref{ta:ex}). In particular, we verify that certain subsets of manifolds are non-periodic (see Definition~\ref{d:non periodic}), that various pairs of submanifolds $(H_1,H_2)$ are non-looping (see Definition~\ref{d:non looping through}), and that certain submanifolds are non-recurrent {either via coverings (see Definition~\ref{d:non rec cov}) or simply non-recurrent (see Definition~\ref{d:non rec})}. Recall also that if $(H_1,H_1)$ is a non-looping pair, then $H_1$ is non-looping and hence also non-recurrent. Once these conditions are verified, one can directly apply the relevant theorems (Theorem~\ref{t:laplaceWeyl},~\ref{t:laplaceDiag},~\ref{t:laplaceoff}, and~\ref{t:laplaceOff for H}).

\subsection{Manifolds without conjugate points and generalizations}
\label{s:noConj}

Let ${\Xi}$ denote the collection of maximal unit speed geodesics for $(M,g)$. {For $m$ a positive integer, $R>0$, $T\in \R$, and $x \in M$} define
$$
{\Xi}_{x}^{m,R,T}:=\big\{\gamma\in \Xi: \gamma(0)=x,\,\exists\text{ at least }m\text{ conjugate points to } x \text{ in }\gamma(T-R,T+R)\big\},
$$
where we count conjugate points with multiplicity. Next, for a set $\subM \subset M$ write
$$
\mc{C}_{_{\!\subM}}^{m,R,T}:=\bigcup_{x\in \subM}\{\gamma(T): \gamma\in \Xi_{x}^{m,R,T}\}.%\qquad% \mc{C}_{_{\!V}}^{m}:=\bigcap_{r>0}\,\bigcap_{T>0}\,\overline{\bigcup_{t>T}\mc{C}_{_{\!V}}^{m,r,t}}.
$$
Note that if $\Ti(R) \to \infty$ as $R\to 0^+$, then saying $y \in \mc{C}_{x}^{n-1,R,\Ti(R)}$ for $R$ small indicates that $x$ behaves like a point that is maximally conjugate to $y$. Note that if $(M,g)$ has no conjugate points, then $\mc{C}_{x}^{m,r,T} =\emptyset$ for all $x \in M$ and $r<|T|$.

\begin{lemma}
\label{l:noconj}
Let $\alpha>0$, $t_0>0$ and $\Ti(R)=\alpha\log R^{-1}$. Then there are $\Cnl>0$ and $c>0$ such that if $H_1,H_2\subset M$ of co-dimension ${k_1},{k_2}$, and 
$$
d\big(H_1,\mc{C}_{H_2}^{{k_1+k_2}-n-1,R,\Ti(R)}\big)>R
$$
%(where $r_t=\tfrac{1}{a}e^{-at}$ {old}), 
for all $R<e^{-t_0/\alpha}$,
then $(H_1,H_2)$ is a $(t_0,c\log R^{-1})$ non-looping pair with constant $\Cnl$, for $p(x,\xi)=|\xi|_{g(x)}$.
\end{lemma}
\begin{proof}
By~\cite[Proposition 2.2, Lemma 4.1]{CG19dyn} there exist $\tau>0$, $\delta>0$, $\Cnl>0$, $C>0$, such that the pair $(H_1,H_2)$ is a $(t_0,T(h))$ non-looping via $(\tau,h^\delta)$ coverings with constant $\Cnl$ in the window $[a,b]$ for any $0<a<b$, where
$
T(h)=c\log h^{-1}
$
for some $c>0$ depending on $(M,g,\alpha)$. Combining this result with Lemma~\ref{l:4minutes} {completes the proof}.
\end{proof}

\begin{remark}
We note that~\cite[Proposition 2.2]{CG19dyn} was only proved for $H_1=H_2$. However, the same argument works for the general case. 
\end{remark}

\subsubsection{Product Manifolds}
\label{s:prod}

Let $(M_i,g_i)$, $i=1,2$, be two compact Riemannian manifolds. Let $M=M_1\times M_2$ endowed with the product metric $g=g_1\oplus g_2$. By~\cite[Lemma 1.1]{CG19a} we have $\mc{C}^{n-1,r,T}_{x}=\emptyset$ for $0<r<|T|$. Therefore, by Lemma~\ref{l:noconj} for every $\alpha,t_0>0$ there is $\Cnl$ such that every $x\in M$ is $(t_0,\alpha \log R^{-1})$ non-looping with constant $\Cnl$ for $|\xi|_{g(x)}$. Note that, integrating over $M$, {and using}
$$
{\mu\sub{S^*M}(A)=\int_{M}\mu\sub{S^*_xM}(A\cap S^*_xM)\dv_g},
$$
this also implies $M$ is $\alpha \log R^{-1}$ non-periodic. We point out that although $\mc{C}^{{n-1},r,T}_x$ is empty for $0<r<|T|$, $M$ may, and often does, have conjugate points. For example, this is the case when $M^1=S^{n_1}$ with $n_1\geq 2$. 

\subsubsection{Flow invariance of non-looping condition}
\label{s:flowInv}
In this section, we show that non-looping properties of a pair $(H_1,H_2)$ are inherited by their flow-outs $H^t:=\pi (\varphi_t(\SNH)).$ Note, for example, that a geodesic sphere is given by $H^t$ when $H=\{x\}$ is a point for some $t>0$.
\begin{lemma}
\label{l:flowInv}
Suppose $(H_1,H_2)$ is a $(t_0,\Ti)$ non-looping pair. Then, for all $s,t\in \mathbb{R}$ there exists $C>0$ such that $(H^t_1,H^s_2)$ is a $(t_0+|t|+|s|,\tilde{\Ti})$ non-looping pair where $\tilde{\Ti}(R)=\Ti(CR)-(|t|+|s|).$
\end{lemma}
\begin{proof}
First, note that $\SNH^t_j=\varphi_t(\SNH_j)\cup \varphi_{-t}(\SNH_j)$ for $j=1,2$. Let $T>0$ and suppose $\rho \in B(\smalleq{\mc{L}^{{R,1}}\sub{H^t_1,H^s_2}(t_0,T)},R)$. Then, there is $q_1\in \mc{L}^{{R,1}}\sub{H^t_1,H^s_2}(t_0,T)$ such that $d(q_{{1}},\rho)<R$. In particular, there are $q_2\in T^*M$ and $t_0\leq |t_1|\leq T$ such that $d(q_1,q_2)< R$ and 
$
d(\varphi_{t_1}(q_2), \SNH^s_2)<R.
$

Now, either $\varphi_{-t}(q_1)\in \SNH_1$ or $\varphi_t(q_1)\in \SNH_1$. We consider the case $\varphi_{t}(q_1)\in \SNH_1$, the other begin similar. Then, there exist $C_t,C_s>0$ such that
$$
d(\varphi_t(q_1),\varphi_t(q_2))<C_tR,\qquad d(\varphi_{-t+t_1\pm s}\circ\varphi_t(q_2),\SNH_2)<C_sR.
$$
In particular, letting $C=\max(C_t,C_s)$,
$
\varphi_t(q_1)\in \mc{L}\sub{H_1,H_2}^{CR}(t_0+|t|+|s|,T-(|t|+|s|)),
$
and, since $d(\varphi_t(\rho),\varphi_t(q_1))<CR$, 
$$
\varphi_t(\rho)\in B\big(\smalleq{\mc{L}\sub{H_1,H_2}^{{CR,1}}(t_0+|t|+|s|,T-(|t|+|s|))},CR\big).
$$
Repeating this argument when $\varphi_{-t}(q_1)\in \SNH_1$, we obtain
$$
B\sub{\SNH_1^t}(\smalleq{\mc{L}\sub{H^t_1,H^s_2}^{{R,1}}(t_0,T)},R)\subset \bigcup_{\pm}\varphi_{\pm t}\big(B\sub{\SNH_1}(\smalleq{\mc{L}\sub{H_1,H_2}^{{CR,1}}(t_0+|t|+|s|,T-(|t|+|s|))},CR)\big).
$$
In particular, {there is $C>0$} such that 
\begin{align*}
&\mu\sub{\SNH_1^t}\Big(B\sub{\SNH_1^t}(\smalleq{\mc{L}\sub{H^t_1,H^s_2}^{{R,1}}(t_0,T)},R)\Big)\leq \sum_{\pm}C\mu\sub{\SNH_1}\Big(B\sub{\SNH_1}(\smalleq{\mc{L}\sub{H_1,H_2}^{{CR,1}}(t_0+|t|+|s|,T-(|t|+|s|))},CR)\Big).
\end{align*}
Therefore, since $(H_1,H_2)$ is a $(t_0,\Ti)$ non-looping pair, $(H_1^t,H_2^s)$ is a $(t_0+|t|+|s|,\tilde{\Ti})$ non-looping pair with $\tilde{\Ti}(R)=\Ti(CR)-|t|-|s|.$
\end{proof}

Now, by Lemma~\ref{l:noconj}, in the case 
$
d\big(y,\mc{C}_x^{n-1,R,\Ti(R)}\big)>R, 
$
for $R<e^{-t_0/\alpha}$ and $\Ti(R)=\alpha \log R^{-1}$, 
we have $(x,y)$ is a $(t_0,c\log R^{-1})$ non-looping pair. Hence, by Lemma~\ref{l:flowInv} that the geodesic spheres generated by $x$ and $y$ form a non-looping pair with {\resfun} $\Ti(R)=\tilde{C}\log R^{-1}$ for some $\tilde C>0$.

\subsection{Surfaces of revolution}
\label{s:perturb}
Consider $M=S^2$ with the metric a
$
\iota^*g
$
where 
\begin{equation}
\label{e:revolveMetric}
g(s, \theta)=ds^2+\alpha^2(s)d\theta^2,
\end{equation}
and $\iota:[-\frac \pi 2,\frac \pi 2]\times \re/2\pi \mathbb{Z}\to S^2$, with
$
\iota(s,\theta)=(\cos(s)\cos(\theta),\cos(s)\sin(\theta),\sin(s)).
$
%surface of revolution $$ parametrized by coordinates $(s, \theta)\in [-\frac\pi 2 \frac \pi 2]\times S^1$, with metric 
Here, $\alpha$ is a smooth function satisfying $\alpha(\pm\frac{\pi}{2})=0$ and $\pm \alpha'(\pm \pi/2)=1$. 
This assumption implies $g$ is a smooth Riemannian metric. Furthermore, we assume $-s\alpha'(s)>0$ for $s\neq 0$ and $\alpha''(0)<0$. Note that the round sphere is given by $\alpha(s)=\cos (s)$.

For a unit speed geodesic, $t \mapsto (s(t),\theta(t))$ with $(s(0),\theta(0))=(0,0)$, $\dot{\theta}(0)>0$, $\dot{s}(0)>0$, we have by the Clairaut formula (see e.g.~\cite[Proposition 4.7]{Besse})
$$
\big(\dot s(t)\big)^2+\alpha^2(s(t))\big(\dot \theta(t)\big)^2=1
\qquad \text{and }\qquad 
 \dot \theta(t)=\alpha(s_+)\alpha^{-2}(s(t))
$$
where $s_+$ is the maximal value of $s$ on the geodesic. 
In particular, putting $t(s_+)$ for the first time when $s(t)=s_+$, we have $s:[0,t(s_+)]\to [0,s_+]$ is invertible, 
$$
t(s)=\int_0^{s}\frac{\alpha(w)}{\sqrt{\alpha^2(w)-\alpha^2(s_+)}}dw,\qquad 
\theta(t(s_+))=\int_0^{t(s_+)}\frac{\alpha(s_+)}{\alpha^2(s(t))}dt
$$
and, changing variables to $w=s(t)$ and using 
$
\dot s(t)=\sqrt{1-\frac{\alpha^2(s_+)}{\alpha^2(s(t))}},
$
we have
$$
\theta(t(s_+))=\int_0^{s_+}\frac{\alpha(s_+)}{\alpha(w)}\frac{1}{\sqrt{\alpha^2(w)-\alpha^2(s_+)}}dw.
$$
We then define $\theta_+(s_+):=2\theta(t(s_+))$. If we instead suppose $\dot \theta>0$ and $\dot{s}<0$, we can define $\theta_-(s_-)$ analogously where $s_-$ is the minimal $s$ value on the trajectory. 
Now, there is a smooth function
$$
s_-:[0,\pi/2]\to [-\pi/2,0]
$$
such that if $s_+$ is the maximal $s$ value of a trajectory, then $s_-(s_+)$ is the minimal $s$ value. Moreover, $\partial_{s_+}s_-<0$.

Finally, note that for a trajectory with maximal $s$ value $s_+$, $s(0)=0$, $\dot{s}\neq 0$, if $T$ is the second return time to $s(0)=0$, then 
$$
\Theta_0(s_+)=\theta(T)-\theta(0),\qquad \Theta_0(s_+):= \theta_+(s_+)+\theta_-(s_-(s_+)).
$$
Note that apriori, $\theta(T)-\theta(0)$ could depend on the precise geodesic whose maximal $s$ value is $s_+$. However, the integrable torus, $\mathbb{T}_{s_+}$, consisting of all such geodesics has the same $\theta(T)-\theta(0)$ up to sign.

In the next lemmas, we reduce the study of dynamical properties on $(M,g)$ to the Poincar\'e section $\{s(0)=0,\dot{s}(0)>0\} \subset TM$. The function $\Theta_0:(0, \pi/2] \to \re$ is the change in $\theta$ after a return to the Poincar\'e section. In particular, $\mathbb{T}_{s_+}$ is a periodic torus (i.e. all its trajectories are periodic) if and only if for some $p,q\in \mathbb{Z}$, $q\neq 0$,
$$
\Theta_0(s_+)= 2\pi\, p/q.
$$

\begin{lemma}
\label{l:non-periodSurf}
Suppose there exists $b> 0$ such that 
$$
\partial_{s_+}\Theta_0(s_+)\neq 0,\qquad s_+\geq b.
$$
Then, there are $\Cnp,c>0$ such that every subset $U\subset \{s>b\}\cup \{s<s_-(b)\}$ is $\Ti$ non-periodic for $\Ti(R)=cR^{-1/3}$ with constant $\Cnp$.
\end{lemma}
\begin{proof}

Suppose $\rho\in S^*M$ with $s_+(\rho)>b$, and {let $t\in \re$ be such that}
\begin{equation}\label{e:rhoLoops}
\varphi_t(B\sub{S^*M}(\rho,R))\cap B\sub{S^*M}(\rho,R)\neq \emptyset.
\end{equation}
Then, there is $|t_1|\leq R$ such that 
$
d(\varphi_{t+t_1}(\rho),\rho)< (1+C (|t|+|t_1|))R.
$
Now, for some $0\leq t_2\leq c$, we have
$
{s}(\varphi_{t_2}(\rho))=0
$
and 
$$
d(\varphi_{t+t_1+t_2}(\rho),\varphi_{t_2}(\rho))< (1+C(|t|+|t_1|+t_2))R.
$$

Let $s_+$ be the maximal $s$ value for the trajectory through $\rho$. 
Then, there are $p,q\in \mathbb{Z}$ with $|p|,|q|\leq C(1+|t|)$, $|q|\geq c(1+|t|)$ such that 
\begin{equation}
\label{e:nearlyRational}
\Big|\Theta_0(s_+)-2 \pi\, p/q\Big|<C(1+C (|t|+|t_1|+t_2))R/q\leq CR.
\end{equation}

We have shown that if $\rho \in S^*M$ is such that \eqref{e:rhoLoops} holds, then $\rho \in \bigcup_{s_+ \in A(t)} \mathbb T_{s_+}$,
where 
 $$A(t):=\{s_+ \in (b, \tfrac{\pi}{2}]:\; \exists p,q \in \Z,\; |p|,|q|\leq C(1+|t|), \; \eqref {e:nearlyRational}\; \text{holds}, \}.$$
Next, we claim
\begin{equation}\label{e:vol s+}
|A(t)| \leq C(1+|t|)^2R.
\end{equation}
Indeed, ${\#}\{r \in [0,1]:\; \exists p,q \in \Z,\; r=p/q,\;\; |p|,|q|\leq C(1+|t|)\} \leq C(1+|t|)^2$ and hence, the volume of possible values of $\Theta_0(s_+)$ 
such that \eqref{e:nearlyRational} holds is bounded by
$
C(1+|t|)^2R.
$
The claim in \eqref{e:vol s+} then follows from the assumption $\partial_{s_+} \Theta_0(s_+)\neq 0$ on $s_+\geq b$. % the volume of $s_+$ such that~\eqref{e:nearlyRational} holds with $|p|,|q|<C(1+|t|)$ satisfies, $$V(t)\leq C(1+|t|)^3R.$$

Our next goal is to show that the bound in \eqref{e:vol s+} translates to a bound on the set of $\rho$ with \eqref{e:rhoLoops}.
 To see this, note that $\mathbb{T}_{s_+}=\{ |\xi_\theta|=\alpha(s_+)\}{\cap S^*M}$ where we work in the cotangent bundle with coordinates $(s, \theta, \xi_s,\xi_\theta)$. Therefore, when $\alpha(s_+)<\alpha(s_0)$, the intersection $\mathbb{T}_{s_+}\cap S^*_{(s_0,\theta)}M$ is transversal for any $\theta$. In particular, for any $\e>0$ and $s_0\geq 0$, there exists $C_\ep>0$ such that for any $A\subset [s_0+\e, \pi/2]$
$$
\mu\sub{S^*_{(s_0,\theta)}M}\Big(\bigcup_{s_+\in A}\mathbb{T}_{s_+}\cap S^*_{(s_0,\theta)}M\Big)\leq C_\e |A|.
$$
Moreover, since there is $T>0$ such that the restriction of the map $ (t,q)\mapsto \varphi_t(q)$ 
$$[-T,T]\times 
\Big( \bigcup_{\substack{s_+\geq s_0+\e\\{\theta\in[0,2\pi]}}}S^*_{(s_0,{\theta})}M\cap\mathbb{T}_{s_+}\Big) \to \bigcup_{s_+{\geq} s_0+\e}\mathbb{T}_{s_+}
$$
is a surjective local diffeomorphism,
\begin{equation}
\label{e:awayTangent}
 \mu\sub{S^*M}\Big(\bigcup_{s_+\in A}\mathbb{T}_{s_+}\cap S^*M\Big)\leq C_\ep |A|.
\end{equation}
In particular, by \eqref{e:vol s+}, since $b>0$, there exists $C_b>0$ such that
$$
\mu\sub{S^*M}\Big(\bigcup_{s_+ \in A(t)}\mathbb{T}_{s_+}\cap S^*M\Big)\leq C_b \,|A(t)| \leq C_b (1+|t|)^2R.
$$
%
%
%Since all of the Lagrangian tori passing over $\{s>b\}\cup\{s<s_-(b)\}$ are well parametrized by $s_+$, i.e. there volume is controlled by the volume of the corresponding $s_+$, we have that $\rho$ belongs to a set of integrable tori with volume bounded by$C(1+|t|)^3R. $
%
Hence, for $U\subset \{s>b\}\cup \{s<s_-(b)\}$, 
$$
\mu\sub{S^*\!M}\!\Big(B\sub{S^*\!M}\!\big({\smalleq{ \mc{P}^R_U(t_0,\Ti(R))}}, R\big)\Big)\leq C(1+|\Ti(R)|)^2R. $$
So, provided $\Ti(R)\leq R^{-1/3}$, $U$ is $\Ti(R)$ non-periodic with constant $\Cnp=C/2$.
\end{proof}

{
\begin{lemma}
\label{l:nonloopSurf}
Suppose $x_0$ is a pole, and $x_1=(s_1,\theta_1)$ for $-\pi/2<s_1<\pi/2$. Then, there is $\Cnl>0$ such that $(x_0,x_1)$ is a $\Ti(R)=R^{-1}$ non-looping pair.
\end{lemma}
\begin{proof}
Suppose $x_0$ is the pole with $s=\pi/2$. Suppose $\rho\in S^*_{x_1}M$ and there exists $\rho_1\in S^*_{x_1}M$ such that $d(\rho,\rho_1)<R$ and $\varphi_t(B(\rho_1,R))\cap B(S^*_{x_0}M,R)\neq \emptyset$. Then, there is $\rho_2\in B(\rho_1,R)$ such that $s_+(\rho_2)>\pi/2-R$. Therefore, there is $C>0$ such that $s_+(\rho)>\pi/2-CR$ and (since $|s_1|<\pi/2$), 
$$
\mu\sub{S^*_{x_1}M}\Big(\bigcup_{s_+>\pi/2-CR}\mathbb{T}_{s_+}\cap S^*_{x_1}M\Big)\leq CR.
$$
In particular, for any $t_0>0$, $T>0$,
$$
\mu\sub{S^*_{x_1}M}\Big(B(\smalleq{\mc{L}_{x_1,x_0}^{{R,1}}(t_0,T)},R)\Big)\leq CR
$$
and hence $(x_0,x_1)$ is a $\Ti(R)=R^{-1}$ non-looping pair.
\end{proof}
}

\begin{lemma}
Suppose the assumptions of Lemma~\ref{l:non-periodSurf} hold and $x_0=(s_0,\theta_0)$ with $s\in (-\pi/2, s_-(b))\cup (b,\pi/2).$ Then there is $\delta>0$ such that $x_0$ is $\Ti(R)=R^{-\delta}$ non-looping.
\end{lemma}
\begin{proof}
The proof is identical to~\cite[Lemma 5.1]{CG19a}.
\end{proof}

\subsubsection{Perturbed spheres}
\label{s:perturbEx}

Next, we construct examples which have large (positive measure) periodic sets as well as large non-periodic sets. In particular, we find examples where the assumptions of Lemma~\ref{l:non-periodSurf} hold and such that there is $c>0$ {with the property that} the flow is periodic on $-c<s<c$. If $s_0>0$, we will call $(s_0,\theta_0)$ \emph{aperiodic} if 
$$
\partial_{s_+}\Theta_0(s_+)\neq 0\text{ on }\{s_+\geq s_0\}.
$$
In the case $s_0<0$, we require the same condition on $\{\alpha(s_+)\leq \alpha(s)\}$. We define the \emph{aperiodic set} to be the set of aperiodic points and Theorem~\ref{t:laplaceWeyl} holds for any $U$ inside this set.

In order to do this, we make a small perturbation of the round metric ($\alpha(s)=\cos s$). 
First, we compute 
\begin{align*}
\partial_{s_+}\theta_+&=2\alpha'(s_+)\int_{a}^{s_+}\big[\alpha^2(w)-2\alpha^2(s_+)\big]\frac{2(\alpha'(w))^2+\alpha(w)\alpha''(w)}{\sqrt{\alpha^2(w)-\alpha^2(s_+)}\alpha^3(w)(\alpha'(w))^2}dw\\
&\qquad -2\alpha'(s_+)\frac{\alpha^2(b)-2\alpha^2(s_+)}{\sqrt{\alpha^2(b)-\alpha^2(s_+)}\alpha^2(b)\alpha'(b)}+2\alpha'(s_+)\int_0^b\frac{\alpha(w)}{(\alpha^2(w)-\alpha^2(s_+))^{3/2}}dw.
\end{align*}
Let $0<a<b<\pi/2$ and $\alpha_\e=\alpha_0+\e (f_++f_-)$, with $\supp f_+\subset (a,b)$ and $\supp f_- \subset (-\pi/2,0)$. We have for $s_+\geq b$, 
$$
\partial_\e \partial_{s_+}\theta_+\Big|_{\e=0}=-2\alpha_0'(s_+)\int_0^b f_+(w)\frac{2\alpha_0^2(w)+\alpha_0^2(s_+)}{(\alpha_0^2(w)-\alpha_0^2(s_+))^{5/2}}dw.
$$

Arguing identically for $\theta_-$, if $\alpha_\e=\alpha_0+\e(f_++f_-)$ with $\supp f_-\subset (s_-(b),s_-(a))$ and $\supp f_+\subset (0,\pi/2)$, then 
$$
\partial_{\e}\partial_{s_-}\theta_-\Big|_{\e=0}=-2\alpha_0'(s_-)\int_{-b}^0f_-(w)\frac{2\alpha_0^2(w)+\alpha_0^2(s_-)}{(\alpha_0^2(w)-\alpha_0^2(s_-))^{5/2}}dw.
$$

To construct an example where the assumptions of Lemma~\ref{l:non-periodSurf} hold, let $\alpha_0(s)=\cos(s)$ so that $\alpha_0$ induces the standard round metric. Let $0<a<b<\frac{\pi}{2}$, $f_+$ not identically 0 and $f_+\geq 0$ with $\supp f_+\subset (a,b)$, and let $f_-\geq 0$ with $\supp f_-\subset (s_-(b),s_-(a))$. Then, we have for $s_+\geq b$, and $\Theta_{0,\e}$ corresponding to the perturbed metric with $\alpha_\e$,
$$
\partial_{\e}\partial_{s_+}\Big(\Theta_{0,\e}(s_+)\Big)>0,\qquad s_+\geq b.
$$
In particular, we may choose $\e_0>0$ small enough such that for $0<\e<\e_0$ and $\alpha=\alpha_{\e}$, we have $-s\alpha_\e'(s)>0$ when $s\neq 0,$
and
$$
\partial_{s_+}\Big(\Theta_{0,\e}(s_+)\Big)>0,\qquad s_+\geq b.$$

Moreover, since $\alpha_0$ is the round metric on the sphere, the flow is periodic for trajectories not leaving $(s_-(a),a)$. (See Figure~\ref{f:perturb})

\subsubsection{The spherical pendulum}
\label{s:spherePend}
{We now recall the spherical pendulum on $S^2$ whose Hamiltonian is given in the $(s,\theta)$ coordinates by
$$
q(s,\theta,\xi_s,\xi_\theta)=\xi_s^2+\cos^{-2}(s)\xi^2_\theta+2\sin s -E.
$$
 This Hamiltonian describes the movement of a pendulum of mass $1$ moving without friction on the surface of a sphere of radius $1$. When $E>2$, up to reparametrization of the integral curves, the dynamics for the spherical pendulum are equivalent to those for the Hamiltonian $p=|\xi|^2_{\iota^*g}$ and $g$ is given by 
$$
g=(E-2\sin(s))ds^2+(E-2\sin(s))\cos^2(s)d\theta^2.
$$
Making a further change of variables in the $s$ variable, we can put the metric in the form~\eqref{e:revolveMetric} and, moreover, by~\cite{Horozov} for $E\geq \frac{14}{\sqrt{17}}$, $|\partial_{s_+}\Theta_0|> c>0$ for $s_+\in(0,\pi/2]$. Note that the failure of this condition at the torus $\mathbb{T}_{0}$ is due to the fact that this torus is singular, consisting of the two curves $\{s=0,\theta\in \mathbb{R}/2\pi \mathbb{Z}, \xi_r=0,|\xi_\theta|=\alpha(0)\}$. In fact, it is easy to see that $|\Theta_0(s_+)|>cs_+^{1/2}$ for $s_+$ near $0$. This, together with Lemmas~\ref{l:non-periodSurf} and~\cite[Lemma 5.1]{CG19a} are enough to obtain the results in Table~\ref{ta:ex} and that Theorem~\ref{t:laplaceWeyl} applies to the spherical pendulum with $U=M$.}

 \subsection{Submanifolds of manfiolds with Anosov geodesic flow}
 
 \label{s:anosov}
 We next recall some examples when $(M,g)$ has Anosov geodesic flow. The geodesic flow is Anosov if there is ${\bf{B}}>0$ such that for all $\rho\in T^*M$ there is a splitting 
 $$
 T_\rho T^*M=E_+(\rho)\oplus E_-(\rho)\oplus \mathbb{R}H_p(\rho)
 $$
 such that 
\begin{equation*}
 |d\varphi_t({\bf{v}})|\leq {\bf{B}} e^{\mp \frac{t}{{\bf{B}}}}|{\bf{v}}|,\qquad{ {\bf{v}}\in E_\pm(\rho),\quad t\to \pm\infty,}
\end{equation*}
where $|\cdot|$ is the norm induced by a Riemannian metric on $T^*M$. Here, $E_+(\rho)$ is called the stable space and $E_-(\rho)$, the unstable space.

 We also note (see~\cite{Eberlein73,Kling}) that a manifold with non-positive sectional curvature has no conjugate points and that 
 $$\text{negative sectional curvature}\quad \Rightarrow\quad\text{Anosov geodesic flow}\quad\Rightarrow\quad\text{no conjugate points}.
 $$
 Note that these implications are \emph{not} equivalences. Indeed, there exist manifolds with Anosov geodesic flow containing sets with strictly positive sectional curvature as well as manifolds with no conjugate points which do not have Anosov geodesici flow. 
 
 One of the main goals of~\cite{CG19dyn} was to prove that various submanifolds of manifolds with the Anosov or non-focal property are non-recurrent via coverings. We will review only some of these results here, referring the reader to~\cite{CG19dyn} for further examples. In what follows we present several dynamical lemmas which yield the statements from Table~\ref{ta:ex}.

Define for a submanifold $H\subset M$, and for every $\rho \in \SNH$
 $$
 m_{\pm}(H,\rho):=\dim ( E_\pm(\rho)\cap T_\rho \SNH).
 $$
 Note that in two dimensions $m_{\pm}(H,\rho)\neq 0$ is equivalent to $H$ being tangent to, and having the same curvature as, a stable/unstable horosphere with conormal $\rho$. 
In fact, in any dimension, a generic $H \subset M$ satisfies $m_{\pm}(H,\rho)=0$ for all $\rho \in \SNH$.

 \begin{lemma} Let $H\subset M$ be a smooth submanifold .
 Suppose $(M,g)$ is a manifold with Anosov geodesic flow and for all $\rho \in \SNH$ 
 $$
 m_{+}(H,\rho) +m_-(H,\rho)<n-1\qquad\text{or}\qquad m_-(H,\rho)m_+(H,\rho)=0.
 $$
 %or that $(M,g)$ has no focal points and for all $\rho \in \SNH$
 %$$
 % m_{+}(H,\rho) +m_-(H,\rho)<n-1.
 % $$
 Then there are $c,\delta,\tau>0$ such that {for all $0<a<b$,} $H$ is $c\log h^{-1}$ non-recurrent via $(\tau,R(h))$ coverings {for the symbol $p(x,\xi)=|\xi|_{g(x)}$ in the window $[a,b]$}. 
 \end{lemma}
 \begin{proof}
 The proof of this result is that of~\cite[Theorem 6]{CG19dyn}, see~\cite[Section 5.1]{CG19dyn}.
 \end{proof}

 \begin{lemma}
 Suppose $(M,g)$ is a manifold with Anosov geodesic flow and $H_1,H_2\subset M$ are a smooth submanifolds such that for $i=1,2,$
 $
 \sup_{\rho \in \SNH_i}m_{\pm}(H_i,\rho)=0.
 $
 Then there are $c,t_0>0$ such that for all $0<a<b$, $(H_1,H_2)$ is a $(t_0,c\log R)$ non-looping pair for $p(x,\xi)=|\xi|_{g(x)}$ in the window $[a,b]$.
 \end{lemma}
\begin{proof}
By~\cite[Proposition 2.2, Lemma 5.1]{CG19dyn} (in particular, adapting the arguments in~\cite[``Treatment of $D\in \{D_i\}_{i\in \mc{I}_K}$", page 38]{CG19dyn}) there exist $\tau>0$, $\delta>0$, $\Cnl>0$, $C>0$, such that the pair $(H_1,H_2)$ is a $(t_0,T(h))$ non-looping via $(\tau,h^\delta)$ coverings with constant $\Cnl$ in the window $[a,b]$ for any $0<a<b$, where
$
T(h)=c\log h^{-1}
$
for some $c>0$ depending on $(M,g,\alpha)$. Combining this result with Lemma~\ref{l:4minutes} yields the claim.
\end{proof}

Recall that a stable/unstable horosphere is defined by the property that $T_\rho\SNH=E_\pm(\rho)$ for all $\rho\in \SNH$.

 \begin{lemma}
 \label{l:horrorsphere}
 Suppose $(M,g)$ is a manifold with Anosov geodesic flow, $H_{\pm}\subset M$ is a {compact} subset of a stable/unstable horosphere and $H_2\subset M$ is a submanifold with $m_{\pm}(H_2,\rho)<n-1$ for all $\rho \in \SNH_2$. Then, there are $c,t_0>0$ such that for all $0<a<b$, $(H_\pm,H_2)$ is a $(t_0,c\log R)$ non-looping pair for $p(x,\xi)=|\xi|_{g(x)}$ in the window $[a,b]$.
 \end{lemma}
For simplicity, we prove only Lemma~\ref{l:horrorsphere} but point out that the arguments similar to those in~\cite[Lemma 5.1]{CG19dyn} can be used to obtain much more general statements.
 \begin{proof}
 We consider the case $H_+$. The other case following identically. By Lemma~\ref{l:4minutes} it suffices to show $(H_{+},H_2)$ is a non-looping pair via coverings. Thus, by~\cite[Proposition 2.2]{CG19dyn} and Lemma~\ref{l:4minutes} it suffices to show there exists $\alpha>0$ such that {for all} $(t, \rho) \in [t_0, T_0]\times \SNH_+$ such that $d(\varphi_t(\rho),\SNH_2)\leq \, e^{-{\alpha}|t|}/\alpha$, there exists ${\bf w} \in T_{\rho}\SNH_+$ for which the restriction
$$ 
d\psi_{(t,\rho)}: \R \partial_t \times \R {\bf w} \to T_{\psi(t,\rho)}\re^{n+1} 
$$ 
has left inverse $L_{(t, \rho)}$ with $\|L_{(t, \rho)}\|\leq \alpha e^{ \alpha |t|}$. Here, $\psi:\re \times \SNH_+\to \re^{n+1}$ is given by
$
\psi(t,\rho)=F\circ\varphi_t(\rho)
$
and $F:T^*M \to \R^{n+1}$ is a defining function for $\SNH_2=F^{-1}(0)$.

Note that $T_\rho\SNH_+=E_+(\rho)$ and there is ${\bf{D}}>0$ such $d\varphi_t:E_+(\rho)\to E_+(\varphi_t(\rho))$ is invertible with inverse satisfying
$$
\|(d\varphi_t)^{-1}\|\leq e^{-{\bf{D}}|t|}/{\bf{D}}.
$$

Since $H_2$ is compact, and $m_+(H_2,q)<n-1$ for all $q\in \SNH_2$, there is $c>0$ such that for all $q\in \SNH_2$ there is ${\bf{u}}\in E_+(q)$ with $|{\bf{u}}|=1$ such that $|dF{\bf{u}}|\geq c|{\bf{u}}|.$

Since $\rho\mapsto E_+(\rho)$ is $\nu$-H\"older continuous for some $\nu>0$ \cite[Theorem 19.1.6]{KatokHasselblatt}, there is $C_M>0$ and ${\bf{\tilde{u}}}\in E_+(\tilde{q})$ with 
$$
d({\bf{\tilde{u}}},{\bf{u}})< C_M d(q,\tilde{q})^\nu ,\qquad |{\bf{\tilde{u}}}|=1.
$$
Therefore, 
$$
|dF{\bf{\tilde{u}}}|\geq (c-C d(q,\tilde{q})^\nu)|{\bf{\tilde{u}}}|.
$$

Let $\tilde{q}=\varphi_t(\rho)$, so that $d(q,\tilde{q})<e^{-\alpha t}/\alpha$ and set
${\bf w}=(d\varphi_t)^{-1}({\bf{\tilde{u}}})$. The claim follows provided $\alpha>1$ is large enough (depending on ${\bf{D}},\nu, c,C$).
 \end{proof}
 
 \begin{lemma}
 Suppose $(M,g)$ has Anosov geodesic flow and non-positive curvature. Then if $H\subset M$ is a totally geodesic submanifold, $m_\pm(H,\rho)\equiv 0$.
 \end{lemma}
 \begin{proof}
 We need only show that for a totally geodesics submanifold $m_+(H,\rho)=m_-(H,\rho)=0$. It is easier to work on the tangent space side, so we will do so, denoting $E^\sharp_{\pm}(\rho^\sharp)$ for the dual stable and unstable bundles. 
 
Suppose $\rho^{\sharp} \in SNH$.
Then, arguing as in~\cite[Proof of Theorem 4.C]{CG19dyn}, and using that $H$ is totally geodesic, we have for all $v\in T_{\rho^\sharp}SNH$
$$
-\langle \tilde{\nabla}_{d\pi v}N,d\pi v\rangle =\langle \rho^\sharp, \Pi_H(d\pi v,d\pi v)\rangle =0.
$$
Here $N:(-\e,\e)\to NH$ is a smooth vectorfield with $N(0)=\rho^\sharp$ and $N'(0)=v$, $\tilde{\nabla}$ is the Levi-Civita connection on $M$, and $\Pi_H$ is the second fundamental form to $H$. 
On the other hand, by~\cite[(5.46)]{CG19dyn}, for $v_\pm\in E^\sharp_\pm(\rho^\sharp)$, 
$$
|-\langle \tilde{\nabla}_{d\pi v_\pm}N, d\pi v_\pm\rangle| =|\langle \rho^\sharp, \Pi_{\mc{W}_\pm}(d\pi v,d\pi v)\rangle|>0,
$$
where $\mc{W}_{\pm}$ is a stable/unstable horosphere with normal vector $\rho^\sharp$. Therefore, $T_{\rho^\sharp}SNH\cap E^\sharp_\pm(\rho^\sharp)=\emptyset$ and in particular $m_{\pm}(H,\rho)=0$. 
 \end{proof}

\FloatBarrier
\bibliography{biblio}

\begin{thebibliography}{10}

\bibitem{Ava}
V.~G. Avakumovi\'c.
\newblock \uppercase{\"u}ber die {E}igenfunktionen auf geschlossenen
  {R}iemannschen {M}annigfaltigkeiten.
\newblock {\em Math. Z.}, 65:327--344, 1956.

\bibitem{Berard77}
P.~H. B{\'e}rard.
\newblock On the wave equation on a compact {R}iemannian manifold without
  conjugate points.
\newblock {\em Math. Z.}, 155(3):249--276, 1977.

\bibitem{Besse}
A.~L. Besse.
\newblock {\em Manifolds all of whose geodesics are closed}, volume~93 of {\em
  Ergebnisse der Mathematik und ihrer Grenzgebiete [Results in Mathematics and
  Related Areas]}.
\newblock Springer-Verlag, Berlin-New York, 1978.
\newblock With appendices by D. B. A. Epstein, J.-P. Bourguignon, L.
  B\'{e}rard-Bergery, M. Berger and J. L. Kazdan.

\bibitem{Bo:17}
Y.~Bonthonneau.
\newblock The {$\Theta$} function and the {W}eyl law on manifolds without
  conjugate points.
\newblock {\em Doc. Math.}, 22:1275--1283, 2017.

\bibitem{Br:81}
R.~W. Bruggeman.
\newblock {\em Fourier coefficients of automorphic forms}, volume 865 of {\em
  Lecture Notes in Mathematics}.
\newblock Springer-Verlag, Berlin-New York, 1981.
\newblock Mathematische Lehrb\"{u}cher und Monographien, II. Abteilung:
  Mathematische Monographien [Mathematical Textbooks and Monographs, Part II:
  Mathematical Monographs], 48.

\bibitem{BuPa:95}
K.~Burns and G.~P. Paternain.
\newblock On the growth of the number of geodesics joining two points.
\newblock In {\em International {C}onference on {D}ynamical {S}ystems
  ({M}ontevideo, 1995)}, volume 362 of {\em Pitman Res. Notes Math. Ser.},
  pages 7--20. Longman, Harlow, 1996.

\bibitem{Ca20}
Y.~Canzani.
\newblock Monochromatic random waves for general riemannian manifolds.
\newblock In {\em Frontiers in Analysis and Probability}. Springer, 2020.

\bibitem{CG19dyn}
Y.~Canzani and J.~Galkowski.
\newblock Improvements for eigenfunction averages: an application of geodesic
  beams.
\newblock {\em \arXiv{1809.06296}, to appear in J. Differential Geom.}, 2019.

\bibitem{CG17}
Y.~Canzani and J.~Galkowski.
\newblock On the growth of eigenfunction averages: microlocalization and
  geometry.
\newblock {\em Duke Math. J.}, 168(16):2991--3055, 2019.

\bibitem{CG20Lp}
Y.~Canzani and J.~Galkowski.
\newblock Growth of high ${L}^p$ norms for eigenfunctions: an application of
  geodesic beams.
\newblock {\em \arXiv{2003.04597} to appear in Anal. PDE}, 2020.

\bibitem{CG19a}
Y.~Canzani and J.~Galkowski.
\newblock Eigenfunction concentration via geodesic beams.
\newblock {\em J. Reine Angew. Math.}, 775:197--257, 2021.

\bibitem{CaGa:22}
Y.~Canzani and J.~Galkowski.
\newblock Logarithmic improvements in the {W}eyl law and exponential bounds on
  the number of closed geodesics are predominant.
\newblock {\em \arXiv{2204.11921}}, 2022.

\bibitem{CaHa:15}
Y.~Canzani and B.~Hanin.
\newblock Scaling limit for the kernel of the spectral projector and remainder
  estimates in the pointwise {W}eyl law.
\newblock {\em Anal. PDE}, 8(7):1707--1731, 2015.

\bibitem{CaHa:18}
Y.~Canzani and B.~Hanin.
\newblock {$C^\infty$} scaling asymptotics for the spectral projector of the
  {L}aplacian.
\newblock {\em J. Geom. Anal.}, 28(1):111--122, 2018.

\bibitem{Cha:74}
J.~Chazarain.
\newblock Formule de {P}oisson pour les vari\'{e}t\'{e}s riemanniennes.
\newblock {\em Invent. Math.}, 24:65--82, 1974.

\bibitem{CdV:73}
Y.~Colin~de Verdi\`ere.
\newblock Spectre du laplacien et longueurs des g\'{e}od\'{e}siques
  p\'{e}riodiques. {II}.
\newblock {\em Compositio Math.}, 27(2):159--184, 1973.

\bibitem{DuGu:75}
J.~J. Duistermaat and V.~W. Guillemin.
\newblock The spectrum of positive elliptic operators and periodic
  bicharacteristics.
\newblock {\em Invent. Math.}, 29(1):39--79, 1975.

\bibitem{DyGu:14}
S.~Dyatlov and C.~Guillarmou.
\newblock Microlocal limits of plane waves and {E}isenstein functions.
\newblock {\em Ann. Sci. \'{E}c. Norm. Sup\'{e}r. (4)}, 47(2):371--448, 2014.

\bibitem{Eberlein73}
P.~Eberlein.
\newblock When is a geodesic flow of \uppercase{A}nosov type? {I}.
\newblock {\em Journal of Differential Geometry}, 8:437--463, 1973.

\bibitem{Ga:53}
L.~G\aa~rding.
\newblock On the asymptotic distribution of the eigenvalues and eigenfunctions
  of elliptic differential operators.
\newblock {\em Math. Scand.}, 1:237--255, 1953.

\bibitem{Gdefect}
J.~Galkowski.
\newblock Defect measures of eigenfunctions with maximal {$L^\infty$} growth.
\newblock {\em Ann. Inst. Fourier (Grenoble)}, 69(4):1757--1798, 2019.

\bibitem{Good}
A.~Good.
\newblock {\em Local analysis of {S}elberg's trace formula}, volume 1040 of
  {\em Lecture Notes in Mathematics}.
\newblock Springer-Verlag, Berlin, 1983.

\bibitem{Hej}
D.~A. Hejhal.
\newblock Sur certaines s\'eries de {D}irichlet associ\'ees aux g\'eod\'esiques
  ferm\'ees d'une surface de {R}iemann compacte.
\newblock {\em C. R. Acad. Sci. Paris S\'er. I Math.}, 294(8):273--276, 1982.

\bibitem{Ho68}
L.~H{\"o}rmander.
\newblock The spectral function of an elliptic operator.
\newblock {\em Acta Math.}, 121:193--218, 1968.

\bibitem{Horozov}
E.~Horozov.
\newblock On the isoenergetical nondegeneracy of the spherical pendulum.
\newblock {\em Phys. Lett. A}, 173(3):279--283, 1993.

\bibitem{IoWy:19}
A.~Iosevich and E.~Wyman.
\newblock Weyl law improvement for products of spheres.
\newblock {\em Anal. Math.}, 47(3):593--612, 2021.

\bibitem{Iv:80}
V.~J. Ivri\u{\i}.
\newblock The second term of the spectral asymptotics for a
  {L}aplace-{B}eltrami operator on manifolds with boundary.
\newblock {\em Funktsional. Anal. i Prilozhen.}, 14(2):25--34, 1980.

\bibitem{Iw:84}
H.~Iwaniec.
\newblock Nonholomorphic modular forms and their applications.
\newblock In {\em Modular forms ({D}urham, 1983)}, Ellis Horwood Ser. Math.
  Appl.: Statist. Oper. Res., pages 157--196. Horwood, Chichester, 1984.

\bibitem{KatokHasselblatt}
A.~Katok and B.~Hasselblatt.
\newblock {\em Introduction to the modern theory of dynamical systems},
  volume~54 of {\em Encyclopedia of Mathematics and its Applications}.
\newblock Cambridge University Press, Cambridge, 1995.
\newblock With a supplementary chapter by Katok and Leonardo Mendoza.

\bibitem{Ke:19}
B.~Keeler.
\newblock A logarithmic improvement in the two-point weyl law for manifolds
  without conjugate points.
\newblock {\em \arXiv{1905.05136}}, 2019.

\bibitem{Kling}
W.~Klingenberg.
\newblock Riemannian manifolds with geodesic flow of {A}nosov type.
\newblock {\em Ann. of Math. (2)}, 99:1--13, 1974.

\bibitem{KTZ}
H.~Koch, D.~Tataru, and M.~Zworski.
\newblock Semiclassical {$L^p$} estimates.
\newblock {\em Ann. Henri Poincar{\'e}}, 8(5):885--916, 2007.

\bibitem{Ku:80}
N.~V. Kuznecov.
\newblock The {P}etersson conjecture for cusp forms of weight zero and the
  {L}innik conjecture. {S}ums of {K}loosterman sums.
\newblock {\em Mat. Sb. (N.S.)}, 111(153)(3):334--383, 479, 1980.

\bibitem{Lev}
B.~M. Levitan.
\newblock On the asymptotic behavior of the spectral function of a self-adjoint
  differential equation of the second order.
\newblock {\em Izvestiya Akad. Nauk SSSR. Ser. Mat.}, 16:325--352, 1952.

\bibitem{Pl:49}
S.~Minakshisundaram and A.~. Pleijel.
\newblock Some properties of the eigenfunctions of the {L}aplace-operator on
  {R}iemannian manifolds.
\newblock {\em Canad. J. Math.}, 1:242--256, 1949.

\bibitem{SaVa:97}
Y.~Safarov and D.~Vassiliev.
\newblock {\em The asymptotic distribution of eigenvalues of partial
  differential operators}, volume 155 of {\em Translations of Mathematical
  Monographs}.
\newblock American Mathematical Society, Providence, RI, 1997.
\newblock Translated from the Russian manuscript by the authors.

\bibitem{Saf88}
Y.~G. Safarov.
\newblock Asymptotic of the spectral function of a positive elliptic operator
  without the nontrap condition.
\newblock {\em Functional Analysis and Its Applications}, 22(3):213--223, 1988.

\bibitem{Se:67}
R.~T. Seeley.
\newblock Complex powers of an elliptic operator.
\newblock pages 288--307, 1967.

\bibitem{SoggeBook}
C.~D. Sogge.
\newblock {\em Fourier integrals in classical analysis}, volume 105 of {\em
  Cambridge Tracts in Mathematics}.
\newblock Cambridge University Press, Cambridge, 1993.

\bibitem{SoggeZelditch}
C.~D. Sogge and S.~Zelditch.
\newblock Riemannian manifolds with maximal eigenfunction growth.
\newblock {\em Duke Math. J.}, 114(3):387--437, 2002.

\bibitem{Vo:90}
A.~V. Volovoy.
\newblock Improved two-term asymptotics for the eigenvalue distribution
  function of an elliptic operator on a compact manifold.
\newblock {\em Comm. Partial Differential Equations}, 15(11):1509--1563, 1990.

\bibitem{Vo:90b}
A.~V. Volovoy.
\newblock Verification of the {H}amilton flow conditions associated with
  {W}eyl's conjecture.
\newblock {\em Ann. Global Anal. Geom.}, 8(2):127--136, 1990.

\bibitem{We:75}
A.~Weinstein.
\newblock Fourier integral operators, quantization, and the spectra of
  {R}iemannian manifolds.
\newblock In {\em G\'{e}om\'{e}trie symplectique et physique math\'{e}matique
  ({C}olloq. {I}nternat. {CNRS}, {N}o. 237, {A}ix-en-{P}rovence, 1974)}, pages
  289--298. 1975.
\newblock With questions by W. Klingenberg and K. Bleuler and replies by the
  author.

\bibitem{We:12}
H.~Weyl.
\newblock Das asymptotische {V}erteilungsgesetz der {E}igenwerte linearer
  partieller {D}ifferentialgleichungen (mit einer {A}nwendung auf die {T}heorie
  der {H}ohlraumstrahlung).
\newblock {\em Math. Ann.}, 71(4):441--479, 1912.

\bibitem{Zel}
S.~Zelditch.
\newblock Kuznecov sum formulae and {S}zeg{\H{o}} limit formulae on manifolds.
\newblock {\em Comm. Partial Differential Equations}, 17(1-2):221--260, 1992.

\bibitem{EZB}
M.~Zworski.
\newblock {\em Semiclassical analysis}, volume 138 of {\em Graduate Studies in
  Mathematics}.
\newblock American Mathematical Society, Providence, RI, 2012.

\end{thebibliography}
\bibliographystyle{abbrv}

\end{document}